\documentclass[11pt,a4paper]{article}
\usepackage{authblk}
\usepackage[french,english]{babel}
\usepackage[utf8]{inputenc}
\usepackage{amsmath}
\usepackage{amsfonts}
\usepackage{amssymb}
\usepackage{amsthm}
\usepackage{hyperref}
\hypersetup{
    colorlinks=true,
    urlcolor=blue,
    linkcolor=blue,
    breaklinks=true,
    citecolor=blue,
}
\usepackage{url}
\usepackage{color}
\usepackage[dvipsnames]{xcolor}
\colorlet{darkgreen}{green!50!black}
\definecolor{blueperso}{RGB}{30,144,255} 
\definecolor{orangeperso}{RGB}{244,164,96} 
\definecolor{redperso}{RGB}{255,128,128} 
\usepackage{amsbsy}
\usepackage{amsmath}
\usepackage{dsfont}

\usepackage{geometry}
\usepackage{graphicx}
\usepackage{tikz}
\usepackage[numbers,sort&compress]{natbib}
\usepackage[margin=1cm]{caption}

\renewcommand{\geq}{\geqslant}
\renewcommand{\leq}{\leqslant}

\usepackage{amsthm}
\usepackage{oubraces}
\usepackage[shortlabels]{enumitem}
\usepackage{bm}
\usepackage{multirow}
\usepackage{nicefrac}
\usepackage{tabularx}

\newtheorem{theorem}{Theorem}
\newtheorem*{theorem*}{Theorem}

\newtheorem{lemma}[theorem]{Lemma}
\newtheorem{proposition}[theorem]{Proposition}
\newtheorem{corollary}[theorem]{Corollary}

\theoremstyle{definition}
\newtheorem{remark}[theorem]{Remark}
\newtheorem*{remark*}{Remark}
\newtheorem{example}[theorem]{Example}
\newtheorem{definition}[theorem]{Definition}

\usepackage[font={small,it}]{caption}

\newcommand{\x}{\boldsymbol{\mathrm{x}}}
\newcommand{\y}{\boldsymbol{\mathrm{y}}}
\newcommand{\w}{\boldsymbol{\mathrm{w}}_2}
\newcommand{\wb}{\boldsymbol{\mathrm{w}}_1}

\author[$1$]{Thomas Dreyfus}
\author[$2$]{Jules Flin}
\author[$2$]{Sandro Franceschi}

\affil[$1$]{Université Bourgogne Europe, CNRS, IMB UMR 5584, F-21000 Dijon, France}
\affil[$2$]{T\'el\'ecom SudParis, Institut Polytechnique de Paris}

\title{Degenerate systems of three Brownian particles with\\ asymmetric collisions: invariant measure of gaps}
\date{}

\begin{document}

\maketitle
\thispagestyle{empty}

\abstract{We consider a degenerate system of three Brownian particles undergoing asymmetric collisions.
We study the gap process of this system and focus on its invariant measure. The gap process is described as an obliquely reflected degenerate Brownian motion in a quadrant.
%For each of these cases, 
For all possible parameter cases, we compute the Laplace transform of the invariant measure, and fully characterise the conditions under which it belongs to the following classes: rational, algebraic, differentially finite, or differentially algebraic. We also derive explicit formulas for the invariant measure on the boundary of the quadrant, expressed in terms of a Theta-like function, to which we apply a polynomial differential operator.

In this study, we introduce a new parameter called $\gamma$ (along with two additional parameters $\gamma_1$ and $\gamma_2$) which governs many properties of the degenerate process. This parameter is reminiscent of the famous parameter $\alpha$ introduced by Varadhan and Williams \cite{Varadhan1984} (and the two parameters $\alpha_1$ and $\alpha_2$ recently introduced by Bousquet-Mélou et al. \cite{BoMe-El-Fr-Ha-Ra}) to study nondegenerate reflected Brownian motion in a wedge.

To establish our main results we start from a kernel functional equation characterizing the Laplace transform of the invariant measure. By an analytic approach, we establish a finite difference equation satisfied by the Laplace transform. Then, using certain so-called decoupling functions, we apply Tutte's invariant approach to solve the equation via conformal gluing functions. Finally, difference Galois theory and exhaustive study allows us to find necessary and sufficient conditions for the Laplace transform to belong to the specified function hierarchy.}

\tableofcontents

\section{Introduction and main results}\label{sec:intro}

\subsection{Degenerate three-particle systems with asymmetric collisions}\label{subsec:3part}

%%%%%Elastic collisions

Systems of Brownian particles on the real line, interacting through their ranks, have been widely studied in stochastic portfolio theory to model large equity markets, see the books by Fernholz and Karatzas \cite{Fernholz2002,Fernholz2009}. The ordered processes of such systems are typically represented by independent Brownian motions that undergo collisions and reflect symmetrically off each other.
This article focuses on particles that interact asymmetrically by splitting the local times of collisions unevenly between them.
Such processes were introduced by Warren~\cite{Warren2007} and further explored by Karatzas, Pal, and Shkolnikov~\cite{Karatzas_Pal_Shk}.

More specifically, we examine a three-particle system with degeneracy recently studied by Ichiba and Karatzas~\cite{ichiba_karatzas_degenerate_22} and consider its generalization to the case of asymmetric collisions. In this system the leader and laggard particles follow ballistic trajectories (\textit{i.e.}, with zero diffusion coefficients), while the middle particle exhibits diffusive behavior (\textit{i.e.}, with a positive diffusion coefficient).

We begin by introducing rank-based diffusions, describing both symmetric and asymmetric collisions and their connection to local times. We then introduce the gap process, which measures the distances between neighboring particles. This process can be modeled as a reflected Brownian motion in an orthant. Finally, we focus on the degenerate case which is the main purpose of this article.

\paragraph{Rank based diffusion}

For
${\mathbf{X}(t)=\left(X_1(t), \ldots, X_n(t)\right)}$ an $n$-dimensional process we denote by $\mathbf{R}(t)=(R_1(t),\dots , R_n(t))$ the associated ranked process.
This means that for $i\in\{ 1,\dots , n\}$, ${X_i(t)}$ is the {$i$}-th coordinate of the process, {$i$ is referred to as the \emph{name} of the particle},
and for $k\in\{1,\dots , n\}$, ${R_k(t)}$ is the {$k$}-th largest coordinate, {$k$ is said to be the \emph{rank} of the particle}. We then have
$$\max_{i=1,\dots,n}X_i(t)={R_1(t)} \geq \ldots \geq {R_n(t)}=\min_{i=1,\dots,n}X_i(t)
.$$ 
Let us consider the diffusion described by
$$
{\mathrm{d} X_i(t)}=\sum_{{k}=1}^n {b_k} \cdot \mathds1_{\left\{{X_i(t)}={R_k(t)}\right\}} \mathrm{d} t+\sum_{{k}=1}^n {a_k} \cdot \mathds1_{\left\{{X_i(t)}={R_k(t)}\right\}} {\mathrm{d} B_i(t)}
$$
where the drift
${b_k} \in \mathbb{R}$ depends on the {rank}, the diffusion coefficient
${a_k}\in\mathbb{R}_{\geqslant 0}$ also depends on the {rank} and
${\mathbf{B}(\cdot)=\left(B_1(\cdot), \ldots, B_n(\cdot)\right)}$ is an $n$-dimensional Brownian motion.
We consider solutions with the nonstickiness conditions, that is
$$
\int_0^{\infty} 1_{\left\{{R_k(t)}={R_{k+1}(t)}\right\}} \mathrm{d} t=0, \quad k=1,2, \ldots, n-1 .
$$
In the nondegenerate case, that is, when $a_k>0$ for all $k$, many studies have investigated the fundamental properties of the existence of this process. The existence of weak solutions can be established via the martingale problem theory of Stroock and Varadhan. Bass and Pardoux~\cite{Bass1987} studied the uniqueness in distribution. Prokaj~\cite{Prokaj2013}, along with Ichiba, Karatzas and Shkolnikov~\cite{Ichiba2013}, and Fernholz et al.~\cite{fernholz2013planar}, demonstrated that pathwise uniqueness holds up to the time of a triple collision. The positive recurrence property was investigated by Pal and Pitman~\cite{PalPitman2008} and later by Dembo and Tsai~\cite{DemboTsai2017}. Pathwise differentiability was analyzed by Lipshutz and Ramanan~\cite{Lipshutz19}. Finally, Ichiba et al.~\cite{Ichiba2013}, and 
Sarantsev et al.~\cite{Sarantsev2015,IchibaSarantsev2017} showed that a strong solution exists without triple collisions when a concavity condition on the diffusion coefficients holds.

A rich body of literature exists on these rank-based diffusions, addressing many of their key properties. As $n\to\infty$ the limiting particle density is determined by the unique solution of a McKean-Vlasov equation, with the cumulative distribution function evolving according to a porous medium equation, see the work of Dembo, Shkolnikov, Varadhan and Zeitouni~\cite{Dembo2016}. The convergence to this limit is exponentially fast, with fluctuations around it described by a Gaussian process governed by a stochastic partial differential equation, see the study by Kolli and Shkolnikov~\cite{Kolli2018}. This framework also has significant applications in financial markets, as demonstrated by the work of Banner, Fernholz and Karatzas~\cite{Banner2005}, Ichiba et al.~\cite{Ichiba2011}, Jourdain and Reygner~\cite{Jourdain2015}, or Banerjee and Budhiraja~\cite{BanerjeeBudhiraja} about Atlas models.

\paragraph{Ranked process and local time}
For $n$-semimartingales ${\mathbf{X}(t)=\left(X_1(t), \ldots, X_n(t)\right)}$,  Banner and Ghomrasni~\cite{Banner2008}
derive a semimartingale decomposition of
the corresponding ranked process $\mathbf{R}(t)=(R_1(t),\dots , R_n(t))$.
Assuming that there is no triple collision, that is, 
$$\mathbb{P} (\exists t\geqslant 0 , 1\leqslant i <j<k\leqslant n : R_i=R_j=R_k)=0,$$ their result (see Corollary 2.6. in \cite{Banner2008}) states that
\begin{equation}
{\mathrm{d} R_k(t)}={b_k} \mathrm{d} t + {a_k} 
\underbrace{\sum_{{i}=1}^n \mathds1_{\left\{{X_i(t)}={R_k(t)}\right\}} {\mathrm{d} B_i(t)}}_{\displaystyle=:{\mathrm{d}W_k(t)}} - \frac{1}{2} {\mathrm{d}\Lambda^{(k-1,k)}(t)} + \frac{1}{2} {\mathrm{d}\Lambda^{(k,k+1)}(t)}
\label{eq:banner_ghom}
\end{equation}
where $W_k(\cdot)$ is a Brownian motion which depends on the {rank $k$}
and ${\Lambda^{k,k+1}(t)}=L^{{G}_k}(t)$ is the local time accumulated 
at the origin before time $t$ by the semimartingale ${G}_k(t):={R_{k}(t) - R_{k+1}(t)}$. Note that this formula can be extended to the case of multiple collisions.
Heuristically, this decomposition, and more particularly the coeficients $1/2$ in front of the local times, highlight the \emph{symmetric collisions} that occur between the interacting particles of the ranked process. 

\paragraph{Asymmetric collisions} 
Warren \cite{Warren2007} and then Karatzas, Pal and Shkolnikov \cite{Karatzas_Pal_Shk} recently introduced a system of Brownian particles which collide asymmetrically.
The dynamics of this process is given by ${\mathbf{R}(t)=\left(R_1(t), \ldots, R_n(t)\right)}$ such that
\begin{equation}
{\mathrm{d} R_k(t)}={b_k} \mathrm{d} t + {a_k} 
\mathrm{d}{{W_k(t)}} - {q_k^+} \mathrm{d}\Lambda^{(k-1,k)}(t) + {q_k^-} \mathrm{d}\Lambda^{(k,k+1)}(t)
\label{eq:ranledoblique}
\end{equation}
where the parameters $q_k\in(0,1)$ satisfy the relation
$$
{q_k^-+q_{k+1}^+=1}.
$$
Whereas in the formula~\eqref{eq:banner_ghom} the local times were split equally between two colliding particles, this time the local times are divided unequally with coefficients $q_k\in(0,1)$ depending on the rank $k$. 
This process takes his values in the Weyl
chamber $\mathbb{W}^n=\{(r_1,\dots,r_n): r_1 \geqslant r_2\geqslant \dots \geqslant r_n \}$ and can be seen as the ranked process of the process $\mathbf{X}$ defined by
\begin{align*}
{\mathrm{d} X_i(t)}
&=\sum_{{k}=1}^n {b_k} \cdot \mathds1_{\left\{{X_i(t)}={R_k(t)}\right\}} \mathrm{d} t+\sum_{{k}=1}^n {a_k} \cdot \mathds1_{\left\{{X_i(t)}={R_k(t)}\right\}} {\mathrm{d} B_i(t)}
\\ &+
\sum_{{k}=1}^n (q_k^--\tfrac{1}{2}) \cdot \mathds1_{\left\{{X_i(t)}={R_k(t)}\right\}} \mathrm{d} \Lambda^{(k,k+1)}(t)
-\sum_{{k}=1}^n (q_k^+-\tfrac{1}{2}) \cdot \mathds1_{\left\{{X_i(t)}={R_k(t)}\right\}} \mathrm{d} \Lambda^{(k,k-1)}(t) .
\end{align*}
This process can model the logarithmic capitalizations in large equity markets. Karatzas et al. explain it is a generalization of the first-order models of stochastic portfolio theory. Of course, taking $q_k^\pm=\frac{1}{2}$ for all $k$ leads to symmetric collisons. If we allow to take $q_k^-=0$ and $q_k^+=1$ for all $k$, the process $\mathbf{R}(t)$ is the continuous version of the Totally Asymmetric Simple Exclusion Process (TASEP) \cite{Gorin2015,Warren2007,Karatzas_Pal_Shk,OCOr-14}.

\paragraph{Gap process and reflected Brownian motion}

The gap process $\mathbf{G}(t)=({G}_1(t),\dots , {G}_{n-1}(t))$ is defined as the distance between two consecutive particles, \textit{i.e.}, for all $k=1,\dots, n-1$ we have
\begin{equation}\label{eq:defgap}
{G}_k(t) =R_{k}(t)-R_{k+1}(t) .
\end{equation}
If we denote $-\mu_k=b_k-b_{k+1}$ and $\widetilde{W}_k(t)=a_kW_k(t)-a_{k+1}W_{k+1}(t)$ we obtain
$$
{G}_k(t)= {G}_k(0)-\mu_k t + \widetilde{W}_k(t)  - {q_k^+} L^{{G}_{k-1}}(t) +  L^{{G}_k}(t) - q_{k+1}^- L^{{G}_{k+1}}(t).
$$
remembering that ${\Lambda^{k,k+1}(t)}=L^{{G}_k}(t)$ and where by convention we consider that $L^{{G}_{-1}}=L^{{G}_{n}}=0$.
The Gap process is then a reflected Brownian motion (RBM) in the nonnegative orthant $\mathbb{R}_+^{n-1}$ with drift vector
$-\bm{\mu}=(b_1-b_2,\dots , b_{n-1}-b_n)$, covariance matrix 
$$\mathcal{A}=\left(\begin{array}{cccc}a_1^2+a_2^2 & -a_2^2 & 0 & 0 \\ -a_2^2 & a_2^2+a_3^2 & -a_3^2 & 0 \\ 0 & \ddots & \ddots & \ddots \\ 0 & 0 & -a_{n-1}^2 & a_{n-1}^2+a_n^2\end{array}\right)$$
and reflection matrix
$$\mathcal{R}=\mathbf{I}_{n-1}-\mathcal{Q}, \quad\text{ where }\quad \mathcal{Q}:=\left(\begin{array}{cccc}0 & q_2^{-} & 0 & 0 \\ q_2^{+} & 0 & q_3^{-} & 0 \\ 0 & \ddots & \ddots & \ddots \\ 0 & 0 & q_{n-1}^{+} & 0\end{array}\right) .$$
If we denote $\mathbf{L}(t)=(L^{{G}_1}(t),\dots , L^{{G}_{n-1}}(t))$ the local time vector of $\mathbf{G}$ and $\widetilde{\mathbf{W}}(t)=(\widetilde{W}_1(t),\dots ,\widetilde{W}_{n-1}(t))$ the Brownian motion of covariance matrix $\mathcal{A}$, 
we then have
$$
\mathbf{G}(t)=\mathbf{G}(0)- \bm{\mu} t + \widetilde{\mathbf{W}}(t) + \mathbf{L}(t) \mathcal{R}^\top 
$$
which is the standard way of writing a semimartingale reflecting Brownian motion (SRBM).
The reader may consult the works of Harrison and Reiman \cite{Harrison1981} and Williams \cite{Williams1995} for an overview of SRBM in orthants.

\paragraph{Degenerate case}
We now consider degenerate cases by allowing
some of $a_k$ to be zero in~\eqref{eq:ranledoblique}.
The degenerate case has been recently studied.
The case of two particles $n=2$ was studied by Fernholz, Ichiba, Karatzas and Prokaj~\cite{fernholz2013planar} and later by Ichiba, Karatzas, Prokaj and Yan~\cite{Ichiba2018}. The case of three particles $n=3$ was recently investigated in depth by Ichiba and Karatzas~\cite{ichiba_karatzas_degenerate_22} in two configurations. 
\begin{enumerate}[label=(\roman*)]
\item\label{casi} When $a_1>0$, $a_3>0$ and $a_2 = 0$, the leader and laggard particles exhibit diffusive behavior, while the middle follow ballistic trajectorie. In this case, they demonstrated the existence of a weak solution that is unique in distribution.
\item\label{casii} In contrast, when $a_1 = a_3 = 0$ and $a_2 > 0$, the leader and laggard particles follow ballistic trajectories, while the middle particle exhibits diffusive behavior. In this case, they established the existence of a pathwise unique, strong solution with no triple collisions. See Figure~\ref{fig:simul} which draws a path for such a process. 
\end{enumerate}
In the latter case, Franceschi, Ichiba, Karatzas and Raschel~\cite{FIKR23+} examined in details the invariant measure of the gaps of this degenerate systems with \emph{symmetric collisions}.
The focus of this article is on degenerate three-particle systems in the case~\ref{casii}, generalised to \emph{asymmetric collisions}. The study of the invariant measure of the gaps, considering asymmetric collisions, reveals many different and original behaviors that differ from the symmetric case. The following paragraph sets out the notations used in the rest of the article and establishes the conditions necessary for the study of this process.

\paragraph{Degenerate reflected Brownian motion in the quadrant}

We set the covariance matrix, the reflection matrix and the drift
\begin{equation}\label{eq:defparam}
\mathcal{A}=
\begin{pmatrix}
\sigma_1 & -\sqrt{\sigma_1\sigma_2}
\\
-\sqrt{\sigma_1\sigma_2} & \sigma_{2}
\end{pmatrix}
,\quad
\mathcal{R}=
\begin{pmatrix}
1 & r_2
\\
r_1 & 1
\end{pmatrix}
\quad\text{and}\quad
-\bm{\mu}=(-\mu_1,-\mu_2)
\end{equation}
and we consider the associated degenerate reflected Brownian motion $\mathbf{G}(t)=({G}_1(t),{G}_2(t))$ which is a continuous semimartingale defined by the Skorokhod-type decomposition
\begin{equation}
\begin{cases}
{G}_1(t) = {G}_1(0)+\mu_1 t +  \widetilde{W}_1(t)   +  L^{{G}_1}(t) +r_2 L^{{G}_{2}}(t),
\\
{G}_2(t) = {G}_2(0)+\mu_2 t +   \widetilde{W}_2(t)  +r_1 L^{{G}_{1}}(t) +  L^{{G}_2}(t) .
\end{cases}
\label{def:G}
\end{equation}
If we take $\sigma_1=\sigma_2=a_2^2$, $r_1=- {q_2^+}$, $r_2=- q_{2}^-$, $\mu_1=b_2-b_1$ and $\mu_2=b_3-b_1$ we find the case~\ref{casii} (where $a_1=a_3=0$ and $a_2>0$) and $G$ is then the gap process of a degenerate three-particle system with asymmetric collisions. The set of parameters considered in this paper to study $\mathbf{G}$ is slightly more general than in the case~\ref{casii}, we allow $\sigma_1$ to be different from $\sigma_2$ and we do not impose a sign condition on $r_1$ and $r_2$.
The covariance matrix $\mathcal{A}$ is symmetric and nonnegative-definite (\textit{i.e.}, positive semi-definite) and $\sigma_1>0$ and $\sigma_2>0$. The degeneracy comes from the fact that $$\det \mathcal{A}=0.$$
To ensure that the process $\mathbf{G}(t)$ is well defined and recurrent, we need to work under the following assumptions.
First, there is the existence condition of the process 
\begin{equation}\tag{$\textsc{h}_1$}
1-r_1r_2>0 \text{ or } \{r_1>0\text{ and }r_2>0\}
\label{eq:H1}
\end{equation}
which follows from the classical conditions for the existence of an SRBM \cite{Reiman1988,TaWi-93} and from \cite{Harrison1981} which allows the covariance matrix to be degenerate (\textit{i.e.}, not necessarily positive-definite). The work of Ichiba and Karatzas~\cite{ichiba_karatzas_degenerate_22} establishes the existence of a pathwise unique strong solution of the associated particle system $\mathbf{X}$ for the values $r_1=r_2=-1/2$. Appendix~A in~\cite{ichiba_karatzas_degenerate_22} establishes the existence of a weak solution for $r_1=-1/2$ and $r_2=-3/2$ and this result can be generalised to all oblique reflections satisfying~\eqref{eq:H1}. 
Second, there is the recurrence condition of the process
\begin{equation}\tag{$\textsc{h}_2$}
\mu_1-r_2\mu_2>0 \text{ and } \mu_2-r_1\mu_1>0 .
\label{eq:H2}
\end{equation}
In this case, the process is positive recurrent, has a unique invariant
measure $\bm\pi$. Historically, this recurrence condition comes from Hobson and Rogers~\cite{Hobson1993} and has recently been extended to the degenerate case in~\cite[\S 2.3.]{ichiba_karatzas_degenerate_22}.
Finally, we restrict ourselves to the case of negative drift $-\bm\mu$, \textit{i.e.},
\begin{equation}\tag{$\textsc{h}_3$}
\mu_1>0 \text{ and } \mu_2>0,
\label{eq:H3}
\end{equation}
which is equivalent to $b_1<b_2<b_3$.
The same study could be generalised to all types of drift, but in order to limit the number of cases to handle we make this common assumption, which is also present in~\cite{ichiba_karatzas_degenerate_22,BoMe-El-Fr-Ha-Ra} and many other papers.

The following section sets out the equations characterising the invariant measure and presents the main results of the article.

\begin{figure}[t]
\begin{center}
\includegraphics[width=0.9\textwidth]{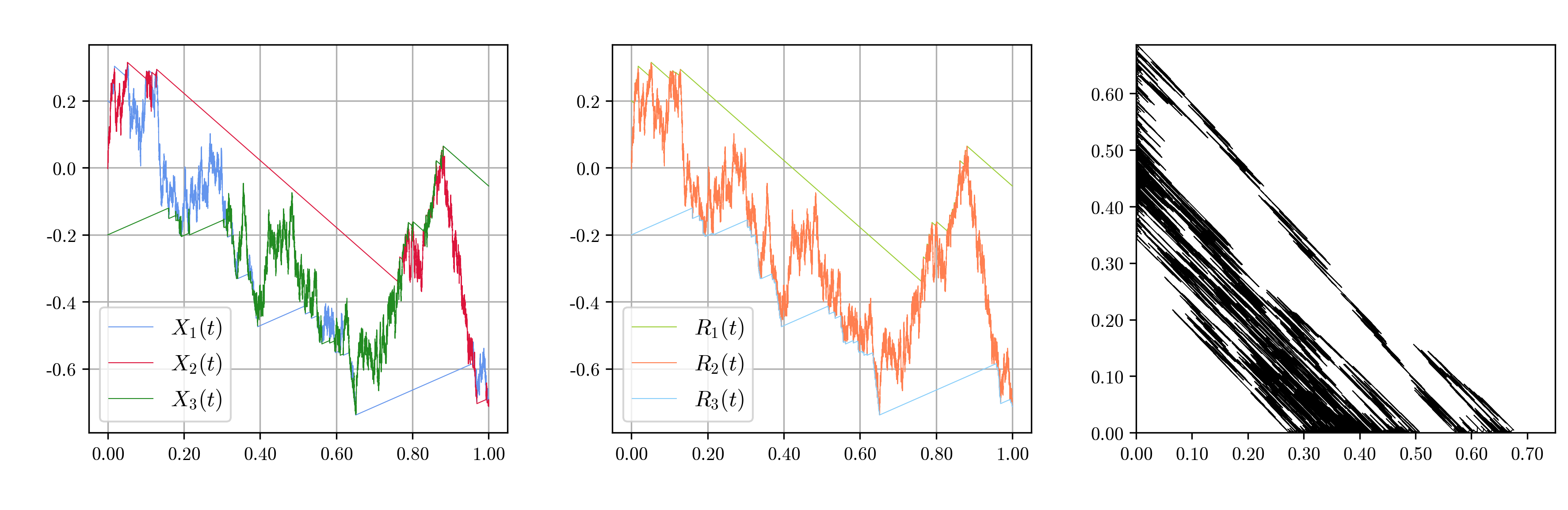}
\end{center}
\caption{paths of the interacting Brownian particles, colored by \textbf{name} (left), colored by \textbf{rank} (middle) and the gap process $({G}_1,{G}_2)$ (right).}
\label{fig:simul}
\end{figure}

\subsection{Kernel functional equation}
\label{subsec:FE}

We consider the gap process $\mathbf{G}=(G_1,G_2)$ defined in~\eqref{def:G} and we work under hypotheses~\eqref{eq:H1}, \eqref{eq:H2} and~\eqref{eq:H3}, \textit{i.e.}, the existence and recurrence conditions and the negativity of drift. Let $\bm\pi$ be the invariant probability measure of the gap process $\mathbf{G}$ and let us define the lateral invariant measures
\begin{equation}\label{eq:deflateral}
{\bm \nu_1}(A)=\mathbb{E}_{\bm\pi}\int_0^2 \mathds1_A({G}_2(t))\mathrm{d}L^{{G}_1}(t)
\quad\text{and}\quad
{\bm \nu_2}(A)=\mathbb{E}_{\bm\pi}\int_0^2 \mathds1_A({G}_1(t))\mathrm{d}L^{{G}_2}(t)
\end{equation}
for all $A\in\mathcal{B}\left((0,\infty)\right)$.
For all $x,y\in\mathbb C$ such that $\mathfrak{Re}(x)\leqslant 0$ and $\mathfrak{Re}(y)\leqslant 0$, we define the Laplace transforms
\begin{align}
\phi(x,y)&=\mathbb{E}_{\bm\pi} e^{x{G}_1(t)+y{G}_2(t)}=\int\!\!\!\!\!\int_{(0,\infty)^2} e^{x g_1+y g_2} \bm\pi(\mathrm{d}g_1,\mathrm{d}g_2),\label{eq:defphi}\\
\phi_1(y)&= \mathbb{E}_{\bm\pi} \int_0^2 e^{y{G}_2(t)} \mathrm{d}L^{{G}_1}(t)=\int_{(0,\infty)} e^{y g_2} \boldsymbol{\nu}_1(\mathrm{d}g_2),\label{eq:defphi1}\\
\phi_2(x)&=
\mathbb{E}_{\bm\pi} \int_0^2 e^{x{G}_1(t)} \mathrm{d}L^{{G}_2}(t)=\int_{(0,\infty)} e^{x g_1} \boldsymbol{\nu}_2(\mathrm{d}g_1).\label{eq:defphi2}
\end{align}

\begin{remark}
The notations $\phi_1(y)$ and $\phi_2(x)$ may seem unfortunate, but they are standard in the literature and arise naturally from a generalization in higher dimensions: the \emph{first} Laplace transform $\phi_1$ is associated with the lateral measure on the hyperplane where the \emph{first} coordinate vanishes.
\end{remark}
The invariant probability measure $\boldsymbol{\pi}$ is characterized by the well-known Basic Adjoint Relationship (BAR) \cite{Harrison1987b,KuSt-01}. In the now classic way~\cite{FIKR23+,Dai2011}, the BAR can be rewritten as a functional equation, which in this case is the following
\begin{proposition}[Functional equation] For $\mathfrak{Re}(x)\leqslant 0$ and $\mathfrak{Re}(y)\leqslant 0$ we have
\begin{equation}\label{eq:FE}
-K(x,y)\phi(x,y)=k_1(x,y)\phi_1(y)+k_2(x,y)\phi_2(x)
\end{equation}
where the kernel is defined by
\begin{equation}\label{eq:noyau}
K(x,y)= ({\sigma_1}x-{\sigma_2}y)^2 - 2\mu_1 x - 2\mu_2 y 
\end{equation}
and
\begin{equation}\label{eq:k12}
k_1(x,y)=x+r_1 y ,
\quad
k_2(x,y)=y+ r_2 x.
\end{equation}
\end{proposition}
In this paper, we fully solve the functional equation under the assumptions~\eqref{eq:H1}, \eqref{eq:H2}, \eqref{eq:H3} by determining the unique Laplace transform $\phi$ of a probability measure satisfying~\eqref{eq:FE}. We now give a few preliminar results which we can simply deduce from this functional equation.

We note the densities with respect to suitable Lebesgue measures $\lambda^{(d)}$ on $\mathbb R_+^d$:
\begin{equation}
\pi:=\frac{\mathrm{d}\boldsymbol{\pi}}{\mathrm{d}\lambda^{(2)}},\quad\quad \nu_1:=\frac{\mathrm{d}\boldsymbol{\nu}_1}{\mathrm{d}\lambda^{(1)}},\quad\quad \nu_2:=\frac{\mathrm{d}\boldsymbol{\nu}_2}{\mathrm{d}\lambda^{(1)}}.
\label{eq:densitedef}
\end{equation}
\begin{lemma}[Relation between $\pi$, $\nu_1$ and $\nu_2$] The densities $\pi$, $\nu_1$ and $\nu_2$ satisfy the following relations
\begin{equation}
    \nu_1(v)=\sigma_1^2 \pi(0,v),\quad\quad \nu_2(u)=\sigma_2^2 \pi(u,0)
\end{equation}
\end{lemma}
\begin{proof}
Since $\pi$ is a probability density, it is nonnegative and integrable, so the dominated convergence theorem justifies exchanging the limit with the integral:
    $$\lim_{x\to -\infty} x\phi(x,y)=\lim_{x\to-\infty} x\int\!\!\!\!\!\int_{(0,\infty)^2}e^{xu+yv}\pi(u,v)\mathrm{d}u\mathrm{d}v=\int_0^{+\infty}\!\!e^{yv}\left(\lim_{x\to-\infty}\int_{0}^{+\infty}e^{xu}\pi(u,v)\mathrm{d}u\right)\mathrm{d}v.$$
    The limit inside the parentheses is, according to the initial value theorem, equal to $-\pi(0,v)$. Dividing the functional equation by $x$ and taking the limit as $x$ approaches $-\infty$ yields the first relation. By a similar argument, we obtain the second one. 
\end{proof}
\begin{lemma}[Value of Laplace transforms at $0$] 
We have
\begin{equation}
\label{eq:normcst}
\Big(\phi_1(0),\phi_2(0)\Big)=\left(2\frac{\mu_1-r_2\mu_2}{1-r_1r_2},2\frac{\mu_2-r_1\mu_1}{1-r_1r_2}\right).
\end{equation}
\end{lemma}
\begin{proof}
Taking $x=0$ in the functional equation~\eqref{eq:FE}, we obtain
$$-y\Big(\sigma_2^2y-2\mu_2\Big)\phi(0,y)=r_1y\phi_1(y)+y\phi_2(0).$$
Recalling that $\phi(0,0)=\boldsymbol{\pi}(\mathbb R_+^2)=1$, one can now divide by $y$, and take $y=0$:
$$2\mu_2
%=2\mu_2\phi(0,0)
=r_1\phi_1(0)+\phi_2(0).$$
Symmetrically, 
$$2\mu_1=\phi_1(0)+r_2\phi_2(0).$$
The resulting system of linear equations admits a unique solution given by~\eqref{eq:normcst}, since hypotheses~\eqref{eq:H1} and~\eqref{eq:H2} ensure that the determinant $r_1 r_2 -1$ of the system is nonzero (by~\eqref{eq:H1}, if $r_1r_2-1 = 0$, then $r_2=1/r_1>0$, and multiplying~\eqref{eq:H2} by this quantity yields the contradiction $0<\mu_1-r_2\mu_2<0$).
\end{proof}

With a simple change of variables in the functional equation given in Appendix~\ref{sec:homogeneity}, one can always assume that
\begin{equation}\tag{$\textsc{h}_4$}
\sigma_1=\sigma_2=1 \text{ and } \mu_1+\mu_2=1.
\label{eq:H4}
\end{equation}
As a result, in the rest of the article, we will work under this assumption.

\subsection{Main results}\label{subsec:mainres}
\noindent\textbf{Notation.} We denote by $\mathbb{N} = \{1, 2, 3, \dots\}$ the set of positive integers, and by $\mathbb{N}_0 = \mathbb{N} \cup \{0\}$ the set of non-negative integers. This convention is fixed throughout the article.\\

The following results will all be stated under the assumptions~\eqref{eq:H1}, \eqref{eq:H2}, \eqref{eq:H3}, \eqref{eq:H4} and we will also assume~\eqref{eq:H5} where 
\begin{equation}\tag{$\textsc{h}_5$}
r_1\ne -1\quad\text{and}\quad r_2\ne -1.
\label{eq:H5}
\end{equation}
These cases where~\eqref{eq:H5} is not satisfied, \textit{i.e.}, $r_1=-1$ or $r_2=-1$, are treated separately in Appendix~\ref{sec:r1r2}.
The generalization of our results to cases where hypothesis~\eqref{eq:H4} is not satisfied is done in Appendix~\ref{sec:homogeneity}.

We have three main types of results: the explicit calculation of Laplace transforms, their differential properties and the explicit calculation of the gap process invariant densities.

\paragraph{Explicit expression of the Laplace transforms}

To state our results we need to define
$$
s_1:=\frac{r_1\mu_1-\mu_2}{1+r_1} \quad\text{and}\quad s_2:=\frac{\mu_1-r_2\mu_2}{1+r_2}
$$
and the function
\begin{equation}
\label{eq:Ddec}
D(s):=\frac{\Gamma(s-s_1)\,\Gamma(s+s_2+\mu_2)}{\Gamma(s-s_2)\,\Gamma(s+s_1+\mu_2)} 
\end{equation}
where $\Gamma$ is the gamma function.
\renewcommand{\thetheorem}{A}
\begin{theorem}[Explicit expression of $\phi_1$]
\label{thm:main}
There exist a rational function $R\in\mathbb{C}(X)$ such that the Laplace transform $\phi_1$ satisfies
$$
\phi_1(y)=D\left(\frac{1}{2}\Big(\mu_1+\sqrt{2y+\mu_1^2}\Big)\right)R\left(\tan\Big(\frac{\pi}{2} \sqrt{2y+\mu_1^2}\Big)\right).
$$
An explicit expression depending on the parameters is given in Theorems~\ref{thm:N}, \ref{thm:-N}, \ref{thm:notN} and \ref{thm:general}.
To be more precise, setting 
$s=\frac{1}{2}\Big(\mu_1+\sqrt{2y+\mu_1^2}\Big)$, one obtains from Proposition~\ref{prop:consistency} that 
\begin{equation}
\phi_1(y)\propto
D(s)\times \begin{cases}
\dfrac{1}{\sin\bigl(\pi(s+s_1+\mu_2)\bigr)\,\sin\bigl(\pi(s-s_2)\bigr)},
& \text{if $s_1<0$ and $s_2>0$,} \\[1.2em]
\dfrac{\sin\bigl(\pi(s+s_2+\mu_2)\bigr)}{\sin\bigl(\pi(s+s_1+\mu_2)\bigr)},
& \text{if $s_1<0$ and $s_2<0$,} \\[1.2em]
\dfrac{\sin\bigl(\pi(s-s_1)\bigr)}{\sin\bigl(\pi(s-s_2)\bigr)},
& \text{if $s_1>0$ and $s_2>0$.}
\end{cases}
\end{equation}
where $\propto$ is the proportionality symbol.
\end{theorem}
\addtocounter{theorem}{-1}
\renewcommand{\thetheorem}{\arabic{theorem}}
A symmetrical expression holds for $\phi_2(x)$ and with the functional equation~\eqref{eq:FE} we directly obtain a formula for the bivariate Laplace transform $\phi(x,y)$.
Even if it is not obvious because of the square root, this Laplace transform is of course analytic for $y$ of negative real part and even admits a meromorphic extension to the whole of $\mathbb{C}$, see Remark~\ref{rem:sinc} for more details.

\paragraph{Differential properties of the Laplace transforms}

We introduce the parameters
$$
\gamma:=\frac{1-r_1r_2}{(1+r_1)(1+r_2)}=\frac{1}{1+r_1}+\frac{1}{1+r_2}-1,\quad \gamma_1:=\frac{1+\mu_2-r_1\mu_1}{1+r_1},\quad \gamma_2:=\frac{1+\mu_1-r_2\mu_2}{1+r_2}.
$$
These parameters are reminiscent of the famous parameter $\alpha$ introduced by Varadhan and Williams \cite{Varadhan1984} and the two parameters $\alpha_1$ and $\alpha_2$ recently introduced by Bousquet-Mélou et al. \cite{BoMe-El-Fr-Ha-Ra} in the nondegenerate case. We refer to the Appendix~\ref{sec:C} for more details on the links with these parameters.

The function $D$, which will be referred below as the decoupling function, simplifies to a rational function given in Proposition~\ref{prop:decoupling} if and only if
$$
\gamma\in \mathbb Z\text{ or }\{\gamma_1,\gamma_2\}\subset \mathbb Z.
$$
Using these parameters, we also determine the differential nature of the Laplace transform. More precisely, we identify necessary and sufficient conditions on these parameters under which the Laplace transform is:
\begin{itemize}
    \item \emph{rational}, \textit{i.e.}, the quotient of two polynomials, 
    \item \emph{algebraic}, \textit{i.e.}, it satisfies a non-trivial polynomial equation with coefficients in $\mathbb{R}(X)$, the field of rational functions over~$\mathbb{R}$,
    \item \emph{differentially finite} or \emph{D-finite} or \emph{holonomic}, \textit{i.e}., it satisfies a non-trivial linear differential equation with coefficients in $\mathbb{R}(X)$,
    \item \emph{differentially algebraic} or \emph{D-algebraic}, \textit{i.e.}, it satisfies a non-trivial polynomial differential equation in $\mathbb{R}(X)$,
    \item \emph{differentially transcendental} or \emph{D-transcendental} if it is not \emph{differentially algebraic}. 
   
\end{itemize}
See Definition~\ref{def:hierarchy} for subtleties concerning these definitions in the context of bivariate functions. We talk about function hierarchy because of the following inclusions
\begin{equation*}
\text{rational}\subset\text{algebraic}\subset\text{D-finite}\subset\text{D-algebraic}.
\end{equation*}
\renewcommand{\thetheorem}{B}
\begin{theorem}[Algebraic and differential nature]
Table~\ref{tab:classification} gives necessary and sufficient condition for $\phi_1$, $\phi_1$ and $\phi$ to belongs to this differential hierarchy. The classification made in this table is proven in Section~\ref{sec:hierarchy}.
\end{theorem}
\addtocounter{theorem}{-1}
\renewcommand{\thetheorem}{\arabic{theorem}}
\begin{table}[!htbp]
\centering
\begin{tabular}{|c|c|c|c|c|c|}
\hline 
\begin{tabular}{@{}c@{}}\textbf{nature of} $\boldsymbol{\phi_1}$\\  $\boldsymbol{\phi_2}$ \textbf{and} $\boldsymbol{\phi}$ \end{tabular} & rational & algebraic & D-finite & D-algebraic & D-transcendental \\ 
\hline 
\textbf{condition} & \multicolumn{3}{c|}{$\gamma\in-\mathbb{N}$} & \begin{tabular}{@{}c@{}}$\gamma\in\mathbb{Z}$ or \\  $\{ \gamma_1,\gamma_2\}\subset \mathbb{Z}$\end{tabular}   & \begin{tabular}{@{}c@{}}$\gamma\notin\mathbb{Z}$ and \\  $\{ \gamma_1,\gamma_2\}\not\subset \mathbb{Z}$ \end{tabular}    \\ 
\hline 
\end{tabular} 
\caption{Differential and algebraic nature of the Laplace transforms.}
\label{tab:classification}
\end{table}

This differential structure of the Laplace transform has various implications for the invariant measure itself. For instance, if the transform is rational, the invariant measure density can be written as a linear combination of exponentials multiplied by polynomials. If it is D-algebraic, it yields a recurrence relation for the moments of the invariant measure and this recurrence relation is linear if it is D-finite. We refer the reader to the introduction of Bousquet-Mélou et al. \cite{BoMe-El-Fr-Ha-Ra}, which provides further insights into the relevance of this classification within this hierarchy.

\paragraph{Explicit expression of the densities}

Inverting the Laplace transform $\phi$ to recover the density $\pi$ of the invariant measure is, in general, a difficult task. Nonetheless, in some cases we are able to invert the univariate Laplace transform $\phi_1$ to recover the density $\nu_1$ of the lateral invariant measure $\bm\nu_1$: if $\phi_1$ is rational, \textit{i.e.} if $\gamma\in -\mathbb{N}$, then $\nu_1$ can be expressed as a \emph{sum of exponentials}; and if $\phi_1$ is D-algebraic and the key parameter $\gamma$ is positive, \textit{i.e.} if $\gamma\in \mathbb{N}$ or  $\{\gamma_1,\gamma_2 \}\subset \mathbb{N}$, then we describe a polynomial differential operator $\boldsymbol{\mathrm{P}}^*$ and a function $\theta$ (related to the well-known Jacobi theta functions) such that 
$$\nu_1(v)\propto \boldsymbol{\mathrm{P}}^*\theta(e^{-v}).$$ 
These relations lead to a series expansion of $\nu_1$%(see Theorems~\ref{thm:densitygamma} and~\ref{thm:densitygamma12})
. The following theorem summarizes these results.

\renewcommand{\thetheorem}{C}
\begin{theorem}[Explicit expression of $\nu_1$] The density of the lateral invariant measure is given by:
\begin{itemize}
\item When $\gamma\in -\mathbb{N}$, then $\nu_1$ can be expressed as a \emph{sum of exponentials}, see Theorem~\ref{thm:densityrat} for an explicit expression.
\item When $\gamma\in \mathbb N$, then 
$$\nu_1(v)\propto\sum_{n\in\mathbb Z}\left(n+\frac{\gamma_1}{2}\right)P\left(\frac{(2n+\gamma_1)^2-\mu_1^2}{2}\right)\exp\left(-v\frac{(2n+\gamma_1)^2-\mu_1^2}{2}\right),$$
where $P$ is an explicit polynomial given in Theorem~\ref{thm:densitygamma}.
\item When $\{\gamma_1,\gamma_2\}\subset \mathbb{N}$ and $\gamma\notin\mathbb{N}$, then
$$\nu_1(v)\propto \sum_{n\in\mathbb Z}(-1)^nn^2P\left(\frac{n^2-\mu_1^2}{2}\right)\exp\left(-v\frac{n^2-\mu_1^2}{2}\right),$$
where $P$ is an explicit polynomial given in Theorem~\ref{thm:densitygamma12}.
\end{itemize}
\end{theorem}
\addtocounter{theorem}{-1}
\renewcommand{\thetheorem}{\arabic{theorem}}

Recovering the bivariate density $\pi$ would require a more in-depth analysis, possibly using the compensation approach~\cite{FIKR23+}. Since this article is already quite long, we leave this investigation for future work.

\subsection{Strategy of proof and structure of the paper}\label{subsec:strat}

The paper investigates the invariant measure of the gap process associated with a degenerate system of three Brownian particles undergoing asymmetric collisions. This gap process is modeled as an obliquely reflected Brownian motion in a quadrant with a degenerate diffusion matrix (Section~\ref{subsec:3part}). The strategy relies on transforming the problem into a functional equation satisfied by the Laplace transform of the invariant measure (Section~\ref{subsec:FE}). The main objective is to provide an explicit expression for the invariant measure and to classify the nature of its Laplace transform according to its differential properties: whether it is rational, algebraic, D-finite, or D-algebraic (Section~\ref{subsec:mainres}). The aim of the rest of this article is to solve the functional equation~\eqref{eq:FE} under the assumptions~\eqref{eq:H1}, \eqref{eq:H2}, \eqref{eq:H3}, \eqref{eq:H4} and \eqref{eq:H5}. There may be several functions that satisfy this functional equation, but this article as a whole shows that only one is the Laplace transform of a probability measure. The structure of the article reflects the strategy of proof. The key steps are as follows:

\begin{itemize}
    \item Section~\ref{sec:FuncEqToDiffEq} introduces a Riemann surface $\mathcal{S}$ on which the kernel $K$ introduced in~\eqref{eq:noyau} vanishes, along with an explicit elementary parametrization $(\x(s),\y(s))=(2s(s+\mu_2),2s(s-\mu_1))$, \textit{i.e.}
    $$\mathcal{S}=\{(x,y)\in\mathbb C^2 : K(x,y)=0\}=\left\{ (\x(s),\y(s)) : s\in\mathbb{C} \right\}.$$ 
We also introduce symmetries
$
\zeta s:=-s-\mu_2$ and $\eta s :=-s+\mu_1
$
which are automorphisms of the surface $\mathcal{S}$ leaving invariant $\x$ and $\y$ respectively, \textit{i.e.}
$\x(\zeta s)=\x (s)$ and $\y(\eta s)=\y(s)$.
On the Riemann surface $\mathcal{S}$, the functional equation takes the simplified form
$$
0 = k_1(s)\varphi_1(s) + k_2(s)\varphi_2(s)
$$
where we note $\varphi_1(s):=\phi_1(\y (s))$ and $\varphi_2(s):=\phi_2(\x (s))$.
The Laplace transform $\varphi_1$ is then meromorphically continued to the whole complex plane, and its poles and its behavior near infinity are analyzed. From the functional equation, a difference equation of the form
$$
\varphi_1(s+1) = G(s)\varphi_1(s)
$$
is derived, where 
$$G(s)=\frac{(s-s_1)(s+s_2+\mu_2)}{(s-s_2)(s+s_1+\mu_2)}$$
is a simple rational function. The subsequent sections are devoted to solving this difference equation which has several solutions but only one with the desired poles and the right behavior at infinity.
    \item Section~\ref{subsec:decoupl} introduces the decoupling function $D$ already defined in~\eqref{eq:Ddec} using $\Gamma$ functions. This function simplify the structure of the problem since it is a particular solution of the difference equation: 
    $$D(s+1)=G(s)D(s).$$ 
In particular, this implies that $\varphi_1 / D$ is a $1$-periodic function:
$$
\frac{\varphi_1(s+1)}{D(s+1)}=\frac{\varphi_1(s)}{D(s)}.
$$  
    The section then determines the necessary and sufficient conditions for $D$ to be a rational fraction, namely:
    $$
\gamma\in \mathbb Z\text{ or }\{\gamma_1,\gamma_2\}\subset \mathbb Z.
$$
    \item Section~\ref{subsec:Tutte} defines two types of Tutte invariants: 
\begin{itemize}
\item Type~I invariant: the meromorphic $1$-periodic functions on $\mathbb{C}$,
\item Type~II invariant: the meromorphic functions on $\mathbb{C}$ invariant by $\eta$ and $\zeta$.
\end{itemize}    
Thanks to the decoupling function \( D \) studied in the previous section, we observe that \( D/\varphi_1 \) is $1$-periodic and is then a Type~I invariant. Moreover, in the cases where \( D \) is a rational function, either \( D/\varphi_1 \) or \( (D/\varphi_1)^2 \) turns out to be a Type~II invariant. These are called \emph{unknown invariants}, because $\varphi_1$ is the function we are looking for.
For each type of invariant, a conformal gluing function which is in a sense a \emph{canonical invariant}, is introduced, namely:
$$
\wb(s):=\tan\big(\pi(s+\mu_2-1/2)\big)
\quad\text{and}\quad
\w(s):= \cos\big(2\pi(s+\mu_2/2)\big).
$$
The main idea behind Tutte’s invariant method is to use invariant lemmas stating that, under certain growth conditions at infinity, all invariants can be expressed as rational functions of the canonical invariants.
    \item Section~\ref{sec:Laplace} provides the explicit expression of the Laplace transform $\varphi_1$ in terms of $\wb$ or $\w$ using the invariant lemmas proved in the previous section. The proof of Theorem~\ref{thm:main} is obtained by distinguishing several cases, depending on the values of $\gamma$, $\gamma_1$, and $\gamma_2$, as detailed in Theorems~\ref{thm:N}, \ref{thm:-N}, \ref{thm:notN}, and~\ref{thm:general}.

    \item Section~\ref{sec:hierarchy} analyzes the differential hierarchy satisfied by the Laplace transform and proves the necessary and sufficient condition stated  in Table~\ref{tab:classification}.
    The explicit formulas for the Laplace transforms determined in the previous section using Tutte's invariant method give us the sufficient conditions. To prove that they are necessary, we use a result from Galois difference theory.
    \item Section~\ref{sec:density} presents an explicit expression of the invariant measure density $\nu_1$ introduced in~\eqref{eq:densitedef}. To do this, we invert the Laplace transforms using Mittag-Leffler expansions and some polynomial differential operator. 
    \item Appendix~\ref{sec:r1r2} discusses special cases where~\eqref{eq:H5} is not satisfied, \textit{i.e.} $r_1=-1$ or $r_2=-1$. The proof strategy is identical to the general case, so we will just give the main steps in the reasoning and the results.
    \item Appendix~\ref{sec:homogeneity} derives homogeneity relations satisfied by the Laplace transform and the invariant measure density. Thanks to these change-of-variable formulas, we can explicitly generalize our results without relying on assumption~\eqref{eq:H4}.
    \item Appendix~\ref{sec:C} draws connections with nondegenerate systems. In particular, it highlights the connection between the parameter $\gamma$ introduced in this article and the well-known parameter $\alpha$ introduced by Varadhan and Williams in the study of reflected Brownian motion in cones.
    \item Appendix~\ref{sec:notations} is an index of all the notations introduced throughout the paper, provided as a reference for the reader.
\end{itemize}

\begin{remark}[Overview of Tutte's invariant method]
Between the 1970s and 1990s, Tutte developed an algebraic approach based on invariants to solve a functional equation arising in the enumeration of colored triangulations~\cite{tutte_chromatic_1995}. More recently, an analytic counterpart of Tutte’s method has emerged, refining and extending the classical analytic approach based on boundary value problems. This refined method has been successfully applied to the enumeration of planar maps~\cite{albenque_menard_schaeffer,BeBM-11} and to lattice walks confined to the quarter plane~\cite{bernardi2021counting}. This analytic-invariant approach has also proven effective for the study of continuous-time stochastic processes~\cite{BoMe-El-Fr-Ha-Ra,FIKR23+}. 
\end{remark}
\begin{remark}[Overview of difference Galois theory]
\label{rem:galois}
Much like classical Galois theory, difference Galois theory establishes a correspondence between the algebraic relations satisfied by the solutions of a linear functional equation and the algebraic dependencies among its coefficients. This theory provides powerful tools for analyzing the differential properties of solutions to such equations. For an accessible introduction to the subject, see~\cite{vdPS}. In recent years, difference Galois theory has been successfully applied to the enumeration of discrete walks in the quadrant~\cite{DreyfHardtderiv,dreyfus2018nature,DHRS0}.
Recently, difference Galois theory has been applied for the first time to continuous random process such as the reflected Brownian motion in~\cite{BoMe-El-Fr-Ha-Ra}.
\end{remark}

\section{Meromorphic continuation and difference equation}\label{sec:FuncEqToDiffEq}

\subsection{Uniformization and Galois automorphisms}\label{subsec:unif}

Recall that we defined in \eqref{eq:noyau} the kernel which under the hypothesis~\eqref{eq:H4} is equal to $K(x,y)= (x-y)^2 - 2\mu_1 x - 2\mu_2 y $. The goal of this first subsection is to study the set of points that cancel the kernel, \textit{i.e.} the set
\begin{equation}\label{eq:parabola}
\mathcal{S}:=\{(x,y)\in\mathbb C^2 : K(x,y)=0\}.
\end{equation}
\begin{proposition}[Uniformization]\label{prop:unif} The set $\mathcal{S}$ is a Riemann surface uniformized by 
\begin{equation}\label{eq:unif}
\Big(\x(s),\y(s)\Big)=\Big(2s(s+\mu_2),2s(s-\mu_1)\Big)
\end{equation}
\textit{i.e.} $\mathcal{S}$ is conformally equivalent to the complex plane through the map $(\x,\y): \mathbb{C}\to \mathcal{S}$, we then have
$$\mathcal{S}=\left\{ \Big(\x(s),\y(s)\Big) : s\in\mathbb{C} \right\} .$$
\end{proposition}

\begin{proof}
Recall that by \eqref{eq:H4} we have $\sigma_1=\sigma_2=1$ and $\mu_1+\mu_2=1$. One can easily check that
$$
K(\x(s),\y(s))=(2s(s+\mu_2)-2s(s-\mu_1))^2 - 4\mu_1 s(s+\mu_2) - 4\mu_2 s(s-\mu_1)=0$$
so that $(\x,\y)(\mathbb C)\subseteq \mathcal{S}$. Conversely, if $(a,b)\in\mathcal{S}$ then $(a-b)^2={2}\mu_1a+{2}\mu_2b$ and
$$\x\left(\frac{a-b}{2}\right)=2\left(\frac{a-b}{2}\right)\left(\frac{a-b}{2}+\mu_2\right)=a(\mu_1+\mu_2)=a.$$
Similarly, $\y(\frac{a-b}{2})=b$ and we obtain $\mathcal{S} \subseteq (\x,\y)(\mathbb C)$ which makes $(\x,\y):\mathbb C\to \mathcal{S}$ surjective. The relation $\frac{\x(s)-\y(s)}{2}=s$ immediately implies that $s\mapsto (\x (s),\y (s))$ is injective.
\end{proof}

We now define key automorphisms of this Riemann surface that will be useful in our study.
\begin{definition}[Galois automorphisms and branch points]\label{rem1}
Let $\zeta$ and $\eta$ be the automorphisms of the surface $\mathcal{S}$ (through the uniformization \eqref{eq:unif}) defined by
\begin{equation}\label{eq:etazeta}
\zeta s:=-s-\mu_2=-s+2s_-\text{ and }\eta s :=-s+\mu_1=-s+2s_+
\end{equation}
where we refer to the fixed points of these automorphisms as
\begin{equation}\label{eq:s+s-}
s_-:=\frac{-\mu_2}{2}<0\text{ and }s_+:=\frac{\mu_1}{2}>0 .
\end{equation} 
The automorphism $\zeta$ (\textit{resp.} $\eta$) is the central symmetry around $s_-$ (\textit{resp.} $s_+$). The points $s_-$ and $s_+$ are branch points. See Figure~\ref{fig:fundamental} for an illustration.
\end{definition}

\begin{proposition}[Fundamental properties of $\eta$ and $\zeta$]
The functions $\x$ and $\y$ satisfy the fundamental invariance properties
\begin{equation}\label{eq:inv}
\x(\zeta s)=\x (s)\text{ and }\y(\eta s)=\y(s).
\end{equation}
These automorphisms satisfy the key identities
\begin{equation}\label{eq:etacirczeta}
\zeta \circ \eta (s)=s-1\text{ and }\eta \circ \zeta (s)=s+1 .
\end{equation}
\end{proposition}
\begin{proof}
Let us note that
\begin{equation}
\Big(\x(s),\y(s)\Big)=\Big(2s(s-2s_-),2s(s-2s_+)\Big)
\label{eq:xyspm}
\end{equation}
and remember that $\mu_1+\mu_2=1$ from which we deduce the proposition directly.
\end{proof}
\noindent The last equality of the previous proposition, which gives the value of $\eta \circ \zeta$, will produce the  finite-difference equation in Section~\ref{subsec:prolongement}.

\subsection{Analytical continuation and finite difference equation}\label{subsec:prolongement}
Let us set 
\begin{equation}\label{eq:defvarphi}
\varphi_1(s):=\phi_1(\y (s))
\quad\text{and}\quad
\varphi_2(s):=\phi_2(\x (s))
\end{equation}
where $\phi_1$ and $\phi_2$ are the Laplace transforms previously defined. 

\begin{remark}[From $\varphi_1(s)$ to $\phi_1(y)$] To obtain $\phi_1(y)$, it is sufficient to compute $\varphi_1(s)$. Since
\begin{equation}\label{eq:sys}
%(2(s-s_-)-1)^2=(s-s_+)^2
(2s-\mu_1)^2=2\y(s)+\mu_1^2
\end{equation}
we obtain for $y=\y(s)=\y(-s+\mu_1)$, %which leads to 
the following relation between $\varphi_1(s)$ and $\phi_1(y)$:
\begin{equation}\label{eq:varphitophi}
\phi_1(y)=\varphi_1\left(\frac{1}{2}\left(\mu_1\pm\sqrt{2y+\mu_1^2}\right)\right).
\end{equation}
\end{remark}

The functions $\varphi_1$ and $\varphi_2$ are both defined and analytic on their respective domains
\begin{equation}\label{eq:domains}
    \Delta_\mathrm{y}:=\{s\in\mathbb C : \mathfrak{Re}(\y (s))<0\}\text{ and } \Delta_\mathrm{x}:=\{s\in\mathbb C : \mathfrak{Re}(\x (s))<0\}.
\end{equation}
These sets are open sets. See Figure~\ref{fig:fundamental} for an illustration. The goal of this subsection is to continue meromorphically $\varphi_1$ and $\varphi_2$ to the whole complex plane $\mathbb{C}$. In doing so, we will establish a finite-difference equation that will be central to our study.\\
\begin{figure}
    \centering
    \includegraphics[width=8.2cm]{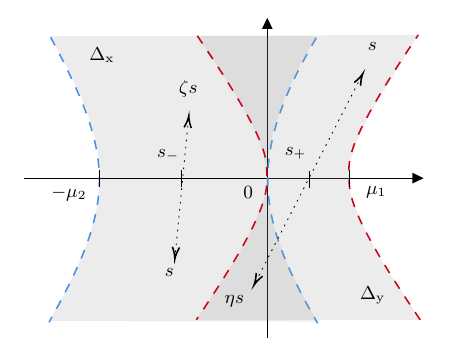}\\
    \caption{The regions of convergence of the Laplace transforms $\phi_1$ (bounded by the red curves) and $\phi_2$ (bounded by the blue curves), seen through the uniformization~\ref{eq:unif}, and their respective Galois automorphisms.} 
    \label{fig:fundamental}
\end{figure}

To write the first step of this meromorphic extension, recall that in~\eqref{eq:k12} we have defined $k_1$ and $k_2$ which are coefficients of the functional equation~\eqref{eq:FE}.
For the sake of readability and conciseness, we will write $k_1(s)$ for $k_1(\x(s),\y(s))$ and $k_2(s)$ for $k_2(\x(s),\y(s))$, so that 
\begin{equation}\label{eq:defk12s}
k_1(s)=\x(s)+r_1\y(s)
%=2s(s+\mu_2)+r_1 2s(s-\mu_1) 
=2s(s+\mu_2+r_1s -r_1\mu_1)={2}(1+r_1)s(s-s_1),
\end{equation}
and
$$
k_2(s)=\y(s)+r_2\x(s)
=2s(s-\mu_1+r_2s+r_2\mu_2)
={2}(1+r_2)s(s-s_2)
$$
where we recall that we defined
\begin{equation}\label{eq:k2}
s_1:=\frac{r_1\mu_1-\mu_2}{1+r_1} \quad\text{and}\quad s_2:=\frac{\mu_1-r_2\mu_2}{1+r_2}
\end{equation}
which are points that cancel $k_1$ and $k_2$ (which are both well defined, as $r_1\ne-1$ and $r_2\ne -1$). We now state a small technical lemma used in the rest of the article on the values that $s_1$ and $s_2$ cannot take under our assumptions.
\begin{lemma}[Constraints on $s_1$ and $s_2$]\label{lemma:s1s2} Under the hypothesis~\eqref{eq:H1}-\eqref{eq:H4} the following holds \begin{itemize}
\item[$\mathrm{(a)}$]$s_1\ne s_2$,
\item [$\mathrm{(b)}$]$s_1\notin [0,\mu_1]$,
\item [$\mathrm{(c)}$] $s_2\notin [-\mu_2,0]$.
\end{itemize}
\end{lemma}
\begin{proof} Recall that we have assumed throughout this article (except in Appendix~\ref{sec:r1r2}) that $r_1\neq-1$ and $r_2\neq -1$.
\begin{itemize}
\item[(a)] Assuming that $s_1=s_2$, we obtain $(1+r_2)(r_1\mu_1-\mu_2)=(1+r_1)(\mu_1-r_2\mu_2),$ and then $1-r_1r_2=0$ (or equivalently, $r_1=1/r_2$). According to~\eqref{eq:H1}, both $r_1$ and $r_2$ are therefore positive. Replacing $r_1$ by $1/r_2$ in~\eqref{eq:H2} yields 
$$\mu_1-r_2\mu_2>0\text{ and }\mu_2-\mu_1/r_2>0,$$
and multiplying the second inequality by $r_2>0$ leads to 
$\mu_1-r_2\mu_2>0\text{ and }\mu_1-r_2\mu_2<0$
which cannot be true.
\item[(b)] By \eqref{eq:H2}, $r_1\mu_1-\mu_2\neq 0$ so $s_1\neq 0$. Notice that 
 $$s_1:=\frac{r_1\mu_1-\mu_2}{1+r_1}>0$$
 implies that $r_1\mu_1-\mu_2$ and $1+r_1$ share the same sign. But $r_1\mu_1-\mu_2<0$ by~\eqref{eq:H2}, so that 
 $1+r_1<0$. Now, consider the second inequality:
 $$\frac{r_1\mu_1-\mu_2}{1+r_1}\leqslant \mu_1.$$
Mutiplying both sides by $1+r_1<0$ yields $r_1\mu_1-\mu_2\geqslant \mu_1(1+r_1)$, \textit{i.e.}, $-\mu_2\geqslant \mu_1$ which is always false by~\eqref{eq:H3}.
\item[(c)] Similarly to $\mathrm{(b)}$, $s_1\neq 0$ and 
 $$s_2:=\frac{\mu_1-r_2\mu_2}{1+r_2}<0$$
 implies that $\mu_1-r_2\mu_2$ and $1+r_2$ have opposite signs, yet $\mu_1-r_2\mu_2>0$ by~\eqref{eq:H2}, so $1+r_2<0$. Mutiplying both sides by $1+r_2<0$ in $-\mu_2\leqslant s_2$ yields $-\mu_2(1+r_2)\geqslant \mu_1-r_2\mu_2$, \textit{i.e.}, $-\mu_2\geqslant \mu_1$ which never holds by~\eqref{eq:H3}.
\end{itemize}
\end{proof}
We now establish the invariance properties of $\varphi_1$ and $\varphi_2$.
\begin{lemma}[Invariance on $\mathcal{S}$]\label{lem:abc}
The functions $\varphi_1$ and $\varphi_2$ satisfy the following relations:
\begin{equation}\tag{$\mathfrak{a}$}\label{eq:eqfunc}
    \forall s \in \Delta_\mathrm{x}\cap\Delta_\mathrm{y},\quad  k_1(s)\varphi_1(s)+k_2(s)\varphi_2(s)=0,
\end{equation}
\begin{equation}\tag{$\mathfrak{b}$}\label{eq:inv1}
    \forall s\in \Delta_\mathrm{y},\quad
    \varphi_1(\eta s)=\varphi_1(s),
\end{equation}
\begin{equation}\tag{$\mathfrak{c}$}\label{eq:inv2}
     \forall s \in \Delta_\mathrm{x},\quad   \varphi_2(\zeta s)=\varphi_2(s).
\end{equation}
\end{lemma}
\begin{proof}
Since $K(\x(s),\y(s))=0$, equation~\eqref{eq:eqfunc} is just the functional equation~\eqref{eq:FE} evaluated at $(\x(s),\y(s))\in\mathcal{S}$. To establish Equation~\eqref{eq:inv1}, we must first note that the domain $\Delta_\mathrm{y}$ defined in~\eqref{eq:domains} remains invariant under $\eta$ (indeed for all $s\in\Delta_\mathrm{y}$, $\mathfrak{Im}(\y(\eta s))=\mathfrak{Im}(\y(s))<0$, hence $\eta \Delta_\mathrm{y}\subset \Delta_\mathrm{y}$ and the reverse inclusion is due to $\eta^2=\text{Id}_\mathbb C$), see Figure~\ref{fig:fundamental}. The invariance property~\eqref{eq:inv1} for $\varphi_1$ follows directly from the invariance property~\eqref{eq:inv} satisfied by $\y$. The proof for~\eqref{eq:inv2} proceeds in the same way.
\end{proof}
Now, we progressively extend $\varphi_1$, $\varphi_2$ and the relations~\eqref{eq:eqfunc}, \eqref{eq:inv1} and \eqref{eq:inv2} to the complex plane. First we would like to understand what is the domain in the $s$-plane corresponding to $\Delta_\mathrm{x}$ and $\Delta_\mathrm{y}$.  
To state the next lemma, we introduce four new functions $f^\pm,g^\pm : \mathbb R\to\mathbb R$ defined by 
\begin{equation}
f^\pm (b):=s_-\pm\sqrt{s_-^2+b^2}\text{ and }g^\pm (b):=s_+\pm \sqrt{s_+^2+b^2} .
\label{eq:fgpm}
\end{equation}
\begin{lemma}[Domains $\Delta_\mathrm{x}$ and $\Delta_\mathrm{y}$]\label{lemma:DxDy}
The domains $\Delta_\mathrm{x}$ and $\Delta_\mathrm{y}$ satisfy the following equalities:
\begin{equation}\label{eq:DxDy}
   \Delta_\mathrm{x}=\{a+ib : f^-(b)<a<f^+(b)\}, \quad \Delta_\mathrm{y}=\{a+ib : g^-(b)<a<g^+(b)\},
\end{equation} 
In particular, $\Delta_\mathrm{x}$ and $\Delta_\mathrm{y}$ are open and connected.
\label{lem:DxDy}
\end{lemma}
\begin{proof}
For $s=a+ib$, we have by~\eqref{eq:unif} that
\begin{equation}\label{eq:Reys}
\mathfrak{Re}(\x(s))=2a^2-4as_--2b^2
\quad\text{and}\quad
\mathfrak{Re}(\y(s))=2a^2-4as_+-2b^2.
\end{equation}
By the definition of $f^\pm$ and $g^\pm$ given in~\eqref{eq:fgpm} we have
$$
\mathfrak{Re}(\x(s))=2(a-f^+(b))(a-f^-(b))
\quad\text{and}\quad
\mathfrak{Re}(\y(s))=2(a-g^+(b))(a-g^-(b))
$$ 
and then by the definition~\eqref{eq:domains} of the domains $\Delta_\mathrm{x}$ and $\Delta_\mathrm{y}$ we obtain~\eqref{eq:DxDy}.
\end{proof}

This previous lemma and the following are illustrated on Figure \ref{fig:1}.

\begin{lemma}[Some inequalities]\label{lemma:ineq}
For all $b\in\mathbb R$, the following holds
\begin{equation}
    -g^+(b)-\mu_2<f^-(b)< g^-(b)\leqslant f^+(b)<g^+(b)<-f^-(b)+\mu_1
\label{eq:ineqlem}
\end{equation}
and $g^-(b)=f^+(b)$ if and only if $b=0$. 
\label{lem:fginegalite}
\end{lemma}
\begin{proof}
Recall that $\mu_1>0$, $\mu_2>0$ and $s_-=-\mu_2/2$, $s_+=\mu_1/2$. Using the definitions of $f^\pm$ and $g^\pm$ in~\eqref{eq:fgpm} the first inequality of~\eqref{eq:ineqlem} is equivalent to
\begin{align*}
& f^-(b)+\mu_2+g^+(b)>0 \Longleftrightarrow  \frac{\mu_1}{2}+\frac{\mu_2}{2}
+\sqrt{\frac{\mu_1^2}{4}+b^2} > \sqrt{\frac{\mu_2^2}{4}+b^2}
\\
\Longleftrightarrow \ &
\left( \mu_1+\mu_2
+\sqrt{\mu_1^2+4 b^2} \right)^2 > \mu_2^2+4b^2
\Longleftrightarrow 
 2\mu_1^2+2\mu_1\mu_2  +2(\mu_1+\mu_2)\sqrt{\mu_1^2+4 b^2}> 0
\end{align*}
and the last inequality is obviously true. The other inequalities of~\eqref{eq:ineqlem} can be proved in the same way. For the equality case, note that
$$f^+(b)-g^-(b)=-\frac{\mu_2}{2}+\sqrt{\frac{\mu_2^2}{4}+b^2}-\frac{\mu_1}{2} +\sqrt{\frac{\mu_1^2}{4}+b^2}.$$
cancels out only for $b=0$.
\end{proof}
We now adopt the notations 
\begin{equation}\label{eq:Delta}
\delta:=\Delta_\mathrm{x}\cap\Delta_\mathrm{y},\quad\quad \Delta:=\Delta_\mathrm{x}\cup \Delta_\mathrm{y}\cup\{0\},\quad\quad \widetilde{\Delta}:=\Delta\cup \eta \Delta\cup \zeta \Delta.
\end{equation}
\begin{lemma}[Domains $\delta$, $\Delta$ and $\widetilde{\Delta}$]\label{lemma:domains}
The domains $\delta$, $\Delta$ and $\widetilde{\Delta}$ satisfy the following equalities:
\begin{equation}
 \delta = \{a+ib : g^-(b)<a<f^+(b)\},
 \label{eq:deltagf}
\end{equation}
\begin{equation}
 \Delta = \{a+ib : f^-(b)<a<g^+(b)\},\,\,
 \label{eq:Deltagf}
\end{equation}%
\begin{equation}
\quad\quad\quad\quad\quad\,\,\widetilde{\Delta}= \{a+ib : -g^+(b)-\mu_2<a<-f^-(b)+\mu_1\}.
\label{eq:etazetaDxDy}
\end{equation}
In particular,
\begin{itemize}
    \item $\delta$ is a nonempty open set,
    \item $\Delta$ and $\widetilde{\Delta}$ are open and connected.
\end{itemize}
\label{lem:deltaDelta}
\end{lemma}
\begin{proof}
This lemma derives from Lemmas~\ref{lem:DxDy} and~\ref{lem:fginegalite}.
By \eqref{eq:DxDy} and the definition of $\delta$ and $\Delta$, we then find~\eqref{eq:deltagf} and~\eqref{eq:Deltagf}. 
Since $\zeta s:=-s-\mu_2$ and $\eta s :=-s+\mu_1$ we deduce 
$$\zeta  \Delta= \{a+ib : -g^+(b)-\mu_2 <a<-f^-(b)-\mu_2\}$$
and
 $$\eta  \Delta= \{a+ib : -g^+(b)+\mu_1 <a<-f^-(b)+\mu_1\}.$$
Since $\mu_1+\mu_2=1$ we have $f^-=-f^+-\mu_2$ and $g^-=-g^++\mu_1$. Using these relations and Lemma~\ref{lemma:ineq}, one can immediately deduce that $-f^++\mu_1\leqslant g^+$ and $f^-\leqslant -g^--\mu_2$. These inequalities describe how $\Delta$, $\eta\Delta$ and $\zeta\Delta$ overlap, leading to \eqref{eq:etazetaDxDy}.
\end{proof}
 
\begin{figure}
    \centering
    \includegraphics[width=17cm]{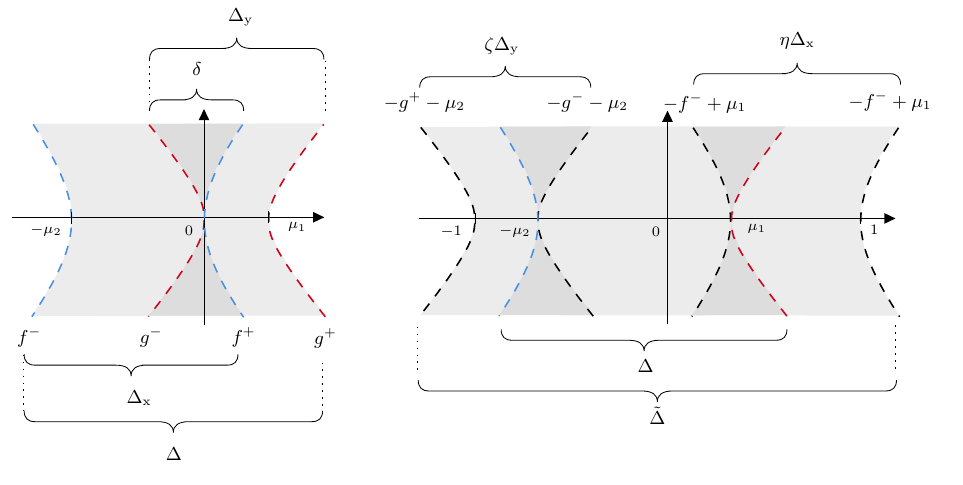}\\
    \caption{The sets to which the Laplace transforms are successively extended. Each curve is labelled with a function $h$ such that the curve is defined by $\{h(y)+iy : y\in\mathbb R\}$. Colors are chosen as in Figure~\ref{fig:fundamental}.} 
    \label{fig:1}
\end{figure}

Recall that $\varphi_1$ and $\varphi_2$ are respectively defined and holomorphic  on $\Delta_\mathrm{y}$ and  $\Delta_\mathrm{x}$. The next step is to extend them on  $\Delta$.
\begin{proposition}[Meromorphic continuation to $\Delta$]\label{prop:conttoDelta} The functions $\varphi_1$ and $\varphi_2$ can be meromorphically continued to $\Delta$ and the functional equation~\eqref{eq:eqfunc} remains valid on $\Delta$.
\end{proposition}
\begin{proof}
Recall that $\varphi_1$ is initially defined on $\Delta_\mathrm{y}$ and $\varphi_2$ on $\Delta_\mathrm{x}$, see~\eqref{eq:domains}. The functional equation~\eqref{eq:eqfunc} of Lemma~\ref{lem:abc} implies that
\begin{equation}\label{eq:pr1}
    \varphi_1(s)=-\frac{k_2(s)}{k_1(s)}\varphi_2(s),\quad\quad \forall s\in\delta=\Delta_\mathrm{x}\cap\Delta_\mathrm{y} \ne \varnothing .
\end{equation}
By Lemma~\ref{lem:deltaDelta}, $\delta=\Delta_\mathrm{x}\cap\Delta_\mathrm{y}$ is a nonempty open set, then
the principle of analytical extension allows us to extend $\varphi_1$ to $\Delta_{\mathrm{x}}$ by the formula~\eqref{eq:pr1} setting 
$$
\varphi_1(s)=-\frac{k_2(s)}{k_1(s)}\varphi_2(s),\quad\quad \forall s\in\Delta_{\mathrm{x}}.
$$ We have thus extended $\varphi_1$ to the connected domain $\Delta_\mathrm{x}\cup\Delta_\mathrm{y}=\Delta\setminus\{0\}\supset \delta$. The same method works to extend $\varphi_2$ to $\Delta\setminus\{0\}$. We need to show that $0$ is a removable singularity for $\varphi_1$ and $\varphi_2$. As is the case for all Laplace transforms of finite measure, $\phi_1$ and $\phi_2$ are continuous at $0$ and
\begin{equation}\label{eq:contat0}
\lim_{\substack{y\to 0\\ \mathfrak{Re}(y)<0}}\phi_1(y)=\phi_1(0)=\boldsymbol{\nu}_1(\mathbb R_+),
\quad 
\lim_{\substack{x\to 0\\ \mathfrak{Re}(x)<0}}\phi_2(x)=\phi_2(0)=\boldsymbol{\nu}_2(\mathbb R_+).
\end{equation}
Recalling the definition of $\Delta_\mathrm{y}$ given in~\eqref{eq:domains} we have, on the one hand
     $$\lim_{\substack{s\to 0 \\ s\in \Delta_{\mathrm{y}}}}\varphi_1(s)=\lim_{\substack{y\to 0\\ \mathfrak{Re}(y)<0}}\phi_1(y)=\boldsymbol{\nu}_1(\mathbb R_+)$$
     and
     $$\lim_{\substack{s\to 0 \\ s\in \Delta_{\mathrm{x}}}}\varphi_1(s)=\lim_{\substack{s\to 0 \\ s\in \Delta_{\mathrm{x}}}} -\frac{k_2(s)}{k_1(s)}\varphi_2(s)=\frac{\mu_2r_2-\mu_1}{\mu_1 r_1-\mu_2}\boldsymbol{\nu}_2(\mathbb R_+).$$
    on the other hand. Hence $\varphi_1$ is bounded near $0$ since $\Delta_{\mathrm{x}}\cup \Delta_{\mathrm{y}}\cup \{0\} =\Delta$ is a neighborhood of $0$ (which follows from the fact that we have seen in Lemma~\ref{lem:deltaDelta} that $\Delta$ is open). This proves that $0$ is a removable singularity for $\varphi_1$.    
    Note that $\varphi_1$ is then continuous and the above limits must coincide, which shows that
    $$\boldsymbol{\nu}_1(\mathbb R_+)=\frac{\mu_2r_2-\mu_1}{\mu_1r_1-\mu_2}\boldsymbol{\nu}_2(\mathbb R_+).$$
The same way, we show that $0$ is a removable singularity for $\varphi_2$ which concludes the proof.
\end{proof}

\begin{remark}[Invariances in $\Delta_\mathrm{x}$ and $\Delta_\mathrm{y}$]
\label{rem:invDxy}
At this stage of the argument, we meromorphically extended $\varphi_1$ to the set $\Delta$ (Proposition~\ref{prop:conttoDelta}) and by the principle of analytic continuation the invariance properties \eqref{eq:inv1} and \eqref{eq:inv2} are then valid on the biggest subset of $\Delta$ stable by $\eta$ or $\zeta$ respectively, \textit{i.e.}, \eqref{eq:inv1} holds on
$$\Delta\cap\eta \Delta=\{a+ib : -g^+(b)+\mu_1 <a<g^+(b)\}=\{a+ib : g^-(b) <a<g^+(b)\}=\Delta_\mathrm{y},$$ and similarly, \eqref{eq:inv2} holds on $$\Delta\cap\zeta \Delta=\{a+ib : f^-(b) <a<-f^-(b)-\mu_2\}=\{a+ib : f^-(b) <a<f^+(b)\}=\Delta_\mathrm{x}.$$
This shows that, at this stage, the domains on which the invariance properties hold are the same as in Lemma~\ref{lem:abc} and have therefore not yet been extended.
 \end{remark}

Let us now resume the task of meromorphically extending $\varphi_1$ and $\varphi_2$ to the complex plane. Recall that we already extended them to $\Delta$. The next step is to continue them on  $\widetilde{\Delta}$.
\begin{proposition}[Meromorphic continuation to $\widetilde{\Delta}$] \label{prop:continuationDeltatilde}
The functions $\varphi_1$ and $\varphi_2$ can be meromorphically continued to $\widetilde{\Delta}=\Delta\cup \eta \Delta\cup \zeta \Delta$. The functional equation~\eqref{eq:eqfunc} remains valid on $\widetilde{\Delta}$, the invariance propertie~\eqref{eq:inv1} remains valid on $\widetilde{\Delta} \cap \eta\widetilde{\Delta}$ and~\eqref{eq:inv2} remains valid on $\widetilde{\Delta} \cap \zeta\widetilde{\Delta}$.
\end{proposition}
\begin{proof}
First, let us see how to extend $\varphi_1$ to $\Delta\cup\eta\Delta$ via the relation~\eqref{eq:inv1}. We set for $s\in \eta \Delta$ (and then $\eta s\in \Delta$)
$$
\varphi_1(s)=\varphi_1(\eta s) .
$$ 
Since $\Delta \cap \eta \Delta=\Delta_\mathrm{y} \neq \varnothing$ (see Remark~\ref{rem:invDxy}), the principle of analytic continuation can be used to extend $\varphi_1$ to the open connected set $\Delta\cup\eta \Delta$. We apply the same method to extend $\varphi_2$ to $\Delta\cup \zeta\Delta$ via \eqref{eq:inv2}. We extend $\varphi_1$ to $\zeta\Delta$ and $\varphi_2$ to $\eta \Delta$ thanks to \eqref{eq:eqfunc} in the same way as Proposition~\ref{prop:conttoDelta}. The relations \eqref{eq:eqfunc}, \eqref{eq:inv1} and \eqref{eq:inv2} stay valid by the analytic continuation principle.
\end{proof}

\begin{lemma}[Covering of the complex plane]\label{lemma:cover}
 The inclusion $\{s\in\mathbb C : \mathfrak{Re}(s)\in (-1,1)\}\subseteq \widetilde{\Delta}$ holds. As a result,
 \begin{equation}\label{eq:cover}
 \bigcup_{k\in \mathbb Z} (\widetilde{\Delta}+k)=\mathbb C.
 \end{equation}
\end{lemma}
\begin{proof}
Recall~\eqref{eq:etazetaDxDy} which gives a precise description of the set $\widetilde{\Delta}$.  We have  $s=a+ib\in \widetilde{\Delta}$ if and only if
\begin{equation}\label{eq:widetildeDelta}
 -g^+(b)-\mu_2<a<-f^-(b)+\mu_1.
\end{equation}   
From the definitions~\eqref{eq:fgpm} of $f^-$ and $g^+$, it is easy to see that
$$-g^+(b)-\mu_2<-g^+(0)-\mu_2=-1\text{ and }-f^-(b)+\mu_1>-f^-(0)+\mu_1=1,$$
which implies that 
    $$\{a+ib\in\mathbb C : -1\leq a\leq 1\}\subseteq\widetilde{\Delta}$$
 as stated. This strip is two unit-large so that the union of all its integer translations covers the whole complex plane.
\end{proof}

Recall that $\varphi_1$ is defined on $\widetilde{\Delta}$.
The same way we needed to be able to evaluate $\varphi_1$ at $s$ and $\eta s$ simultaneously to prove Lemma~\ref{lem:abc}, for the proof of Theorem~\ref{thm:prolong}, we will need to evaluate $\varphi_1$ at $s$, $\zeta s$ and $s+1$. The following Lemma shows that we are indeed allowed to do so.
\begin{lemma}[A nonempty intersection]\label{lemma:evalphi1phi2}
The intersection $\widetilde{\Delta}\cap \zeta\widetilde{\Delta} \cap (\zeta\circ \eta) \widetilde{\Delta}$ is open and nonempty. 
\end{lemma}
\begin{proof} Let us show that $-\frac{1}{2}$ belongs to this set. Recall from Lemma~\ref{lemma:cover} that $(-1,1)\subset \widetilde{\Delta}$, which immediately implies that $-\frac{1}{2}\in \widetilde{\Delta}$. Moreover, both $\zeta^{-1}(-\frac{1}{2})=\frac{1}{2}-\mu_2$ and $(\zeta \circ \eta)^{-1}(-\frac{1}{2})=\frac{1}{2}$ are elements of $(-1,1)$, leading to $-\frac{1}{2}\in \widetilde{\Delta}\cap \zeta\widetilde{\Delta} \cap (\zeta\circ \eta) \widetilde{\Delta}$. This set is open as the intersection of open sets.
\end{proof}

\begin{theorem}[Meromorphic continuation to $\mathbb C$ and difference equation] \label{thm:prolong}
The function $\varphi_1$ can be meromorphically continued to $\mathbb C$, through the difference equation
\begin{equation}\label{eq:qdiff}
    \varphi_1(s+1)=G(s)\varphi_1(s)
\end{equation}
where
\begin{equation}
    G(s):=\frac{k_1(s)k_2(\zeta s)}{k_2(s)k_1(\zeta s)}=\frac{(s-s_1)(s+s_2+\mu_2)}{(s-s_2)(s+s_1+\mu_2)}.
    \label{eq:Gdef}
\end{equation}
The functional equation~\eqref{eq:eqfunc} and the invariance properties~\eqref{eq:inv1} and~\eqref{eq:inv2} remain valid on the whole complex plane $\mathbb{C}$.
\end{theorem}
\begin{proof} In Proposition~\ref{prop:continuationDeltatilde}, we have extended $\varphi_1$ and $\varphi_2$ to $\widetilde{\Delta}$.
   We now extend $\varphi_1$ to 
   $$\widetilde{\Delta}\cup (\widetilde{\Delta}-1)=\widetilde{\Delta}\cup (\zeta \circ \eta )\widetilde{\Delta}.$$ For $s\in \widetilde{\Delta}\cap \zeta \widetilde{\Delta}\cap (\widetilde{\Delta}-1)$ which is nonempty by Lemma~\ref{lemma:evalphi1phi2}, we can evaluate \eqref{eq:eqfunc} at $s$ and $\zeta s$ to get to following system:
$$\left\{\begin{array}{lr}
         0 = k_1(s)\varphi_1(s)+k_2(s)\varphi_2(s) &\quad (\text{L}_1)\\[0.2cm]
         0 = k_1(\zeta s)\varphi_1(\zeta s) + k_2(\zeta s)\varphi_2(\zeta s)& (\text{L}_2)
    \end{array}\right.$$
The invariance properties \eqref{eq:inv1} and \eqref{eq:inv2} yields
    $$\varphi_1(\zeta s)=\varphi_1(\eta \circ \zeta s)=\varphi_1(s+1)\text{ and }\varphi_2(\zeta s)=\varphi_2(s),$$
    hence 
    $$\left\{\begin{array}{lr}
         0 = k_1(s)\varphi_1(s)+k_2(s)\varphi_2(s) &\quad (\text{L}_1)\\[0.2cm]
         0 = k_1(\zeta s)\varphi_1( s+1) + k_2(\zeta s)\varphi_2(s)& (\text{L}'_2)
    \end{array}\right.$$
    We can therefore eliminate $\varphi_2(s)$ by considering $(\text{L}_1)-(\text{L}'_2)k_2(s)/k_2(\zeta s)$ and obtain 
    \begin{equation*}
    \varphi_1(s+1)=G(s)\varphi_1(s).
    \end{equation*} 
    The left-hand side of this equation is meromorphic on $\widetilde{\Delta}-1$, and the right-hand side is meromophic on $\widetilde{\Delta}$. The intersection of these domains is nonempty, since, for instance, $-\frac{1}{2}$ belongs to it (see the proof of Lemma~\ref{lemma:evalphi1phi2}), and their union is connected. The analytic continuation theorem allows us to extend $\varphi_1$ to $\widetilde{\Delta}\cup (\widetilde{\Delta}-1)$ using the formula $\varphi_1(s+1)=G(s)\varphi_1(s)$. The functional equation~\eqref{eq:eqfunc} and the invariance properties~\eqref{eq:inv1} and~\eqref{eq:inv2} remain valid. By induction, we can extend $\varphi_1$ to the domain $\bigcup_{k\leqslant 0}(\widetilde{\Delta}+k)$. Using the same argument, one can extend $\varphi_1$ from $\widetilde{\Delta}$ to $\widetilde{\Delta}+1$ and by recursion to $\bigcup_{k\geqslant 0}(\widetilde{\Delta}+k)$. According to Lemma~\ref{lemma:cover}, $\varphi_1$ is now (meromorphically) extended to the whole complex plane.
\end{proof}
\begin{remark}[Difference equation for $\varphi_2$]
A similar  method can be applied to extend $\varphi_2$ to $\mathbb C$. We have 
$$\left\{\begin{array}{lr}
         0 = k_1(s)\varphi_1(s)+k_2(s)\varphi_2(s), &\\[0.2cm]
         0 = k_1(\eta s)\varphi_1(s) + k_2( \eta s)\varphi_2(s-1),
    \end{array}\right.$$
    and therefore, 
    $$\varphi_2(s-1)=\frac{k_2(s)k_1(\eta s)}{k_1( s)k_2(\eta s)}\varphi_2(s)=\frac{(s-s_2)(s+s_1-\mu_1)}{(s-s_1)(s+s_2-\mu_1)}\varphi_2(s) .$$
\end{remark}

\subsection{Zeros and poles of $\varphi_1$ in two fundamental strips}

Now that we have meromorphically extended $\varphi_1$ to the entire complex plane $\mathbb{C}$, we study its zeros and poles in two fundamental strips, which will be useful later on. We introduce the sets
\begin{equation}\label{eq3}
\mathfrak{B}_1:=\{s\in\mathbb C : -\mu_2<\mathfrak{Re}(s)<\mu_1\}\text{ and }\mathfrak{B}_2:=\{s\in\mathbb C : s_-<\mathfrak{Re}(s)<s_+\},
\end{equation}
which are vertical strips delimited by two lines that we will denote by $d_1$ and $d_2$ for $\mathfrak{B}_1$ and $d_-$ and $d_+$ for $\mathfrak{B}_2$: 
\begin{equation}\label{eq2}
d_1:=\{s\in\mathbb C : \mathfrak{Re}(s)=\mu_1\}, \quad\quad d_2:=\{s\in\mathbb C : \mathfrak{Re}(s)=-\mu_2\},
\end{equation}
\begin{equation*}
d_-:=\{s\in\mathbb C: \mathfrak{Re}(s)=s_-\},\quad\; d_+:=\{s\in\mathbb C: \mathfrak{Re}(s)=s_+\}.\;\;\;
\end{equation*}

Note that $d_-$ and $d_+$ are respectively $\zeta$-stable and $\eta$-stable and that $d_2+1=d_1$. These domains will play a key role in defining the notion of invariant. We have the inclusions
\begin{equation}\label{eq:inclusionBdelta}
{\mathfrak{B}}_2 \subset {\mathfrak{B}}_1 \subset \Delta
\end{equation}
which are depicted in Figure~\ref{fig:inclusion}. 
The previous inclusions derives from~\eqref{eq:s+s-} and~\eqref{eq:Deltagf}. For later use, we now study the zeros and poles of $\varphi_1$ in these strips and its behavior at infinity.

\begin{figure}
\centering 
\includegraphics[width=7.1cm]{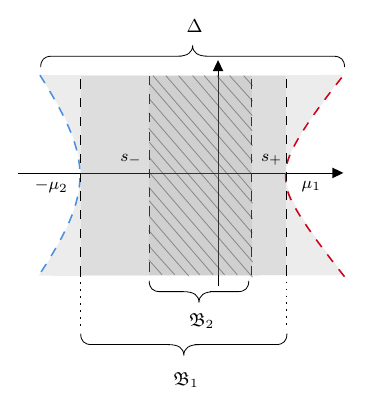}\\
\caption{The strips $\mathfrak{B}_2\subset \mathfrak{B}_1$ are subsets of the domain $\Delta$. Colors correspond to those used in Figures~\ref{fig:fundamental} and~\ref{fig:1}.}\label{fig:inclusion}
\end{figure}

\begin{lemma}[Zeros and poles in the strips]\label{lemma:stripe} The Laplace transform $\varphi_1$ satisfies the following properties:
\begin{enumerate}
\item  $\varphi_1$ has no real zeros in $\Delta$. In particular, $\varphi_1$ has no zeros in $\overline{\mathfrak{B}}_2\cap\mathbb R$, or in $\overline{\mathfrak{B}}_1\cap\mathbb R$.
\item $\varphi_1$ has at most one pole in $\overline{\mathfrak{B}}_1$. More precisely, $\varphi_1$ has a (simple) pole at $s_1$ if and only if 
\begin{equation}\label{eq:condpole}
 -\mu_2 \leqslant s_1< 0,\text{ \textit{i.e.} }s_1\in\overline{\mathfrak{B}}_1\cap \Delta_\mathrm{x}.
\end{equation}
In particular, this simple pole is in $\overline{\mathfrak{B}}_2\subset \overline{\mathfrak{B}}_1$ if and only if $s_-\leqslant s_1 <0$.
\item As $s$ stays in $\mathfrak{B}_1$,
\begin{equation}
\label{eq:lim0B}
\lim_{|s|\to \infty \atop s\in \mathfrak{B}_1}\varphi_1(s)=0.
\end{equation}
\end{enumerate}
Symmetric results hold for $\varphi_2$.
\end{lemma}
\begin{proof}
\begin{enumerate}
\item For all $y\in \mathbb R_{\leqslant 0}$ we have $\phi_1(y)\neq 0$ since it is the (convergent) integral of a nonzero positive function. By the formula defining $\y$ we have $s\in \mathbb{R}$ implies $\y(s)\in \mathbb{R}$. Then, $\varphi_1(s)=\phi_1(\y(s))$
yields that $\varphi_1(s)$ has no zeros in $\Delta_\mathrm{y}\cap \mathbb R$. 
By the same argument, $\varphi_2(s)$  has no zeros  on $\Delta_\mathrm{x}\cap\mathbb R$. 
By the relation~\eqref{eq:pr1} of Proposition~\ref{prop:conttoDelta}, for $s\in\Delta$ we have 
\begin{equation}
\varphi_1(s)=-\frac{k_2(s)}{k_1(s)}\varphi_2(s),
\label{eq:phi1kpole}
\end{equation}
so to study the zeros  of $\varphi_1(s)$ in $\Delta$ it suffices to study the zeros  of
\begin{equation}
\frac{k_2(s)}{k_1(s)}=\frac{(1+r_1)(s-s_2)}{(1+r_2)(s-s_1)}
\label{eq:k1:k2}
\end{equation}
and the poles of $\varphi_2$ in  $\Delta_\mathrm{x}$. The function $\varphi_2$ cannot have any pole in $\Delta_{\mathrm{x}}$ as it would imply the existence of a pole for the Laplace transform $\phi_2$ in its region of convergence $\{x\in \mathbb C : \mathfrak{Re}(x)<0\}$. Since $s_1\neq s_2$ by Lemma~\ref{lemma:s1s2}, the only potential  zero  of $\varphi_1$  in $\Delta$ is $s_2\in\mathbb R$ (and it would be a simple zero). Note that by the description of $\Delta_\mathrm{x}$ given in Lemma~\ref{lemma:DxDy}, we have 
$$\mathbb{R}\cap\Delta_\mathrm{x}=(f^{-}(0),f^{+}(0))=(-\mu_2,0).$$
Hence the only possible zero $s_2\in \Delta_\mathrm{x}$ if and only if $-\mu_2 < s_2<0$, which never holds according to Lemma~\ref{lemma:s1s2}. The proof for $\varphi_2$ is similar.
\item
As already mentioned for $\phi_2$, the Laplace transform $\phi_1$ is analytic on $\{\mathfrak{Re}(y)<0\}$, hence $\varphi_1$ has no pole on $\Delta_\mathrm{y}$. Similarly, $\varphi_2$ has no pole on $\Delta_\mathrm{x}$. Then, by~\eqref{eq:phi1kpole} the only potential poles of $\varphi_1$ in $\Delta$ are the poles of $k_2(s)/k_1(s)$ in  $\Delta$. By~\eqref{eq:k1:k2} we see that the only such pole is $s_1$ and is simple (since $s_1\neq s_2$ by Lemma~\ref{lemma:s1s2}). We obtain, that $s_1$ is the only possible pole of $\varphi_1$ in $\Delta$ and is of order $1$. It is in $\overline{\mathfrak{B}}_2$ (\textit{resp} $\overline{\mathfrak{B}}_1$) if and only if $s_1\in \overline{\mathfrak{B}}_2\cap \Delta_\mathrm{x}=[s_-,0)$ (\textit{resp.} $s_1\in \overline{\mathfrak{B}}_1\cap \Delta=[-\mu_2,0)$). 
\item 
For $s=a+ib\in \mathfrak{B}_1$, when $|s|\to +\infty$, $a$ stays bounded (by definition of $\mathfrak{B}_1$, see~\eqref{eq3}), $b\to \pm \infty$ and by \eqref{eq:Reys}, 
$$\mathfrak{Re}(\y(s))=2a^2-4as_+-2b^2 \to -\infty.$$ 
Since the Laplace transfom tend to zero when the real part of its argument tends to $-\infty$,  $\varphi_1(s)$ must tends to $0$.
\end{enumerate}
\end{proof}

\section{Decoupling functions}\label{subsec:decoupl}

\subsection{Explicit decoupling and key parameters}
To solve the difference equation~\eqref{eq:qdiff} satisfied by $\varphi_1$,
\begin{equation*}
    \varphi_1(s+1)=G(s)\varphi_1(s),
\end{equation*}
we introduce the notion of decoupling function.
\begin{definition}[Decoupling function]\label{def:decoupling} Given a rational function $g\in\mathbb{C}(s)$, we say that a nonzero {meromorphic} function $d\in\mathbb{C}(s)$ is a \emph{decoupling function} of $g$ if 
\begin{equation}
{g(s)=\frac{d(s+1)}{d(s)}.}
\end{equation}
\end{definition}
We recall that in~\eqref{eq:Gdef} we defined $G$ as
$$G(s):=\frac{k_1(s)k_2(\zeta s)}{k_2(s)k_1(\zeta s)}=\frac{(s-s_1)(s+s_2+\mu_2)}{(s-s_2)(s+s_1+\mu_2)}.
$$
Therefore, we can easily see that $G$ always admits a meromorphic invariant $D$ defined by
\begin{equation}\label{eq:decouplageGamma}
D(s):=\frac{\Gamma(s-s_1)\,\Gamma(s+s_2+\mu_2)}{\Gamma(s-s_2)\,\Gamma(s+s_1+\mu_2)}
\end{equation}
where $\Gamma$ is the well-known gamma function. This function will play a crucial role in the following. We recall that $\Gamma$ satisfies $\Gamma(s+1)=s\Gamma(s)$ which directly leads to 
$$G(s)=\frac{D(s+1)}{D(s)}.$$
It is known that the function $\Gamma$ has poles on $-\mathbb{N}$. It is then possible that some poles of the numerator and the denominator of the decoupling function $D$ in~\eqref{eq:decouplageGamma} cancel each other out. To study this, we introduce new constants that will play a key role in the following:
\begin{equation}\label{eq:gammai}
    \gamma:=s_2-s_1,\quad \gamma_1:=\mu_1-2s_1,
    \quad\gamma_2:=\mu_2+2s_2.
\end{equation}
We will see bellow in Proposition~\ref{prop:decoupling} that when $\gamma \in \mathbb{Z}$, or $\{\gamma_1,\gamma_2\}\subset \mathbb Z$ the decoupling function $D$ is in fact rational. Replacing $s_1$ and $s_2$ by their definitions (see \eqref{eq:k2}), we get the formulas given in the introduction:
\begin{equation}\label{eq:gamma}
\gamma=\frac{1-r_1r_2}{(1+r_1)(1+r_2)}=\frac{1}{1+r_1}+\frac{1}{1+r_2}-1,\quad \gamma_1=\frac{1+\mu_2-r_1\mu_1}{1+r_1},\quad \gamma_2=\frac{1+\mu_1-r_2\mu_2}{1+r_2}.
\end{equation}
Moreover, one should observe that $\gamma$ does not depend on the drift $\mu$. Note that with~\eqref{eq:H4}, one can easily check the following relation between the three parameters 
\begin{equation}\label{eq4}
    \gamma_1+\gamma_2= 2\gamma+1.
\end{equation}

\begin{remark}[Common parity of $\gamma_1$ and $\gamma_2$]\label{rem:gamma} 
From the previous equation we obtain that if $\gamma\notin \mathbb Z$ and $\{\gamma_1,\gamma_2\}\subset \mathbb Z$ then $\gamma\in\frac{1}{2}+\mathbb{Z}$, \textit{i.e.}, $2\gamma$ is an odd integer. We deduce that in this case, $\gamma_1$ and $\gamma_2$ are either both odd or both even.
\end{remark}

\begin{lemma}[Impossible cases]\label{lemma:impossible} The existence and recurrence conditions~\eqref{eq:H1}-\eqref{eq:H2} prevent the following cases from happening:
\begin{enumerate}
\item $\gamma_1=0$ or $\gamma_2=0$,
\item $\gamma_1< 0$ and $\gamma_2< 0$,
\item $\gamma_1< 0$, $\gamma_2> 0$ and $\gamma>0$,
\item $\gamma_1> 0$, $\gamma_2< 0$ and $\gamma>0$,
\item $\gamma=0$.
\end{enumerate}
\end{lemma}

\begin{proof}
\begin{enumerate}
\item If $\gamma_1=0$ then by~\eqref{eq:gamma} $\mu_2-r_1\mu_1=-1$ which would contradict~\eqref{eq:H2}. Similarly, if $\gamma_2=0$, $\mu_1-r_2\mu_2=-1$ which would be again inconsistent with \eqref{eq:H2}.
\item Let us prove that either $r_1>-1$ or $r_2>-1$. Indeed, if $r_1,r_2\leq -1$, then we would neither have $\{r_1>0\text{ and }r_2>0\}$ (trivial), nor $1-r_1r_2>0$ (the product $r_1r_2$ is greater than $1$), which is the negation of~\eqref{eq:H1}. Suppose $r_1>-1$. Then by~\eqref{eq:gamma} the sign of $\gamma_1$ is the same as that of $1+\mu_2-r_1\mu_1$, which is positive by~\eqref{eq:H2}. If $r_2>-1$, the same reasoning shows that $\gamma_2>0$. Hence, at least one of the $\gamma_i$ is positive. 
\item By the same line of argument, if $\gamma_1<0$ and $\gamma_2>0$ then $r_1<-1$ and $r_2>-1$. We then distinguish two cases according to the sign of $1-r_1r_2$: if $1-r_1r_2<0$ then by~\eqref{eq:H1}, $r_1>0$ and $r_2>0$ which immediately contradicts $r_1<-1$. Otherwise, $1-r_1r_2>0$, combined with $r_1<-1$ and $r_2>-1$ in~\eqref{eq:gammai} yields $\gamma<0$.
\item The proof of this fourth item is similar to that of the third.
\item $\gamma=0$ if and only if $s_1=s_2$ which is impossible according to Lemma~\ref{lemma:s1s2}.
\end{enumerate}
\end{proof}

We now look at the decoupling function $D$ in the case $\gamma \in \mathbb{Z}$, or $\{\gamma_1,\gamma_2\}\subset \mathbb Z$. 

\begin{proposition}[Rationnality of $D$]\label{prop:decoupling}
The decoupling function $D$ defined in~\eqref{eq:decouplageGamma} satisfies
$$
D(s):=\frac{\Gamma(s-s_1)\,\Gamma(s+s_2+\mu_2)}{\Gamma(s-s_2)\,\Gamma(s+s_1+\mu_2)}
\quad\text{and}\quad
G(s)=\frac{D(s+1)}{D(s)}
. 
$$
When
\begin{equation}\label{eq:decouplcond}
\gamma\in \mathbb Z\text{ or }\{\gamma_1,\gamma_2\}\subset \mathbb Z
\end{equation}
the decoupling function $D$ is a rational function, \textit{i.e.} $D\in\mathbb{C}(s)$.
In this case, there exist polynomials $P(y)$ and $Q(y)$ and an integer $\varepsilon\in\{-1,0,1\}$ such that $D$ is given by the following degree $2\gamma$ rational function  

\begin{equation}\label{eq:formedecoupl}
D(s)=2^{-\gamma}\frac{P(\y(s))}{Q(\y(s))}\left[\sqrt{2}\left(s-s_+\right)\right]^\varepsilon.
\end{equation}
More precisely,
\begin{enumerate}[label=$\mathrm{(\alph*)}$]
\item When $\gamma\in-\mathbb N$, one can choose $\varepsilon:=0$, $P(y):=1$ and
\begin{equation}
Q(y):=\prod_{k=0}^{-\gamma-1}\Big(y-\y(s_1-k)\Big) .
\end{equation}
\item When $\gamma \in \mathbb N$, one can choose $\varepsilon:=0$, $Q(y):=1$ and
\begin{equation}\label{eq:decouplageN}
P(y):=\prod_{k=1}^{\gamma}\Big(y-\y(s_1+k)\Big) .
\end{equation}
\item When $\gamma\notin\mathbb Z$, if $\gamma_1\in-2\mathbb N$ and $\gamma_2\in 2\mathbb N$ then one can choose $\varepsilon:=-1$,
\begin{equation}\label{eq:decouplage2}
P(y):=\prod_{k=0}^{\frac{\gamma_2}{2}-1}\Big(y-\y(s_2-k)\Big)\text{ and }Q(y):=\prod_{k=0}^{-\frac{\gamma_1}{2}-1}\Big(y-\y(s_1-k)\Big),
\end{equation}
and if $\gamma_1\in -2\mathbb N+1$ and $\gamma_2\in 2\mathbb N-1$ then one can choose $\varepsilon:=1$,
\begin{equation}
P(y):=\prod_{k=0}^{\frac{\gamma_2-1}{2}-1}\Big(y-\y(s_2-k)\Big) \text{ and }Q(y):=\prod_{k=0}^{-\frac{\gamma_1+1}{2}}\Big(y-\y(s_1-k)\Big).
\end{equation}
Whatever the common parity of $\gamma_1$ and $\gamma_2$ is, the resulting decoupling function is 
\begin{equation}\label{eq:resulting1}
D(s) = \frac{\prod_{k=0}^{\gamma_2-1}\Big(s- (s_2-k)\Big)}{\prod_{k=0}^{-\gamma_1}\Big(s-(s_1-k)\Big)}.
\end{equation}
\item When $\gamma\notin \mathbb Z$, if $\gamma_1\in 2\mathbb N$ and $\gamma_2\in-2\mathbb N$ then one can choose $\varepsilon:=1$,
\begin{equation}
P(y):=\prod_{k=1}^{\frac{\gamma_1}{2}-1}\Big(y-\y(s_1+k)\Big) \text{ and }Q(y):=\prod_{k=1}^{-\frac{\gamma_2}{2}} \Big(y-\y(s_2+k)\Big),
\end{equation}
and if $\gamma_1\in 2\mathbb N-1$ and $\gamma_2\in -2\mathbb N+1$ then one can choose $\varepsilon:=-1$,
\begin{equation}
P(y):=\prod_{k=1}^{\frac{\gamma_1-1}{2}}\Big(y-\y(s_1+k)\Big)\text{ and }Q(y):=\prod_{k=1}^{-\frac{\gamma_2+1}{2}}\Big(y-\y(s_2+k)\Big).
\end{equation}
Whatever the common parity of $\gamma_1$ and $\gamma_2$ is, the resulting decoupling function is 
\begin{equation}\label{eq:resulting2}
D(s)=\frac{\prod_{k=1}^{\gamma_1-1}\Big(s-(s_1+k)\Big)}{\prod_{k=1}^{-\gamma_2}\Big(s-(s_2+k)\Big)}.
\end{equation}
\item When $\gamma\notin \mathbb Z$, $\gamma_1\in \mathbb N$ and $\gamma_2\in \mathbb N$, one can choose $\varepsilon:=1$, $Q(y)=1$ and
\begin{equation}
P(y):=\prod_{k=1}^{\lfloor\frac{\gamma_1-1}{2}\rfloor} \Big(y-\y(s_1+k)\Big)\prod_{k=0}^{\lfloor\frac{\gamma_2}{2}-1\rfloor} \Big(y-\y(s_2-k)\Big) .
\end{equation}
The resulting decoupling function is 
\begin{equation}\label{eq:resulting3}
D(s)=\prod_{k=1}^{\gamma_1-1}\Big(s-(s_1+k)\Big)\prod_{k=0}^{\gamma_2-1}\Big(s-(s_2-k)\Big).
\end{equation}
\end{enumerate}
\end{proposition}
\noindent According to Lemma~\ref{lemma:impossible}, if $\gamma\in\mathbb Z$ or $\{\gamma_1,\gamma_2\}\subset \mathbb Z$, then one of the conditions (a)-(e) in the above proposition is satisfied.

\begin{proof}
The functional relation $\Gamma(x+1)=x\Gamma(x)$ can be easily generalized to 
\begin{equation}\label{eq:Gammarec}
\frac{\Gamma(x+n)}{\Gamma(x)}=\prod_{k=0}^{n-1}(x+k)
\end{equation}
for any $n\in\mathbb N$. This observation will be central to the proof of the proposition.
\begin{itemize}
\item[(a)] If $\gamma\in-\mathbb N$, then we replace $s_2$ by $s_1+\gamma$ in the definition of $D$, yielding
$$
\frac{1}{D(s)}=\frac{\Gamma(s-s_1-\gamma)\,\Gamma(s+s_2-\gamma+\mu_2)}{\Gamma(s-s_1)\,\Gamma(s+s_2+\mu_2)}=\prod_{k=0}^{-\gamma-1}\Big(s-(s_1-k)\Big)\prod_{k=0}^{-\gamma-1}\Big(s-(-s_2-\mu_2-k)\Big).
$$
Reindexing the second product by setting $j=-\gamma-k-1$, we obtain
$$\prod_{k=0}^{-\gamma-1}\Big(s-(-s_2-\mu_2-k)\Big)=\prod_{j=0}^{-\gamma-1}\Big(s-(-s_1+\mu_1+j)\Big)=\prod_{j=0}^{-\gamma-1}\Big(s-\eta(s_1-j)\Big),$$
hence since by \eqref{eq:unif}, $\y(s):=2s(s-\mu_1)$,
$$\frac{1}{D(s)}=\prod_{k=0}^{-\gamma-1}\Big(s-(s_1-k)\Big)\Big(s-\eta(s_1-k)\Big)=2^\gamma\prod_{k=0}^{-\gamma-1}\Big(\y(s)-\y(s_1-k)\Big).$$
\item[(b)] Assume that  $\gamma >0$. Let $\widetilde{G}$ be the function $G$ where we have switched $s_1 \leftrightarrow s_2$. Then, $\widetilde{G}^{-1}(s)=G(s)$.
Note that if we switch  $s_1\leftrightarrow s_2$, then $s_2 -\gamma=s_1$ has to be rewritten as $s_1 +\gamma=s_2$. Hence, the same reasoning as above (with interchange of $s_1 \leftrightarrow s_2$, $\gamma\leftrightarrow  -\gamma$ and  $k \leftrightarrow -k$) proves the following expression for the decoupling function
 $$D(s)=P(\y(s))=2^{-\gamma}\prod_{k=1}^{\gamma}\Big(\y(s)-\y(s_1 +k)\Big).$$
\item[(c)] If $\gamma\notin\mathbb Z$, $\gamma_1\in-\mathbb N$ and $\gamma_2\in\mathbb N$, then we replace $s_1+\mu_2$ and $s_2+\mu_2$ by $-s_1-\gamma_1+1$ and $-s_2-\gamma_2$ respectively to get
$$D(s)=\frac{\Gamma(s-s_1)\,\Gamma(s-s_2+\gamma_2)}{\Gamma(s-s_1-\gamma_1+1)\,\Gamma(s-s_2)}=\prod_{k=0}^{-\gamma_1}\Big(s-(s_1-k)\Big)^{-1}\prod_{k=0}^{\gamma_2-1}\Big(s-(s_2-k)\Big).$$
Now we distinguish two cases according to the parity of $\gamma_1$. If $\gamma_1$ is even (and so is $\gamma_2$), then we split the first product into three parts, and second into two:
\begin{align*}
\prod_{k=0}^{-\gamma_1}\Big(s-(s_1-k)\Big)&=(s-s_+)\prod_{k=0}^{-\frac{\gamma_1}{2}-1}\Big(s-(s_1-k)\Big)\prod_{k=-\frac{\gamma_1}{2}+1}^{-\gamma_1}\Big(s-(s_1-k)\Big),\\
\prod_{k=0}^{\gamma_2-1}\Big(s-(s_2-k)\Big)&=\prod_{k=0}^{\frac{\gamma_2}{2}-1}\Big(s-(s_2-k)\Big)\prod_{k=\frac{\gamma_2}{2}}^{\gamma_2-1}\Big(s-(s_2-k)\Big).
\end{align*}
Now, observe that $s_1-k=\eta\Big(s_1-(-\gamma_1-k)\Big)$ and $s_2-k=\eta\Big(s_2-(\gamma_2-k-1)\Big)$, so that we can rewrite both products in the following way
\begin{align*}
\prod_{k=0}^{-\gamma_1}\Big(s-(s_1-k)\Big)&=(s-s_+)\prod_{k=0}^{-\frac{\gamma_1}{2}-1}\Big(s-(s_1-k)\Big)\Big(s-\eta(s_1-k)\Big),\\
\prod_{k=0}^{\gamma_2-1}\Big(s-(s_2-k)\Big)&=\prod_{k=0}^{\frac{\gamma_2}{2}-1}\Big(s-(s_2-k)\Big)\Big(s-\eta(s_2-k)\Big),
\end{align*}
and finally
$$D(s)=2^{-\gamma-\frac{1}{2}}(s-s_+)^{-1}\prod_{k=0}^{-\frac{\gamma_1}{2}-1}\Big(\y(s)-\y(s_1-k)\Big)^{-1}\prod_{k=0}^{\frac{\gamma_2}{2}-1}\Big(\y(s)-\y(s_2-k)\Big)$$
which match the claimed result in this case. If $\gamma_1$ is odd, then the exact same argument applies, but this time $s-s_+$ will be factored out of the second product. 
\item[(d)] Assume that $\gamma_1$ is positive and $\gamma_2$ is negative. If we switch $s_1$ and $s_2$ then the quantities $\gamma_1=\mu_1-2s_1$ and $\gamma_2=\mu_2+2s_2$ are replaced respectively by $\mu_1-2s_2=1-\mu_2-2s_2=1-\gamma_2$ and $\mu_2+2s_1=1-\mu_1+2s_1=1-\gamma_1$. Then the new quantities satisfies $1-\gamma_2\in \mathbb{N}$ and $1-\gamma_1\in -\mathbb{N}$. Therefore the switch $(s_1,s_2)\leftrightarrow (s_2,s_1)$ induces the switches $(\gamma_1,\gamma_2) \leftrightarrow (1-\gamma_2,1-\gamma_1)$ and $D\leftrightarrow D^{-1}$ and allows us to reduce to the preceding case.
\item[(e)] If both $\gamma_1$ and $\gamma_2$ are positive integers, then 
$$\frac{\Gamma(s-s_1)}{\Gamma(s+s_1+\mu_2)}=\frac{\Gamma(s-s_1)}{\Gamma(s-(s_1+\gamma_1-1))}=\prod_{k=1}^{\gamma_1-1}\Big(s-(s_1+k)\Big)$$
and
$$\frac{\Gamma(s+s_2+\mu_2)}{\Gamma(s-s_2)}=\frac{\Gamma(s-(s_2-\gamma_2))}{\Gamma(s-s_2)}=\prod_{k=0}^{\gamma_2-1}\Big(s-(s_2-k)\Big),$$
hence
$$D(s)=\frac{\Gamma(s-s_1)}{\Gamma(s+s_1+\mu_2)}\cdot \frac{\Gamma(s+s_2+\mu_2)}{\Gamma(s-s_2)}=\prod_{k=1}^{\gamma_1-1}\Big(s-(s_1+k)\Big)\prod_{k=0}^{\gamma_2-1}\Big(s-(s_2-k)\Big).$$
When $ \gamma_1 $ is even, the first product has an even number of factors, which we can group in pairs --- as in the preceding cases --- to express it in terms of $ \y(s) $. Although the second product then has an odd number of factors, we can isolate the central one (namely $ s - s_+ $) and pair the rest. When $ \gamma_1 $ is odd, the roles are reversed, and the same argument applies.
\end{itemize}
\end{proof}

\begin{remark}[Geometric condition for rational decoupling]
One can interpret the decoupling conditions~\eqref{eq:decouplcond} in the following way:
\begin{itemize}
\item if $\gamma\in\mathbb Z$, we can apply several times the identity $\eta\circ \zeta (s)=s+1$
to get
$$s_2=(\eta\circ \zeta)^\gamma(s_1) = s_1+\gamma .$$
This equality can be visualized on the parabola $\mathcal{S}\cap\mathbb R^2$ of Figure~\ref{fig:geomint2} (see Equation~\eqref{eq:parabola}) meaning that $s_1$ and $s_2$ are in the same orbit under the action of $\langle\eta\circ\zeta\rangle$. The parameter $\gamma$ being the number of step between the points $s_1$ and $s_2$, and
the dychotomy $\gamma \in \mathbb{N}$ and $-\gamma\in \mathbb{N}$ is the condition that describe the relative positions of the two points.
\item if $\{\gamma_1,\gamma_2\}\subset \mathbb Z$ then when $\gamma_1$ and $\gamma_2$ are both even we have 
$$s_+=(\eta\circ \zeta)^{\frac{\gamma_1}{2}}(s_1)=s_1+\frac{\gamma_1}{2}\text{ and }s_-=(\zeta\circ \eta)^{\frac{\gamma_2}{2}}(s_2)=s_2-\frac{\gamma_2}{2}$$
and when they are both odd we have 

$$s_-=(\eta\circ \zeta)^{\frac{\gamma_1-1}{2}}(s_1)=s_1+\frac{\gamma_1-1}{2}\text{ and }s_+=(\zeta\circ \eta)^{\frac{\gamma_2-1}{2}}(s_2)=s_2-\frac{\gamma_2-1}{2} .$$
This condition characterizes the relative positions of $s_1$ and $s_2$ with $s_-$ and $s_+$ and can also be visualized geometrically on the parabola of Figure~\ref{fig:geomint2}.
\end{itemize} 
\end{remark}
According to Proposition~\ref{prop:decoupling}, when $\gamma \in \mathbb{Z}$ or $\{ \gamma_1, \gamma_2 \} \subset \mathbb{Z}$, the function $D$ is a rational function of degree $2\gamma$ which behaves asymptotically when $s\to\infty$ as $s^{2\gamma}$. The following lemma establishes that this asymptotic holds in all cases.
\begin{lemma}[Asymptotic behavior of $D$]\label{lemma:asymptotic} The following asymptotic holds
\begin{equation}\label{eq:asymptotic}
D(s)\sim s^{2\gamma},\text{ when }|s|\to+\infty\text{ with }|\mathrm{arg}(s)|<\pi.
\end{equation}
\end{lemma}
\begin{proof}
Let us consider a ratio of two Gamma functions of the form $\Gamma(z+a)/\Gamma(z+b)$. The full asymptotic for such a ratio can be found in~\cite{ratiogamma}. Here, we only derive it up to the first order term. Recall the Stirling approximation $\Gamma(z)\sim \sqrt{2\pi}z^{z-1/2}e^{-z}$ (which holds for $|\text{arg}(z)|<\pi$) and apply it to both the numerator and the denominator:
$$\frac{\Gamma(z+a)}{\Gamma(z+b)}\sim e^{b-a}\left(\frac{z+a}{z+b}\right)^{z+a-1/2}(z+b)^{a-b}.$$
To analyse the second factor, we consider its logarithm:
$$\ln\left(\frac{z+a}{z+b}\right)=\ln\left(1+\frac{a-b}{z+b}\right)=\frac{a-b}{z}+o\left(\frac{1}{z}\right),$$
hence
\begin{align*}
\left(\frac{z+a}{z+b}\right)^{z+a-1/2}&=\exp\left(a-b+\frac{(a-1/2)(a-b)}{z}+o(1)\right)\\
&=e^{a-b}\left(1+\frac{(a-1/2)(a-b)}{z}+o\left(\frac1z\right)\right).
\end{align*}
Plugging it back in the Stirling approximation of the ratio yields
\begin{equation}\label{eq:ratiogamma}
\frac{\Gamma(z+a)}{\Gamma(z+b)}\sim z^{a-b}.
\end{equation}
The function $D$ is the product of two such ratios, hence 
$$D(s):=\frac{\Gamma(s-s_1)\Gamma(s+s_2+\mu_2)}{\Gamma(s-s_2)\Gamma(s+s_1+\mu_2)}\sim s^{-s_1+s_2}s^{s_2-s_1}=s^{2(s_2-s_1)}=s^{2\gamma}.$$
\end{proof}

\begin{center}
\begin{figure}[h]
\centering
\includegraphics[width=15.3cm]{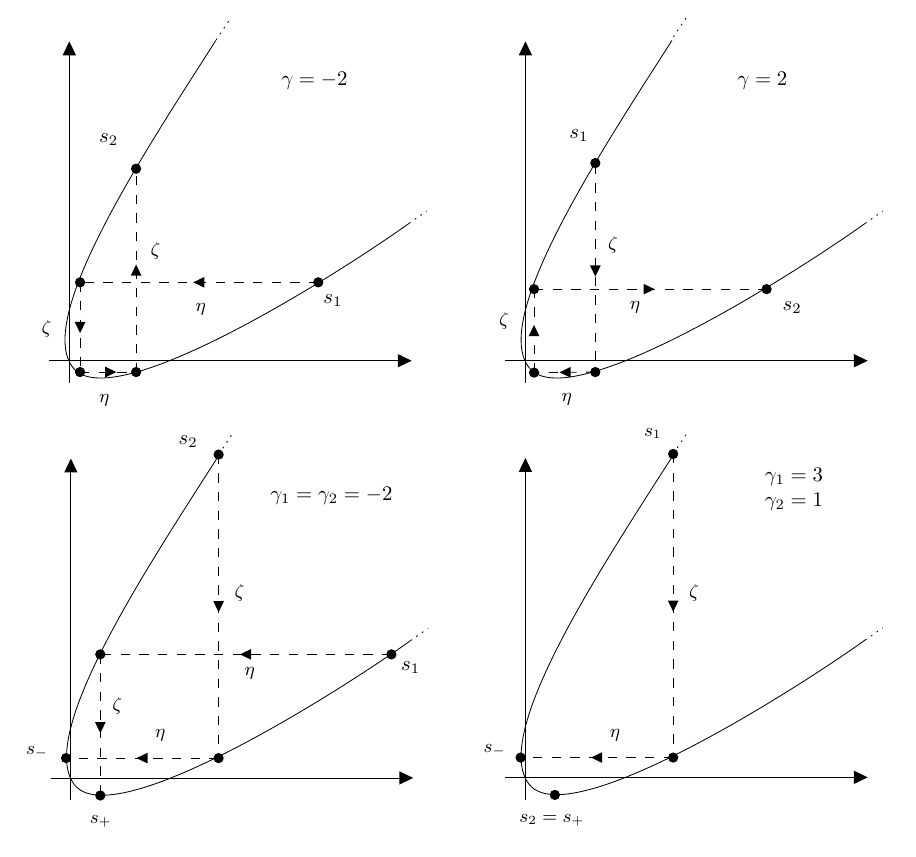}
\caption{Geometric interpretation for the rational decoupling condition~\eqref{eq:decouplcond}. Every point on the parabola $K(x,y)=0$ is labelled with its unique preimage by $(\x,\y)$, \textit{i.e.} $(\x(s),\y(s))$ is labelled by $s$.}
\label{fig:geomint2}
\end{figure}
\end{center}

\subsection{Necessary and sufficient condition for rational decoupling}
We have already seen in Proposition~\ref{prop:decoupling} that $\gamma \in \mathbb{Z}$, or $\{\gamma_1,\gamma_2\}\subset \mathbb Z$ is a sufficient condition to have a rational decoupling. In this section we show that it is also a necessary condition. A necessary and sufficient conditions for $g\in \mathbb{C}(s)$ to have nonzero rational decoupling function $d\in \mathbb{C}(s)$ can be found in the reference book on Galois Theory of Difference Equations by Put and Singer \cite[Section 2.1]{vdPS}. Let us make a quick overview. We first introduce the notion of divisor.

\begin{definition}[Divisor]\label{def:div}
The divisor $\operatorname{div}(g)$ of a rational function $g \in \mathbb{C}(s)^*$ is given by the finite formal sum
\[
\operatorname{div}(g) := \sum_a \operatorname{ord}_a(g)[a]
\]
where $\operatorname{ord}_a(f)$ denotes the order of $f$ at point~$a$, \textit{i.e.}, for $a\in\mathbb{C}$ it is the integer $n\in \mathbb{Z}$ such that $(s-a)^{-n}g(s)$ is holomorphic and nonzero at $a$, and for $a=\infty$ it is the interger $n\in \mathbb{Z}$ such that $s^{n}g(s)$ has a nonzero finite limit at $\infty$. The sum that defines the divisor is taken over all the $a \in \mathbb{C} \cup \{\infty\}$ and is in fact a finite sum since it involves a finite number of terms such that $\operatorname{ord}_a(g)\neq 0$. We call this finite set of terms the support of the divisor.
\end{definition}

\begin{example}[Typical divisor]
If $g(s)=\displaystyle \prod_{i=1}^{k}(s-a_i)^{n_i}\in \mathbb{C}(s)$ with $a_i\in \mathbb{C}$ and $n_i \in \mathbb{Z}$, then $$\mathrm{div}(g)=\sum_{i=1}^{k} n_i [a_i] -\left(\sum_j n_j \right)[\infty].$$
\end{example}

Let us now give a necessary and sufficient condition for the existence of a decoupling function. 
The following lemma is taken from Lemma 2.1 of the book by Put and Singer \cite{vdPS}. Note that the assumption $\infty$ is not in the support of $\mathrm{div}(g)$ in \cite{vdPS} is not necessary since $\lim_{s\to \infty }g(s)=1$ implies directly the latter assumption.
\begin{lemma}[Rational decoupling criterion \cite{vdPS}]\label{lem1}
Let $g\in \mathbb{C}(s)^{*}$, then there exists a rational decoupling function of $g$ if and only if the following properties hold:
\begin{enumerate}
\item The limit of $g(s)$ when $s\to \infty$ is $1$.
\item For every $\mathbb{Z}$-orbit $E$, \textit{i.e.} every subset of $\mathbb{C}$ of the form $e + \mathbb{Z}$ with $e\in \mathbb{C}$, one has $$\sum_{a\in E} \mathrm{ord}_{a}(g)=0.$$
\end{enumerate}
\end{lemma}
For example, $g(s)=\frac{s-1}{s-1/2}$ admits no decoupling function since it satisfies the first property but not the second.
Let us now apply Lemma \ref{lem1} to our problem.

\begin{proposition}[Necessary and sufficient rational decoupling condition of $G$]\label{prop1}
The function $G$ admits a rational decoupling function if and only if 
$$\gamma \in \mathbb{Z}
\text{ or }
\{\gamma_1,\gamma_2\}\subset \mathbb Z.$$
\end{proposition}

\begin{proof} We recall that in~\eqref{eq:Gdef} we defined $G$ as
$$G(s):=\frac{k_1(s)k_2(\zeta s)}{k_2(s)k_1(\zeta s)}=\frac{(s-s_1)(s+s_2+\mu_2)}{(s-s_2)(s+s_1+\mu_2)}.
$$
The divisor of $G$ is then
$$\mathrm{div}(G)=[s_1]+[-s_2-\mu_2]-[s_2]-[-s_1-\mu_2].$$ 
Note that we do not claim that the sum involves four terms. For instance, if $s_1=-s_2-\mu_2$ then we have a double zero and the divisor is  $\mathrm{div}(G)=2[s_1]-[s_2]-[-s_1-\mu_2]$.
The first condition in Lemma \ref{lem1} is automatically satisfied. We just have to check the second one. For $a,b\in \mathbb{C}$, we will say that $a\sim b$ if and only if $b-a\in \mathbb{Z}$. Let $\mathcal{Z}=\{s_1,-s_2-\mu_2 \}$ be the set of zeros of $G$ and $\mathcal{P}=\{s_2,-s_1-\mu_2 \}$ be the set of its poles. If Condition 2 occurs, for all $x\in \mathcal{Z}$ there exists $y\in \mathcal{Z}$ with $x\sim y$ and for all   $y\in \mathcal{Z}$ there exists $x\in \mathcal{P}$ such that $x\sim y$. Since $\sim$ is an equivalent relation we find that if condition 2 occurs then there is a bijection $h$ from  $\mathcal{Z}$ to $\mathcal{P}$ with for all $x\in \mathcal{Z}$, $x\sim h(x)$. Conversely, the existence of such bijection implies that the second condition holds. Then, a rational decoupling function exists if and only  if one of the two facts holds 
\begin{itemize}
\item[(a)] $s_1\sim s_2$ and $-s_1-\mu_2\sim -s_2-\mu_2$.
\item[(b)] $s_1 \sim -s_1-\mu_2 $ and $s_2\sim -s_2-\mu_2$.
\end{itemize} 
We recall that in~\eqref{eq:gammai} we defined $$\gamma=s_2-s_1,\quad \gamma_1=\mu_1-2s_1,%=2s_1+\mu_2-1,
\quad\gamma_2=\mu_2+2s_2.$$ 
We then find 
that 
$$\text{(a)} \Leftrightarrow s_1\sim s_2 \Leftrightarrow -s_1-\mu_2\sim -s_2-\mu_2 \Leftrightarrow \gamma \in \mathbb{Z}$$ 
and remembering that $\mu_1+\mu_2=1$ we find that
$$
\text{(b)} \Leftrightarrow
\begin{cases}
s_1 \sim -s_1-\mu_2 \Leftrightarrow s_1 \sim -s_1+\mu_1 -1 \Leftrightarrow  \gamma_1 \in \mathbb{Z},
\\
s_2\sim -s_2-\mu_2 \Leftrightarrow \gamma_2 \in \mathbb{Z}.
\end{cases}
$$
Then, condition 2 occurs if and only if $\gamma \in \mathbb{Z}$, or $\{\gamma_1,\gamma_2\}\subset \mathbb Z$.
\end{proof}

\section{Tutte's invariants and conformal gluing function}\label{subsec:Tutte}

\subsection{Type I invariant}
\label{subsec:per}
The core concept in Tutte’s method is the notion of invariant, defined (in this context) as follow. 
\begin{definition}[Type I invariant]\label{def:1per} A function $I_1$ that is meromorphic on $\mathbb{C}$ and satisfies the following invariance
\begin{equation}\label{eq:inv1per}
I_1(s+1)=I_1(s)\text{ for all } s\in \mathbb{C}
\end{equation}
will be called \emph{Type I invariant} or \emph{1-periodic}.
\end{definition}
Recall that we defined in~\eqref{eq2} the lines $d_1$ and $d_2$ by
$$
d_{1}=\mu_1+i\mathbb{R}\quad\text{and}\quad d_{2}=-\mu_2+i\mathbb{R}
$$
and in~\eqref{eq3} we defined the strip $$\mathfrak{B}_1=\{s\in\mathbb C : -\mu_2<\mathfrak{Re}(s)<\mu_1\}$$ which is bounded by $d_1$ and $d_2$. Note that $d_1=d_2+1$ since $\mu_1+\mu_2=1$. We say that $\mathfrak{B}_1$ is the fundamental domain of a Type I invariant. Let us introduce the function
\begin{equation}\label{eq:w1}
\wb:\left\{\begin{array}{rcl}
  \mathbb C & \longrightarrow & \mathbb C\\
  s & \longmapsto & \tan\big(\pi(s+\mu_2-\frac{1}{2})\big) .
\end{array}\right.
\end{equation}
The following lemma establishes that $\wb$ is an $1$-periodic function and is conformal if we restrict it to the fundamental domain, see Figure~\ref{fig:w}. 

\begin{lemma}[Periodic conformal gluing function]\label{lemma:w1} The function $\wb$ is a conformal gluing function in the following sense:
\begin{enumerate}
    \item $\wb$ establishes a biholomorphism between $\mathfrak{B}_1$ and the cut plane $\mathbb C\setminus \Big\{(-i\infty,-i]\cup [i,+i\infty)\Big\}$ whose inverse is given by 
    \begin{equation}\label{eq:w1-1}
    \wb^{-1}(\omega):=\frac{1}{2}-\mu_2+\frac{1}{\pi}\arctan(\omega)=\frac{1}{2}-\mu_2+\frac{1}{2i\pi}\ln\left( \frac{1+is}{1-is} \right),
    \end{equation}
    \item $\wb$ is a Type I invariant and
    \begin{equation}\label{eq:w1d1d2}
    \wb(d_1)=\wb(d_2)=(-i\infty,-i]\cup [i,+i\infty).
    \end{equation}
\end{enumerate}
\end{lemma}
\begin{proof} It is a well-known result that the tangent function establishes a biholomorphism between the strip $\left(-\frac{\pi}{2},\frac{\pi}{2}\right)+i\mathbb R$ and the cut plane $\mathbb C\setminus \Big\{(-i\infty,-i]\cup [i,+i\infty)\Big\}$. The inverse is the principal branch of the complex $\arctan$ defined as
$$\arctan(s)=\frac{1}{2i}\ln\left( \frac{1+is}{1-is} \right),$$
where $\ln$ denotes the principal branch of the complex logarithm.
A simple change of variables gives the formula $\wb^{-1}$.
The function $\wb$, which is also defined as a meromorphic function on the whole complex plane, is $1$-periodic
since the tangent function is $\pi$-periodic.

Let us now prove~\eqref{eq:w1d1d2}. Recall from the definition of the tangent function that 
$$\tan\left( \pm\frac{\pi}{2} + iy \right) = i \, \text{coth}(y) \in (-i\infty,-i]\cup [i,+i\infty).$$
Now observe that $d_1 = d_2 +1$ so that their images under $\wb$ coincide and we have:
$$\wb(d_1) = \wb(d_2) = \tan \left(\pm \frac{\pi}{2}+i\mathbb{R}\right)= i\coth \mathbb{R} =(-i\infty,-i]\cup [i,+i\infty).$$
\end{proof}

\begin{figure}
\centering 
\includegraphics[width=13cm]{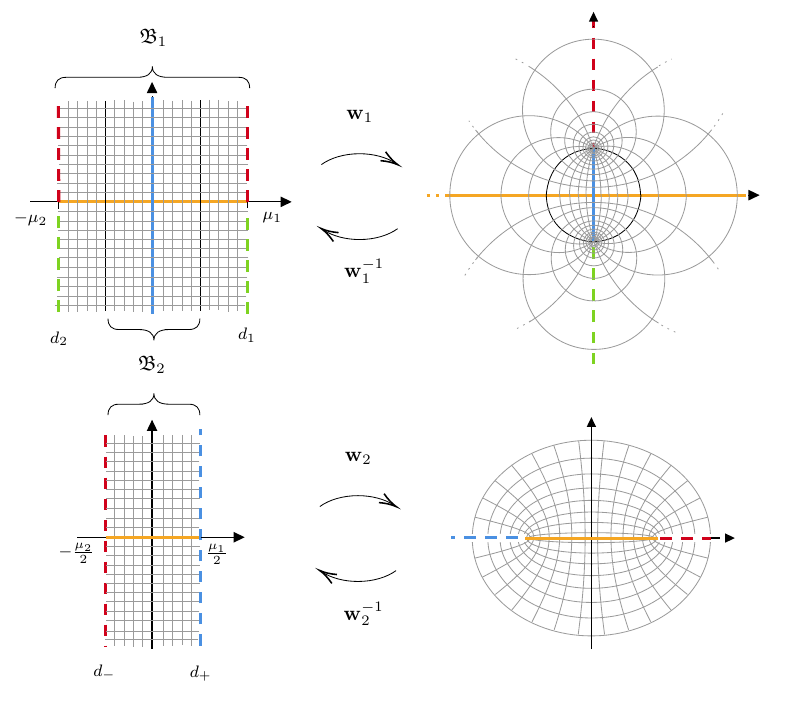}
\caption{Domains and codomains of the conformal gluing functions $\boldsymbol{\mathrm{w}}_1$ and $\boldsymbol{\mathrm{w}}_2$. Subsets are shown in matching colors with their images.}\label{fig:w}
\end{figure}

We now state the first \emph{invariant lemma} which is the main result of this section and which is a key stone of the proof of the main theorems. This lemma states that under some mild conditions on its growth and its poles, a Type I invariant is a rational fraction of $\wb$. This is why the invariant $\wb$ is sometimes called in the theory of Tutte's invariant, a \emph{canonical invariant}.

\begin{proposition}[Type I invariant lemma]\label{lemma:periodic}\
\begin{enumerate}
\item
If $I_1$ is a Type I invariant satisfying the following properties:
\begin{itemize}
\item $I_1$ is holomorphic on a neighborhood of $\overline{\mathfrak{B}}_1$,
\item $I_1$ grows at most polynomially at infinity in $ \overline{\mathfrak{B}}_1$ (i.e when $|z|\to\infty$ and $z\in \overline{\mathfrak{B}}_1$, $|I_1|$ is bounded by the modulus of a polynomial),
\end{itemize}
then $I_1$ is constant. 
\item More generally, if $I_1$ is an $1$-invariant satisfying the following properties:
\begin{itemize}
\item $I_1$ has finitely many poles in $\overline{\mathfrak{B}}_1$ which all belongs to $\mathfrak{B}_1\cup\{-\mu_2,\mu_1 \}$,
\item $I_1$ grows at most polynomially at infinity in $ \overline{\mathfrak{B}}_1$,
\end{itemize} 
then $I_1$ is a rational fraction in $\boldsymbol{\mathrm{w}}_1$, \textit{i.e.} there exists $F_1\in \mathbb C(X)$ such that $I_1=F_1(\boldsymbol{\mathrm{w}}_1)$. 

More precisely, let $p_1,\dots,p_n\in\mathfrak{B}_1$ the poles of $I$ with multiplicities $m_1,\dots,m_n$ and let $m_{\mu}$ the order of the (possible) pole at $-\mu_2$ and $\mu_1$ (note that the order at these two points is always the same since the function is $1$-periodic and $\mu_1+\mu_2=1$). By convention the multiplicity is 0 when there is no pole. Then there exist constants $c_{i,j}$ and $c_{\mu,j}$ such that
\begin{equation}
I_1=F_1\circ \wb=c+\sum_{i=1}^{n}\sum_{j=1}^{m_i}\frac{c_{i,j}}{\left(\wb-\wb(p_i)\right)^j} +\sum_{j=1}^{m_{\mu}}c_{\mu,j}\wb^j.
\label{eq:formuleF1}
\end{equation}
\end{enumerate}
\end{proposition}
\begin{proof}
\begin{enumerate}
\item Assume that $I_1$ is holomorphic on a neighborhood of $\overline{\mathfrak{B}}_1$, and consider $g:=I_1\circ \wb^{-1}$ which by Lemma~\ref{lemma:w1} is \textit{a priori} defined and holomorphic on the slit plane 
\begin{equation}
\wb(\mathfrak{B}_1)=\mathbb C\setminus \Big\{(-i\infty,-i]\cup [i,+i\infty)\Big\}.
\end{equation}
\begin{itemize}
\item
First, we show that $g$ is holomorphic on $\mathbb{C}$. By composition of holomorphic functions, $g$ is holomorphic on $\wb(\mathfrak{B}_1)$. Let us show that $g$ can be extended by continuity on the cut $\wb(\partial \mathfrak{B}_1)=(-i\infty,-i]\cup[i,+i\infty)$.

Let $\omega_0\in (-i\infty,-i)\cup(i,+i\infty)$. We wish to show $g$ extends continuously at $\omega_0$.  For $\omega$ near $\omega_0$ but off the cut, write
\[
s=\wb^{-1}(\omega)\in\mathfrak{B}_1.
\]
As $\omega\to\omega_0$, there are two ways to approach the cut:
from the left  ($\mathfrak{Re}(\omega)<0$), in which case the principal inverse $\wb^{-1}$ gives a limit
\[
s_0 \;=\;\lim_{\substack{\omega\to\omega_0\\\mathfrak{Re}(\omega)<0}} \wb^{-1}(\omega)\,\in\,d_2,
\]
from the right ($\mathfrak{Re}(\omega)>0$), in which case the analytic continuation of $\wb^{-1}$ across the branch cut differs by $+1$, namely
\[
s_0+1 \;=\;\lim_{\substack{\omega\to\omega_0\\\mathfrak{Re}(\omega)>0}} \wb^{-1}(\omega)\in d_1 ,
\]
because we saw that $\wb(d_1)=\wb(d_2)$. Hence
\[
\lim_{\substack{\omega\to\omega_0\\\mathfrak{Re}(\omega)<0}}
 I_1\bigl(\wb^{-1}(\omega)\bigr)
=
I_1(s_0)
=
I_1(s_0+1)
=
\lim_{\substack{\omega\to\omega_0\\\mathfrak{Re}(\omega)>0}}
 I_1\bigl(\wb^{-1}(\omega)\bigr),
\]
where we have used $1$-periodicity of $I_1$. This shows the two one-sided limits coincide:
$$\lim_{\substack{\omega\to\omega_0\\\mathfrak{Re}(\omega)<0}} g(\omega)
=
\lim_{\substack{\omega\to\omega_0\\\mathfrak{Re}(\omega)<0}}
 I_1\bigl(\wb^{-1}(\omega)\bigr)=
\lim_{\substack{\omega\to\omega_0\\\mathfrak{Re}(\omega)>0}}
 I_1\bigl(\wb^{-1}(\omega)\bigr)
 =
\lim_{\substack{\omega\to\omega_0\\\mathfrak{Re}(\omega)>0}}\! g(\omega)$$
 so $g$ extends continuously at~$\omega_0\in (-i\infty,-i)\cup(i,+i\infty)$. It remains to check that $g(\omega)=I_1\bigl(\wb^{-1}(\omega)\bigr)$ extends continuously at the endpoints of the cut, $\pm i$.  By hypothesis, $I_1(s)$ grows at most polynomially as $|s|\to\infty$ in $\overline{\mathfrak B}_1$, and then $I_1(s)=o(s^n)$ for some $n>0$.
Near $\omega=i$, the local expansion of the inverse map gives 
\[
\wb^{-1}(\omega)
\; \underset{\omega\to i}{=}\;
\frac{1}{2\pi i}\,\ln(\omega-i)\;+\;O(1)
 =o\bigl(|\omega-i|^{-\alpha}\bigr)
\]
for every $\alpha>0$.
Hence as $\omega\to i$,
\[
g(\omega)
=\;
I_1\!\Bigl(\wb^{-1}(\omega)\Bigr)
\;=\;
o\bigl(|\omega-i|^{-n\alpha}\bigr)
\]
for every $\alpha>0$.
Thus, taking $\alpha$ small enough, we see that $(\omega-i)g(\omega)\to 0$ as $\omega\to i$ and $i$ is therefore a removable singularity by Riemann's theorem. So $g$ extends continuously (in fact holomorphically) at $\omega=i$. The same argument applies at $\omega=-i$.

In addition to being holomorphic on the slit plane, we proved that $g$ is continuous on the cut $(-i\infty,-i]\cup[i,+i\infty)$. It is therefore holomorphic on the whole complex plane (for more details, see Theorem 16.8 in~\cite{rudin}, whose proof relies on Morera's theorem). 

\item We now show that $g$ is bounded at infinity. For $\omega$ in the slit plane, when $|\omega|\to+\infty$, 
$\arctan(\omega)\rightarrow\pm\frac{\pi}{2}$ depending on the sign of $\mathfrak{Re}(\omega)$. Using the expression for $\wb^{-1}$ given in Lemma~\ref{lemma:w1} we obtain when $|\omega|\to+\infty$ in the slit plane
$$\wb^{-1}(\omega)\longrightarrow 
\begin{cases}
\mu_1 &\text{if } \mathfrak{Re} (\omega)>0
\\
-\mu_2 &\text{if } \mathfrak{Re} (\omega)<0
\end{cases} 
 $$
Since $g$ is $1$-periodic and $\mu_1+\mu_2=1$ we deduce that when $|\omega|\to+\infty$ we have
$$
I_1(\wb^{-1}(\omega)) \longrightarrow I_1(-\mu_2)=I_1(\mu_1)
$$
and $g$ is then bounded at infinity.
\end{itemize}
We have shown that $g$ is holomorphic on $\mathbb{C}$ and is bounded at infinity which implies, by Liouville's theorem, that $g$ is constant. 
We directly deduce that $I_1=g\circ \wb$ is constant.

\item We now assume that $I_1$ is meromorphic on a neighborhood of $\overline{\mathfrak{B}}_1$, with finitely many poles $p_1,\dots,p_n\in\mathfrak{B}_1$ (with multiplicities $m_1,\dots,m_n$) and possibly additional poles at $-\mu_2$ and $\mu_1$ of order $m_{\mu}$. By convention we take $m_\mu=0$ if there is no pole.

\begin{itemize}
\item First, we show that $I_1\circ \wb^{-1}$ is meromorphic on $\mathbb{C}$. Since $\wb^{-1}$ is analytic on $\wb(\mathfrak{B}_1)=\mathbb C\setminus \big\{ (-i\infty,-i]\cup [i,+i\infty)\big\}$, it follows that $I_1\circ \wb^{-1}$ is meromorphic on $\wb(\mathfrak{B}_1)$, with finitely many poles at the points $\wb(p_i)$. Biholomorphisms preserve the order of poles, being locally invertible holomorphic maps, hence $I_1\circ \wb^{-1}$ has a pole of order $m_i$ at $\wb(p_i)$.

Using the same reasoning as in the first part of the proof, we see that $I_1\circ \wb^{-1}$ is continuous on $(-i\infty,-i]\cup[i,+i\infty)$ and is therefore meromorphic on the whole of $\mathbb{C}$.
\item By removing the poles from $I_1\circ \wb^{-1}$, we now construct a function $g$ which is holomorphic on $\mathbb{C}$. We consider 
\begin{equation}
g(\omega):=I_1\circ \wb^{-1}(\omega)-\sum_{i=1}^n\sum_{j=1}^{m_i} \frac{c_{i,j}}{\left(\omega - \wb(p_i) \right)^j}
\label{eq:gI1w}
\end{equation}
for well-chosen constants $c_{i,j}$ such that $g$ has no pole on the complex plane $\mathbb C$ (though it may have a singularity at infinity, corresponding to potential poles of $I_1$ at $\mu_1$ and $-\mu_2$). Hence $g$ is holomorphic on $\mathbb{C}$. 
\item Finally, we show that $g$ grows as a polynomial at infinity. First,
as $\omega \to \infty$, the double sum of~\eqref{eq:gI1w} goes to $0$. Secondly, as $\omega \to \infty$,
$$\wb^{-1}(\omega)\sim c-\frac{1}{\pi \omega}+o\left(\frac{1}{\omega}\right),$$
where $c=\mu_1$ or $c=-\mu_2$ depending on the real part of $\omega$. Since $I_1$ has a pole of order $m_\mu$ at $c$ we deduce that $I_1\circ \wb^{-1}(\omega)$ grows as a polynomial of degree $m_\mu$ as $\omega\to\infty$. Hence, $g(\omega)$ grows at most like a polynomial of degree $m_\mu$ as $\omega\to\infty$.
\end{itemize} 
We have shown that $g$ is holomorphic on $\mathbb{C}$ and grows at infinity as a polynomial of degree $m_\mu$. The \textit{extended} Liouville theorem implies that $g$ is a polynomial of degree
$m_\mu$. We denote $c$ its constant coefficient and by $c_{\mu,j}$ its $j$-th coefficient. This completes the proof of the claimed formula, that is
\begin{equation*}
I_1=F_1\circ \wb\text{ where }F_1(\omega)=c+\sum_{i=1}^{n}\sum_{j=1}^{m_i}\frac{c_{i,j}}{\left(\omega-\wb(p_i)\right)^j} +\sum_{j=1}^{m_{\mu}}c_{\mu,j}\omega^j.
\end{equation*} 
\end{enumerate}
\end{proof}

\subsection{Type II invariant}
\label{subsec:inv}

After introducing the 1-periodic invariants in the previous section, that is, the functions invariant under the transformation \( \eta \circ \zeta(s) = s + 1 \), we now define in this section the functions that are invariant under both \( \eta \) and \( \zeta \). This section~\ref{subsec:inv} is built identically to section~\ref{subsec:per}.
\begin{definition}[Type II invariant]\label{def:inv} A function $I$ that is meromorphic on $\mathbb{C}$ and satisfies the following invariances
\begin{equation}\label{eq:invzetaeta}
I(s)=I(\zeta s)=I(\eta s)\text{ for all } s\in\mathbb{C}
\end{equation}
will be called a \emph{Type II invariant} or an \emph{$\langle \eta,\zeta\rangle$-invariant}. 
\end{definition}
Of course, since it is invariant by $\eta$ and $\zeta$, a Type II invariant is invariant by every element of the group they generate, that is $\langle \eta,\zeta\rangle$. And then, since \( \eta \circ \zeta(s) = s + 1 \), a Type II invariant is also a Type I invariant. Recall that $s_-$ (resp. $s_+$) is the fixed point of $\zeta$ (resp. $\eta$) and that we defined in~\eqref{eq2}  the lines $d_-$ and $d_+$ by
$$
d_{\pm}=s_{\pm}+i\mathbb{R}
$$
and in~\eqref{eq3} we defined the strip $$\mathfrak{B}_2=\{s\in\mathbb C : s_-<\mathfrak{Re}(s)<s_+\}$$ which is bounded by $d_-$ and $d_+$. Note that $\zeta (d_{-})=d_{-}$ and $\eta (d_{+})=d_{+}$. We say that $\mathfrak{B}_2$ is the fundamental domain of a Type II invariant.
Let us introduce the function
\begin{equation}\label{eq:w}
\w:\left\{\begin{array}{rcl}
  \mathbb C & \longrightarrow & \mathbb C\\
  s & \longmapsto & \cos\big(2\pi(s-s_-)\big) .
\end{array}\right.
\end{equation}
The following lemma establishes that $\w$ is an invariant and is conformal if we restrict it to the fundamental domain. We say that $\w$ is a canonical Type II invariant.

\begin{lemma}[Conformal gluing function]\label{lemma:w} The function $\w$ is a conformal gluing function in the following sense:
\begin{enumerate}
    \item $\w$ establishes a biholomorphism between $\mathfrak{B}_2$ and the cut plane $\mathbb C\setminus \Big\{(-\infty,-1]\cup [1,+\infty)\Big\}$ whose inverse is given by 
    \begin{equation}\label{eq:w-1}
    \w^{-1}(\omega):=s_-+\frac{1}{2i\pi}\ln\left(\omega+i\sqrt{1-\omega^2}\right).
    \end{equation}
    \item $\w$ is a Type II invariant and
    \begin{equation}\label{eq:wd1d2}
    \w(d_-)=[1,+\infty)\text{ and }\w(d_+)=(-\infty,-1].
    \end{equation}
\end{enumerate}
\end{lemma}

\begin{proof}
Let us prove the first point. It is a well established result, see e.g. Chapter 3 of \cite{sveshnikov_tikhonov_71}, that the cosine function establishes a biholomorphism between the strip $(0,\pi)+i\mathbb R$ and the cut plane $\mathbb C\setminus \Big\{(-\infty,-1]\cup [1,+\infty)\Big\}$. The inverse is the principal branch of the complex $\arccos$ defined as
$$\arccos(s)=\frac{1}{i}\ln \left(s+i\sqrt{1-s^2}\right),$$
where $\ln$ denotes the principal branch of the complex logarithm and $\sqrt{\,\cdot\,}$ the principal branch of the square root.
A simple change of variables gives the formula $\w^{-1}$ and we do the following calculation to check it. For all $s\in\mathfrak{B}_2$,
\begin{align*}
s_-+\frac{1}{2i\pi}\ln\left(\w(s)+i\sqrt{1-\w(s)^2}\right)&=s_-+\frac{1}{2i\pi}\ln\Big(\cos\big(2\pi(s-s_-)\big)+i\sin\big(2\pi(s-s_-)\big)\Big)\\
&=s_-+\frac{1}{2i\pi}\ln\left(e^{2i\pi (s-s_-)}\right)=s_-+\frac{2i\pi(s-s_-)}{2i\pi}=s,
\end{align*}
which yields the expression for $\w^{-1}$. 

Let us prove the second point. The function $\w$ which is also defined as a holomorphic function on the whole complex plane, satisfies
$$\w(\zeta s)=\w(-s+2s_-)%=\cos\left(2\pi (\zeta s -s_-)\right)
=\cos\left(-2\pi(s-s_-)\right)=\cos\left(2\pi(s-s_-)\right)=\w(s)$$
and remembering that $2(s_+-s_-)=\mu_1+\mu_2=1$,
$$\w(\eta s)=\cos\left(2\pi(-s+2s_+-s_-)\right)
%=\cos\left(2\pi(-s+s_-+2(s_+--s_-))\right)
=\cos\left(2\pi(-s+s_-+1)\right)
=\cos\left(2\pi(s-s_-)\right)=\w(s).$$
Let us now prove~\eqref{eq:wd1d2}. Recall from the definition of the cosine function that 
$$\cos(x+iy)=\cosh(y)\cos(x)-i\sinh(y)\sin(x),$$
so that $\cos (iy)=\cosh(y)$ and $\cos (\pi+iy)=-\cosh(y)$, proving that $\w(s_{-}+i \mathbb{R})=\cos(i\mathbb{R})=[1,+\infty)$ and  $\w(s_{+}+i \mathbb{R})=\cos(\pi+i\mathbb{R})=(-\infty,-1]$.
\end{proof}
Proposition~\ref{lemma:inv} below (called the \emph{Type II invariant lemma}) describes the invariants in terms of the conformal gluing function $\w$.
To state this result we first need a small lemma. Recall that any meromorphic function $f$ admits, at every point $s\in\mathbb C$, a unique Laurent series expansion of the form
$$f(s+z)=\sum_{n\in\mathbb Z}a_nz^n$$
with finitely many nonzero $a_{-n}$ for $n\in\mathbb N$, that holds for $z$ in a neighborhood of $0$.
\begin{lemma}[Invariance and Laurent series]\label{lemma:res} 
If a function $f$ is $\eta$-invariant (resp. $\zeta$-invariant) at a neighborhood of $s_+$ (resp. $s_-$) then the odd coefficients of its Laurent series at $s_+$ (resp. $s_-$) are equal to zero. In particular, if an invariant has a pole at $s_+$ or $s_-$ then this pole has even order.
\end{lemma}

\begin{proof}
Let $\sum _{n=-\infty}^{\infty }a_{n}(z-s_+)^{n}$ be the Laurent series of $f$ at the point $z=s_+$. Since $f$ is $\eta$-invariant, $f(s_++z)=f(s_+-z)$ and we obtain 
$$\sum _{n=-\infty}^{\infty }a_{n}z^{n}=f(s_++z)=f(s_+-z)=\sum _{n=-\infty}^{\infty }(-1)^n a_{n}z^{n}.$$ 
We deduce by uniqueness of the Laurent expansion that $a_{n}=(-1)^n a_{n}$ and therefore that $a_{2n+1}=0$ for all $n\in\mathbb{Z}$. The proof for $s_{-}$ is similar.
\end{proof}

We now state the second \emph{invariant lemma} which describes the Type II invariants in terms of the conformal gluing function $\w$.
\begin{proposition}[Type II invariant lemma]\label{lemma:inv} \
\begin{enumerate}
\item
If $I_2$ is an invariant satisfying the following properties:
\begin{itemize}
\item $I_2$ is holomorphic on a neighborhood of $\overline{\mathfrak{B}}_2$,
\item $I_2$ grows at most polynomially at infinity in $ \overline{\mathfrak{B}}_2$ (i.e when $|z|\to\infty$ and $z\in \overline{\mathfrak{B}}_2$, $|I_2|$ is bounded by the modulus of a polynomial),
\end{itemize}
then $I_2$ is constant. 
\item More generally, if $I_2$ is an invariant satisfying the following properties:
\begin{itemize}
\item $I_2$ has finitely many poles in $\overline{\mathfrak{B}}_2$ which all belongs to $\mathfrak{B}_2\cup\{s_-,s_+ \}$,
\item $I_2$ grows at most polynomially at infinity in $ \overline{\mathfrak{B}}_2$,
\end{itemize} 
then $I_2$ is a rational fraction in $\boldsymbol{\mathrm{w}}_2$, \textit{i.e.} there exists $F_2\in \mathbb C(X)$ such that $I_2=F_2(\boldsymbol{\mathrm{w}}_2)$. 

More precisely, let $p_1,\dots,p_n\in\mathfrak{B}_2$ the poles of $I$ with multiplicities $m_1,\dots,m_n$ and let $2m_-$ (resp. $2m_+$) the order of the (possible) pole $s_-$ (resp. $s_+$) (note that the order of $s_\pm$ is always even by Lemma~\ref{lemma:res} and by convention $m_\pm=0$ when $s_\pm$ is not a pole). Then there exist constants $c_{i,j}$ and $c_{\pm,j}$ such that
\begin{equation}
I_2=F_2\circ \w=c+\sum_{i=1}^{n}\sum_{j=1}^{m_i}\frac{c_{i,j}}{\left(\w-\w(p_i)\right)^j} +\sum_{\pm}\sum_{j=1}^{m_\pm}\frac{c_{\pm,j}}{\left(\w-\w(s_\pm)\right)^j} .
\label{eq:formuleF}
\end{equation}
\end{enumerate}
\end{proposition}
\begin{proof}  The proof is similar to the one of the Type I invariant lemma of Proposition~\ref{lemma:periodic}.
\begin{enumerate}
\item Assume that $I_2$ is holomorphic on a neighborhood of $\overline{\mathfrak{B}}_2$, and consider $g:=I_2\circ \w^{-1}$ which is \textit{a priori} defined and holomorphic on the slit plane 
\begin{equation}
\w(\mathfrak{B}_2)=\mathbb C\setminus \Big\{(-\infty,-1]\cup [1,+\infty)\Big\} .
\end{equation}
\begin{itemize}
\item 
First, we show that $g$ is holomorphic on $\mathbb{C}$. By composition of holomorphic functions, $g$ is holomorphic on $\w(\mathfrak{B}_2)$. Let us show that $g$ can be extended by continuity on the cut $\w(\partial \mathfrak{B}_2)=(-\infty,-1]\cup[1,+\infty)$.
Let $\omega_0\in [1,+\infty)$. We are going to prove that $g$ is continuous at $\omega_0$ (the proof for $\omega_0\in (-\infty,1]$ is similar). We need to show that
$$\lim_{\omega\to \omega_0 \atop \mathfrak{Im}(\omega) > 0} g(\omega)=\lim_{\omega\to \omega_0 \atop \mathfrak{Im}(\omega) < 0} g(\omega).$$
Since $\mathbb{C}\setminus \Big\{(-\infty,-1]\cup [1,+\infty)\Big\}$ is stable under conjugation, and $\omega_0=\overline{\omega_0}$ it suffices to prove that 
\begin{equation}\label{eq1}
\lim_{\omega\to \omega_0 \atop \mathfrak{Im}(\omega) > 0} g(\omega)=\lim_{\omega\to \omega_0 \atop \mathfrak{Im}(\omega) > 0} g(\overline{\omega}).\end{equation}
Let us choose $s_0\in d_-$ such that $\w(s_0)=\omega_0$, this is possible by~\eqref{eq:wd1d2}. Note that we also have $\w(\zeta s_0)=\omega_0$ since $\w$ is an invariant. 

Noticing that $\overline{s_0}=\zeta s_0$ (since $s_0\in d_-$) and $\w^{-1}(\overline{\omega})=\overline{\w^{-1}(\omega)}$ (which derives from standard properties of the logarithm) we obtain 
$$
\lim_{\omega\to\omega_0 \atop \mathfrak{Im}(\omega) > 0} \w^{-1}(\overline{\omega})=
\lim_{\omega\to\omega_0 \atop \mathfrak{Im}(\omega) > 0} \overline{\w^{-1}(\omega)}\in \{ s_0 ,\overline{s_0} \}.
$$
Then since $I_2$ is an invariant we have $I_2(s_0)=I_2(\overline{s_0})$ and since $I_2$ is continuous, we obtain \eqref{eq1}. It is then possible to extend by continuity $g$ at $\omega_0$.

In addition to being holomorphic on the slit plane, we proved that $g$ is continuous on the cut $(-\infty,-1]\cup[1,+\infty)$. It is therefore holomorphic on the whole complex plane (for more details, see Theorem 16.8 in~\cite{rudin}, whose proof relies on Morera's theorem). 

\item We now show that $g$ has subpolynomial growth at infinity. For $\omega$ in the slit plane, using the expression for $\w^{-1}$ given in Lemma~\ref{lemma:w},
$$|\w^{-1}(\omega)|\leqslant |s_-|+\frac{1}{2\pi}\left|\ln\left(\omega+i\sqrt{1-\omega^2}\right)\right|\leqslant |s_-|+\frac{1}{2\pi}\left(\ln\left(\left|\omega+i\sqrt{1-\omega^2}\right|\right)+\pi\right)$$
and when $|\omega|\to+\infty$,
$$\left|\omega+i\sqrt{1-\omega^2}\right|=O(|\omega|),$$
so that by comparing $\ln$ with any polynomial, we get, for all $d\in\mathbb N$ and all $\nu>0$,
$$|\w^{-1}(\omega)|^d=o(|\omega|^{\nu}).$$
By hypothesis of the lemma, let $n>0$ such that for all $s\in \overline{\mathfrak{B}}_2$, $I_2(s)=o(s^n)$ when $s\to\infty$. Hence, for any $\nu>0$ as small as wanted, we have when $|\omega|\to\infty$,
$$|g(\omega)|=|I_2\circ \w^{-1}(\omega)|=o(|\omega|^{n\nu}).$$
\end{itemize}
We have shown that $g$ is holomorphic on $\mathbb{C}$ and has a subpolynomial growth at infinity which implies, by Liouville's theorem, that $g$ is constant.  
We directly deduce that $I_2=g\circ \w$ is constant.

\item We now assume that $I_2$ is meromorphic on a neighborhood of $\overline{\mathfrak{B}}_2$ with finitely many poles 
$p_1,\dots,p_n\in\mathfrak{B}_2$ (with multiplicities $m_1,\dots,m_n$) and possibly a pole at $s_-$ (resp $s_+$) of order $2m_-$ (resp. $2m_+$).  Note that the order of these two last poles is necessary even by parity around the points $s_-$ and $s_+$ (it is a direct consequence of Lemma~\ref{lemma:res}). By convention we take $m_\pm=0$ if there is no pole.

\begin{itemize}
\item First, we show that $I_1\circ \w^{-1}$ is meromorphic on $\mathbb{C}$. Since $\w^{-1}$ is analytic on $\w(\mathfrak{B}_2)=\mathbb C\setminus \big\{ (-\infty,-1]\cup [1,+\infty)\big\}$, we deduce that $I_2\circ \w^{-1}$ is meromorphic on $\w(\mathfrak{B}_2)$, with a finite number of poles at the points $\w(p_i)$ of order $m_i$. 
Using the same reasoning as in the first part of the proof, we see that $I_2\circ \w^{-1}$ is continuous on $(-\infty,-1)\cup(1,+\infty)$ and is then meromorphic on the whole of $\mathbb{C}\setminus \{-1,1 \}$. The only difference here is that the points $\{-1,1 \}=\{\w(s_+),\w(s_-) \}$ are isolated singularities.

We now show that if $I_2$ has a pole of order $2m_-$ (resp. $2m_+$) in $s_-$ (resp. $s_+$) then $I_2\circ \w^{-1}$ has a pole of order $m_-$ (resp. $m_+$) in $\w(s_-)=1$ (resp. $\w(s_{+})=-1$). 
The Taylor expansion of cosine at $s_-$ gives $\w(s)=1- 2\pi^2 (s-s_-)^2+o((s-s_-)^2)$ and taking $s=\w^{-1}(\omega)$ in this formula we deduce that
$$
(\w^{-1}(\omega)-s_-)^2
\underset{\omega\to 1}{\sim}
\frac{1}{2\pi^2}(1-\omega).
$$
Since $I_2$ has a pole of order $2m_-$ at $s_-$, there is a nonzero constant $C$ such that we have
$$
I_2(s)(s-s_-)^{2m_-}\underset{s\to s_-}{\sim} C .
$$
Since $\w^{-1}(\omega)\to s_-$ when $\omega\to 1$, we deduce that 
$$I_2\circ \w^{-1} (\omega)(1-\omega)^{m_-} 
\underset{\omega\to 1}{\sim}
(2\pi^2)^{m_{-}} I_2\circ \w^{-1} (\omega)(\w^{-1}(\omega)-s_-)^{2m_-}
\underset{\omega\to 1}{\sim} (2\pi^2)^{m_{-}} C
$$
which implies that $I_2\circ \w^{-1}$ has a pole of order $m_-$ in $\w(s_-)=1$. The same holds in $\w(s_+)=-1$ which is a pole of order $m_+$.
\item By removing the poles from $I_2\circ \w^{-1}$, we now construct a function $g$ which is holomorphic on $\mathbb{C}$. We consider 
$$g(\omega):=I_2\circ \w^{-1}(\omega)-\sum_{i=1}^n\sum_{j=1}^{m_i} \frac{c_{i,j}}{\left(\omega - \w(p_i) \right)^j} -\sum_{\pm}\sum_{j=1}^{m_\pm} \frac{c_{\pm,j}}{\left(\omega - \w(s_\pm) \right)^j}$$
for well-chosen constants $c_{i,j}$ and $c_{\pm,j}$ such that $g$ has no pole and is then an entire function (the constants $c_{i,j}$ and $c_{\pm,j}$ are just the coefficients of the Laurent series at the points $\w(p_i)$ and $\w(s_\pm)$ of the function $I_2\circ \w^{-1}$).
\item Finaly, since $|\w(s)|\to\infty$ when $s\to\infty$ with $s\in \overline{\mathfrak{B}_2}$, and $I_2$ grows at most polynomially at infinity in $\overline{\mathfrak{B}_2}$, we deduce that 
$$g(\w(s)):=I_2(s) -\sum_{i=1}^n\sum_{j=1}^{m_i} \frac{c_{i,j}}{\left(\w(s) - \w (p_i)  \right)^j}-\sum_{\pm}\sum_{j=1}^{m_\pm} \frac{c_{\pm,j}}{\left(\w(s) - \w(s_\pm) \right)^j}$$
grows at most polynomially at infinity in $\overline{\mathfrak{B}}_2$. 
\end{itemize}
Since $\w$ and $I_2$ are invariants, $g(\w(s))$ is also invariant, and we have shown that it is holomorphic and grows at most polynomially at infinity in $\overline{\mathfrak{B}}_2$.
We deduce from the first result of this lemma proven in 1. that $g(\w(s))$ is equal to a constant $c$
and we obtain that $I_2=F_2\circ \w$ with
\begin{equation*}
F_2(\omega)=c+\sum_{i=1}^{n}\sum_{j=1}^{m_i}\frac{c_{i,j}}{\left(\omega-\w(p_i)\right)^j} +\sum_{\pm}\sum_{j=1}^{m_\pm}\frac{c_{\pm,j}}{\left(\omega-\w(s_\pm)\right)^j} .
\end{equation*} 
\end{enumerate}
\end{proof}
In fact, even if we do not need it in this article, it is possible to relax the hypothesis of the second part of the previous lemma and just assume that $I_2$ has finitely many poles in $\overline{\mathfrak{B}}_2$ which may also lie on $\partial\mathfrak{B}_2$ and still grows at most polynomially. Note that Lemma~\ref{lemma:inv} is similar to the invariant lemma of Proposition~5.4 of \cite{BoMe-El-Fr-Ha-Ra}, for more details on how to relax the assumptions, we refer to its proof.

\begin{remark}[Link between the Type I Invariant Lemma and the Type II Invariant Lemma]
Of course, the results of both invariant lemmas (Proposition~\ref{lemma:periodic} and~\ref{lemma:inv}) are compatible. A function which is invariant under both $\eta$ and $\zeta$ (Type II invariant) is in particular $1$-periodic (Type I invariant), since $\eta\circ\zeta(s)=s+1$. Hence a Type II invariant expressed in terms of 
\[
\w(s)=\cos\left(2\pi\left(s+\frac{\mu_2}{2}\right)\right),
\]
can be rewritten naturally in terms of 
\[
\wb(s)=\tan\left(\pi\left(s+\mu_2-\frac12\right)\right).
\]
This can be seen directly since we can express $\w$ in terms of a rational function of $\wb$ by using the trigonometric identities
\begin{equation}\label{eq5}
\cos (x) =\frac{1-\tan^2(x/2)}{1+\tan^2(x/2)},
\qquad
\tan(x+c)=\frac{\tan( x)+\tan (c)}{1+\tan (x)\,\tan (c)}.
\end{equation}
The second relation is used to adjust the additive constants as desired.  Notice, however, that one cannot express $\tan (x)$ as a rational function of $\cos(2x)$ since the introduction of a square root is needed. Therefore, it is not possible to express $\wb$ in terms of a rational function of $\w$. 

However, one could derive the Type II Invariant Lemma from the Type I Invariant Lemma by the following argument.  If $I$ is a Type II invariant satisfying the hypotheses of the corresponding invariant Lemma, then $I$ is $1$-periodic and hence by the Type I Invariant Lemma can be written as a rational function of $\wb$.  Thus there exists $F\in\mathbb C(X)$ such that
\[
I(s)=F\Big(\tan\left(\pi\left(s-s_+\right)\right)\Big).
\]
The constant $s_+$ in the tangent function has been chosen to be the fixed point of $\eta$.
And since $I$ is invariant by $\eta$ we obtain
\[
F\left(\tan\Big(\pi\left(s-s_+\right)\Big)\right)=I(s)=I(\eta s)=I(\mu_1-s)=F\left(-\tan\Big(\pi\left(s-s_+\right)\Big)\right).
\]
Hence $F(x)=F(-x)$, and it is a standard exercise to show that any rational function satisfying $F(x)=F(-x)$ must be a rational function of $x^2$, say $F(x)=R(x^2)$.  It follows that
\[
I(s)
=R\Big(\tan^2\left(\pi\left(s-s_+\right)\right)\Big),
\]
and since
\[
\tan^2( x)
=\frac{1-\cos(2x)}{1+\cos(2x)},
\]
we conclude that $I$ can indeed be expressed as a rational function in
\[
\w(s) =-\cos\Big(2\pi\left(s-s_+\right)\Big).
\]
We now offer a more algebraic perspective on the preceding discussion. Let $F_{1}$ be the field of functions satisfying the assumptions of the second item of Proposition \ref{lemma:periodic}, and let $F_2$ be the field of functions  satisfying the assumptions of the second item of Proposition \ref{lemma:inv}. We have $F_2\subset F_1$.  We have seen that every element of $F_1$ is a rational function in $\wb$. Conversely, every rational function in $\wb$ is an element of $F_1$. Then, $F_1=\mathbb{C}(\wb)$. Every rational function in $\w$ is a Type II invariant with suitable growth proving that $\mathbb{C}(\w)\subset F_2$. From  \eqref{eq5}, $\w$ may be expressed as a rational function in $\tan(\pi(s+\frac{\mu_2}{2}))$ with numerator and denominator of degree 2. Then the field extension $\mathbb{C}(\tan(\pi(s+\frac{\mu_2}{2}))\,|\,\mathbb{C}(\w)$ is of degree  at most $2$. By the second formula in \eqref{eq5}, $\mathbb{C}(\tan(\pi(s+\frac{\mu_2}{2}))=\mathbb{C}(\wb)$. Then, we deduce that the field extension $\mathbb{C}(\wb)\,|\,\mathbb{C}(\w)$ is of degree  at most $2$.
We have
$$
F_2 \supset \mathbb{C}(\w) \subset \mathbb{C}(\wb)=F_1.
$$
Then, we have an intermediate field extension $F_1\,|\,F_2\,|\,\mathbb{C}(\w)$. This proves that the field extension  $F_1|F_2$ is of degree at most $2$. 
Note that $F_2$ is strictly included in $F_1$ since  $\wb$ is in $F_1$ but not in $F_2$. Then $F_1\,|\,F_2$ is of degree $2$. By the \textit{tower law} for field extensions, the degree of $F_1\,|\,\mathbb{C}(\w)$ is the product of the degree of $F_1\,|\,F_2$ (which is $2$) by the degree of $F_2\,|\,\mathbb{C}(\w)$.
Since $F_1\,|\,\mathbb{C}(\w)$  is of degree at most $2$,the degree of $F_2\,|\,\mathbb{C}(\w)$ has to be $1$. Hence $F_2 = \mathbb{C}(\w)$.
\end{remark}

\subsection{Unknown invariants}
\label{subsec:unknowninv}
In the previous sections we have defined two types of invariants, and their respective canonical invariants $\wb$ and $\w$. In the following lemma we determine invariants using the decoupling function $D$ we have found in~\eqref{eq:decouplageGamma}. We say that these invariants are \emph{unknown} since they depend on $\varphi_1$ which is the function we are looking for.
\begin{lemma}[Unknown invariants]\label{lemma:invariantDphi} 
The function $\varphi_1/D$ is $1$-periodic, \textit{i.e.}, Type I Invariant. Furthermore, if the rational decoupling condition $\gamma\in\mathbb Z$ or $\{\gamma_1,\gamma_2\}\subset\mathbb Z$ holds, then 
\begin{itemize}
\item if $\gamma\in\mathbb Z$, then $\varphi_1/D$ is a Type II invariant;
\item if $\gamma\notin\mathbb Z$ and $\{\gamma_1,\gamma_2\}\subset\mathbb Z$ then $(\varphi_1/D)^2$ is a Type II invariant.
\end{itemize}
All these invariants are meromorphic and grows at most polynomially at infinity in $ \overline{\mathfrak{B}}_1$ and  $ \overline{\mathfrak{B}}_2$ and thus satisfy the invariant lemma hypothesis of Proposition~\ref{lemma:periodic} or~\ref{lemma:inv}.
\end{lemma}
\begin{proof} In both cases, $\varphi_1/D$ is meromorphic on $\mathbb C$ thanks to Theorem~\ref{thm:prolong}. By the Definition~\ref{def:decoupling} of a decoupling function $G(s)=D(s+1)/D(s)$ and with the difference equation~\eqref{eq:qdiff} $\varphi_1(s+1)=G(s)\varphi_1(s)$ we obtain
\begin{equation}
\frac{\varphi_1(s+1)}{D(s+1)}=\frac{G(s)\varphi_1(s)}{D(s+1)}=\frac{\varphi_1(s)}{D(s)}
\label{eq:Dphi1per}
\end{equation}
and we deduce that $\varphi_1/D$ is $1$-periodic.
Let us check that when $\gamma\in\mathbb Z$ or $\{\gamma_1,\gamma_2\}\subset\mathbb Z$ holds the invariance properties~\eqref{eq:invzetaeta} are satisfied. First, remark that a function $f$ which is $1$-periodic and $\eta$-invariant is also $\zeta$-invariant since $\eta\zeta(s)=s+1$ et $\eta^2(s)=s$ (indeed $f(s)=f(s+1)=f(\eta\zeta(s))=f(\eta^2\zeta(s))=f(\zeta s)$).
\begin{itemize}
\item If $\gamma\in\mathbb Z$, cases (a) and (b) of Proposition~\ref{prop:decoupling} state that $D(s)=2^{-\gamma}P(\y(s))/Q(\y(s))$ is a function of $\y(s)$. Since $\varphi_1(s)=\phi_1(\y(s))$ is also a function of $\y(s)$ we deduce that $\varphi_1/D$ is a function of $\y(s)$
and is therefore $\eta$-invariant ($\y$ is $\eta$-invariant, see \eqref{eq:inv}). By~\eqref{eq:Dphi1per} $\varphi_1/D$ is also $1$-periodic and is therefore $\zeta$-invariant. It is then a Type II invariant.
\item If $\gamma\notin \mathbb Z$, $D$ is \textbf{not} a function of $\y(s)$ on  $\mathbb{C}$ because we have $\varepsilon\ne 0$ in the decoupling formula~\eqref{eq:formedecoupl} of Proposition~\ref{prop:decoupling}, see Remark~\ref{rem:gammanotZ} below. In this case we have
$$D(s)=2^{-\gamma}\frac{P\big(\y(s)\big)}{Q\big(\y(s)\big)}\left[\sqrt{2}(s-s_+)\right]^\varepsilon$$ 
where $\varepsilon=\pm 1$ is given in Proposition~\ref{prop:decoupling}. However, since $\eta (s_{+})=s_{+}$, a simple calculation using~\eqref{eq:xyspm} leads to
$$\y(s)-\y(s_+)=2\left(s-s_{+} \right)^2 .$$
Then $(\varphi_1/D)^2$ is a function of $\y(s)$ and thus is $\eta$-invariant and also $1$-periodic by~\eqref{eq:Dphi1per}. The function $(\varphi_1/D)^2$ is therefore $\zeta$-invariant and is then a Type II invariant.
\end{itemize}
Finally, $\varphi_1/D$ (and $(\varphi_1/D)^2$) grows at most polynomially when $s\to\infty$ in the strips $\overline{\mathfrak{B}}_1$ and $\overline{\mathfrak{B}}_2$ since $D(s)\sim s^{2\gamma}$ by \eqref{eq:asymptotic} and $\varphi_1(s)\to 0$ by \eqref{eq:lim0B}. It has also a finite number of poles in these strips (which are all real), see Lemma~\ref{lemma:stripe} and the definition of $D$. The assumptions of the invariant lemmas (Propositions~\ref{lemma:periodic} and~\ref{lemma:inv}) are thus satisfied.
\end{proof}

\begin{remark}[$\varphi_1/D$ is not a Type II invariant when $\gamma\notin \mathbb{Z}$]\label{rem:gammanotZ}
When $\varepsilon\ne 0$ in~\eqref{eq:formedecoupl}, % 
$\varphi_1/D$ is not an $\langle \eta,\zeta \rangle$-invariant. This comes from the fact that $(s-s_+)$ is not $\eta$-invariant since $$\eta( s-s_+)=-(s-s_+)\neq (s-s_+).$$  
Note that it is true that for all $s\in\mathfrak{B}_2$,
$$\sqrt{2}(s-s_+)=\sqrt{\y(s)-\y(s_+)}.$$
However, this formula is not well-defined on a neighborhood of $\overline{\mathfrak{B}}_2$ due to the cut of the square root on $(-\infty,0)$. Indeed for $s\in d_+=s_++i\mathbb{R}$ we have $\y(s)-\y(s_+)=2\left(s-s_{+} \right)^2\leqslant 0$. Therefore, we cannot deduce that $(s-s_+)$ is $\eta$-invariant even if it can be expressed as a function of $\y$ in $\mathfrak{B}_2$ (this domain is nowhere stable by $\eta$ since $\eta \mathfrak{B}_2 \cap \mathfrak{B}_2 =\varnothing$). 
\end{remark}

We can now give a concise proof of Theorem~\ref{thm:main}, stated in Section~\ref{subsec:mainres}.
\begin{proof}[Proof of Theorem~\ref{thm:main}]
According to Lemma~\ref{lemma:invariantDphi}, $\varphi_1/D$ is a Type I invariant. Moreover, both $\varphi_1$ and $1/D$ are meromorphic on a neighborhood of $\overline{\mathfrak{B}}_1$ (see Theorem~\ref{thm:prolong}), have finitely many poles in this strip and grow at most polynomially ($\varphi_1$ tends to $0$ and $D(s)\sim s^{2\gamma}$ according to Lemma~\ref{lemma:asymptotic}). Hence, $\varphi_1/D$ satisfies all the conditions of the Type I invariant lemma (Proposition~\ref{lemma:periodic}), and therefore $\varphi_1/D=F_1(\wb)$ with $F_1\in \mathbb C(X)$. The rest of the proof consists of writing $y=\y(s)$ and applying Equation~\eqref{eq:varphitophi} to obtain $\phi_1(y)$ from $\varphi_1(s)$.
\end{proof}

\section{Explicit expression of the Laplace transform}\label{sec:Laplace}
The strategy of this section is to express the unknown invariants of Section~\ref{subsec:unknowninv} in terms of the canonical invariants $\wb$ and $\w$ using the Invariant Lemmas of Propositions~\ref{lemma:periodic} and~\ref{lemma:inv}.
This will enable us to state and proove our main theorems~\ref{thm:N},~\ref{thm:-N}, \ref{thm:notN} and~\ref{thm:general} that compute explicitly the Laplace transform $\phi_1$ whether under rational decoupling condition~\eqref{eq:decouplcond} or not. Symmetric formulas hold for $\phi_2$ and $\phi$ is obtained using the functional equation~\eqref{eq:FE}.

\begin{remark}[Normalization constants and notation] 
In order to avoid carrying around constants whose expression is of little importance at this stage, in the following theorems we will give expressions of the Laplace transform $\phi_1$ \textit{up to a multiplicative constant}, denoted using the proportionality symbol $\propto$ instead of the equality sign.
One can easily compute the proportionality constants in Theorems~\ref{thm:N},~\ref{thm:-N}, \ref{thm:notN} and~\ref{thm:general} thanks to the value of $\phi_1(0)$ and $\phi_2(0)$ given in~\eqref{eq:normcst}. 
\end{remark}

\begin{remark}[Meromorphic functions on $\mathbb{C}$]\label{rem:sinc}The expressions for the Laplace transforms given in Theorems~\ref{thm:-N} and~\ref{thm:notN} involve (up to pre-composition with an affine function) the following functions:
\begin{equation}
f_1(z)=\cos(\sqrt{z}),\quad f_2(z):=\frac{\sin(\sqrt{z})}{\sqrt{z}},\quad f_3(z):=\frac{\tan(\sqrt{z})}{\sqrt{z}}.
\end{equation}
At first glance, these functions may not appear to be well-defined for negative real values of $z$. However, they are actually meromorphic on the entire complex plane and even holomorphic in the case of the first two. 
The functions $\cos(z)$, $\sin(z)$ are holomorphic on $\mathbb{C}$, and admit Taylor expansions near $z = 0$. It implies that
\[
f_1(z) = \cos(\sqrt{z}) = \sum_{n=0}^\infty \frac{(-1)^n z^n}{(2n)!},
\quad
f_2(z) = \frac{\sin(\sqrt{z})}{\sqrt{z}} = \sum_{n=0}^\infty \frac{(-1)^n z^n}{(2n+1)!},
\]
which are power series in $z$, convergent on all of $\mathbb{C}$, so $f_1$ and $f_2$ are holomorphic on $\mathbb{C}$ and $f_3=f_2/f_1$ is meromorphic on $\mathbb{C}$.
\end{remark}

\subsection{Cases where $\gamma\in\mathbb{Z}$}

The following theorem generalizes a result on a \textit{skew symmetric} case ($\gamma=-1$) by Ichiba and Karatzas~\cite{ichiba_karatzas_degenerate_22}.

\begin{theorem}[Laplace transform, $\gamma\in-\mathbb N$]\label{thm:N} If $\gamma \in-\mathbb N$, then there exists a degree $-\gamma$ polynomial $Q\in\mathbb R[X]$ such that the Laplace transform $\phi_1$ satisfies
$$\phi_1(y)\propto\frac{1}{Q(y)}.$$
Here, the polynomial $Q$ is given by
$$Q(y):=\prod_{k=0}^{-\gamma-1}\Big(y-\y(s_1-k)\Big).$$
\end{theorem}

\begin{proof} 
Recall by Proposition~\ref{prop:decoupling} that $D(s)=2^{-\gamma}/Q(\y (s))$, and by Lemma \ref{lemma:invariantDphi} $\varphi_1/D$ is a Type II invariant. By virtue of the third point of Lemma~\ref{lemma:stripe}, $\varphi_1/D$ is bounded by a polynomial at infinity. According to the second item of the same lemma, $\varphi_1$ is meromorphic and has (at most) a simple pole in $\overline{\mathfrak{B}}_2$ at $s_1$ when $s_1\in \overline{\mathfrak{B}}_2$. Since $1/D(s)=2^{\gamma}Q(\y (s))$ is a polynomial which has a root at $s_1$, we deduce that $\varphi_1/D$ is holomorphic on an neighborhood of $\overline{\mathfrak{B}}_2$. We can thus apply the Type II Invariant Lemma (Proposition~\ref{lemma:inv}), we deduce there exists a constant $c\in\mathbb R$ such that 
$$\varphi_1/D\equiv c$$
which completes the proof.
\end{proof}

The next theorem is a generalization of the one of Franceschi et al.~\cite{FIKR23+} where $\gamma=3$.

\begin{theorem}[Laplace transform, $\gamma \in \mathbb N$]\label{thm:-N} If $\gamma \in\mathbb N$, then there exists a degree $\gamma$ polynomial $P\in\mathbb R[X]$ such that the Laplace transform $\phi_1$ satisfies 
$$\phi_1(y)\propto\frac{P(y)}{\cos(\pi\sqrt{2y+\mu_1^2})-\cos(\pi\gamma_1)}.$$
Here, the polynomial $P$ is given by 
$$P(y):=\prod_{k=1}^{\gamma}\Big(y-\y(s_1+k)\Big).$$
\end{theorem}

\begin{proof}
From Lemma~\ref{lemma:invariantDphi} we know that $\varphi_1/D$ is a Type II Invariant and that we can apply the invariant lemma of Proposition~\ref{lemma:inv}. To do so, we study the poles of $\varphi_1/D$ in $\overline{\mathfrak{B}}_2$. Recall from Proposition~\ref{prop:decoupling} that $D(s)=2^{-\gamma} P(\y(s))$, the zeros of $D$ being the $s_1+k$ and $\eta(s_1+k)$ for $k=1,\dots,\gamma$. We claim that the function $\varphi_1/D$ has exactly one pole in $\overline{\mathfrak{B}}_2$ which we denote $\widetilde{s}$ and whose position and multiplicity only depends on the value of $\gamma_1$ defined at~\eqref{eq:gamma}:
\begin{itemize}
\item if $\gamma_1=1$, then $s_1=s_-$ and $\varphi_1$ has exactly one simple pole in $\overline{\mathfrak{B}}_2$ at $s_-$ (see Lemma~\ref{lemma:stripe}). On the other hand $\eta(s_1)=\eta(s_{-})=\frac{\mu_2}{2}+\mu_1=-\frac{\mu_2}{2}+1=s_1+1$, so that $D$ also has a simple zero at $s_-$, resulting in a double pole for $\varphi_1/D$ at $\widetilde{s}:=s_-$,
\item if $\gamma_1\in (0,1)$, then $s_1\in (s_-,s_+)$ and $\varphi_1$ has exactly one simple pole in $\overline{\mathfrak{B}}_2$ at $\widetilde{s}:=s_1$ (see Lemma~\ref{lemma:stripe}). All the poles of $1/D$ lie outside $\overline{\mathfrak{B}}_2$. More precisely, the points $s_1+k$ all lie to the right of $s_+$ and the $\eta(s_1+k)$ all lie to the left of $s_-$. 
\item if $1<\gamma_1<2\gamma_1+1$ then $\varphi_1$ has no poles in $\overline{\mathfrak{B}}_2$. The zeros of $D$ belong to the stripe $\overline{\mathfrak{B}}_2$ according to the following conditions:
$$
\begin{array}{ccccc}
\displaystyle s_1+ k \in \overline{\mathfrak{B}}_2 &\iff&
\displaystyle-\frac{\mu_2}{2}-s_1\leq  k\leq   \frac{\mu_1}{2}-s_1&\iff&\displaystyle  \frac{\gamma_1-1}{2}\leqslant k \leqslant \frac{\gamma_1}{2}  \\[0.3cm]
\eta(s_1-k)\in \overline{\mathfrak{B}}_2&\iff&\displaystyle-\frac{\mu_1}{2}-s_1+\mu_1\leq  k \leq   \frac{\mu_2}{2}-s_1+\mu_1 &\iff&\displaystyle \frac{\gamma_1}{2}\leqslant k \leqslant \frac{\gamma_1+1}{2}.
\end{array}
$$
One can then easily check that if $\gamma_1$ is an odd (\textit{resp.} even) integer, then $\varphi_1/D$ has a double pole at $\widetilde{s}:=s_-$ (\textit{resp}. $\widetilde{s}:=s_+$) and if there exists $k\in\mathbb Z$ such that $2k<\gamma_1<2k+1$ (\textit{resp}. $2k-1<\gamma_1<2k$) then $\varphi_1/D$ has a simple pole at $\widetilde{s}:=s_1+\lfloor\gamma_1/2\rfloor$ (\textit{resp}. $\widetilde{s}:=\eta(s_1+\lceil \gamma_1/2\rceil)$).
\item otherwise, if $\gamma_1\geqslant 2\gamma+1$ or $\gamma_1\leqslant 0$ we cannot have $\gamma>0$, and this case is therefore incompatible with the hypothesis of the theorem. Indeed, if $\gamma_1< 0$, then by~\eqref{eq4}, $\gamma_2=2\gamma+1-\gamma_1>0$ which leads to the impossible case 3 of Lemma~\ref{lemma:impossible}, and if $\gamma_1\geqslant 2\gamma+1$, then $\gamma_2<0$, which correspond to the impossible case 4 of Lemma~\ref{lemma:impossible}. Note that by case 1 of Lemma~\ref{lemma:impossible}, $\gamma_1\neq 0$.
\end{itemize}
We now apply the Type II Invariant Lemma (Proposition~\ref{lemma:inv}): there exists $c\in\mathbb R$ and $\widetilde{c}\neq 0$ such that
$$\frac{\varphi_1}{D}=c+\frac{\widetilde{c}}{\w-\w(\widetilde{s})}$$
(to match the notations of the invariant lemma, one must set $\widetilde{c}:=c_{\pm,1}$ and $m_\pm=1$ when $\widetilde{s}\in\{s_-,s_+\}$, and $\widetilde{c}:=c_{1,1}$ otherwise). Taking the limit as $s\to +\infty$ for $s\in \overline{\mathfrak{B}}_2$, we have $\varphi_1(s)\to 0$ by~\eqref{eq:lim0B}, $D(s)\to\infty$ by \eqref{eq:asymptotic} and $\w(s)\to \infty$, which ensures $c=0$. We deduce that,
\begin{equation}\label{eq:proof44}
\phi_1(\y(s))=\varphi_1(s)\propto \frac{P(\y(s))}{\w(s)-\w(\widetilde{s})} .
%\propto \frac{P(\y(s))}{\w(s)+\cos(\pi\gamma_1)},
\end{equation}
Note that, in any case, 
$$\w(\widetilde{s})=-\cos(\pi \gamma_1).$$
Indeed, $\widetilde{s}$ is either of the form $s_1+k$ or $\eta(s_1+k)$ for some integer $k$. In the first case 
$$\w(s_1+k)=\w(s_1)=\cos(2\pi(s_1-s_-))=\cos(\pi(2s_1+\mu_2))=\cos(\pi(-\gamma_1+1))=-\cos(\pi\gamma_1),$$
and in the second case
$$\w(\eta(s_1+k))=\cos(2\pi(-(s_1+k)+\mu_1-s_-))=\cos(\pi(\gamma_1-2k+1))=-\cos(\pi\gamma_1).$$
Similarly, using the relations $s=\frac{1}{2}\Big(\mu_1\pm\sqrt{2\y(s)+\mu_1^2}\Big)$ and $\mu_1+\mu_2=1$,
\begin{align*}
\w(s)=\cos\Big(2\pi(s-s_-)\Big)=\cos\left(\pi\Big(\mu_1\pm\sqrt{2\y(s)+\mu_1^2}+\mu_2\Big)\right)=-\cos\left(\pi\sqrt{2\y(s)+\mu_1^2}\right)
\end{align*}
which yields the desired result.
\end{proof}

\subsection{Cases where $\gamma_1,\gamma_2\in\mathbb{Z}$}
Recall that $\w(s)=\cos\big(2\pi(s-s_-)\big)$. First, we state some relations that will be used several times in the rest of the paper. The trigonometric function $\w$ satisfies the so-called double-angle formulas:
\begin{equation}\label{rem2}
\w(s)+1 =2\cos^2\left(\pi(s-s_-)\right), \quad \w(s)-1 =-2\sin^2\left(\pi(s-s_-)\right).
\end{equation}

The following lemma is a technical result that will allow us, during the proofs of the main theorems, to choose the correct expression for $\varphi_1(s)$ from two alternatives. 
\begin{lemma}\label{lemma:sqrt}
Let $a,b\in\mathbb R$. If the function $a\w+b$ admits a meromorphic square root (\textit{i.e.} there exists a meromorphic function $f$ such that $f^2=a\w+b$) then $a=0$ or $a^2=b^2$.
\end{lemma}

\begin{proof} Assume that $a\w+b$ admits a meromorphic square root $f$.
Suppose $a\ne 0$ and let us prove that $a^2=b^2$. Then $a\w +b$ is surjective (because the cosinus is) and there exists $s_0\in\mathbb C$ such that $f(s_0)^2=a\w(s_0)+b=0$. Since $a\w+b$ can be written as the square of a meromorphic function, all its zeros (and poles) must have even multiplicities. In particular, 
$$0=\frac{\mathrm{d}}{\mathrm{d}s}\Big(a\w(s)+b\Big)\Big|_{s=s_0}=2\pi a \sin\left(2\pi\left(s_0-s_-\right)\right).$$
Finally, by the Pythagorean identity
$$1=\cos^2\left(2\pi\left(s_0-s_-\right)\right)+\sin^2\left(2\pi\left(s_0-s_-\right)\right)=\cos^2\left(2\pi\left(s_0-s_-\right)\right).$$
so that $\w (s_0)^2=\cos^2\left(2\pi\left(s_0-s_-\right)\right)=\pm 1$. 
We obtain that
$0=f^2(s_0)=\pm a+b$, 
leading to $a^2=b^2$. 
\end{proof}

The next theorem deals with  
the cases where $\gamma_1,\gamma_2\in\mathbb{Z}$ and $\gamma\notin \mathbb{Z}$ so that $\gamma_1,\gamma_2$ cannot be simultaneously positive integers (Lemma~\ref{lemma:impossible}) and $\gamma_1$ and $\gamma_2$ are either both even or both odd (Remark~\ref{rem:gamma}).

\begin{theorem}[Laplace transform, $\{\gamma_1,\gamma_2\}\subset\mathbb Z$]\label{thm:notN} Suppose $\{\gamma_1,\gamma_2\}\subset \mathbb Z$ and $\gamma\notin \mathbb Z$.
\begin{enumerate}
\item If \textbf{either} $\gamma_1\in-2\mathbb{N}$ and $\gamma_2\in 2\mathbb{N}$ \textbf{or} $\gamma_1\in2\mathbb{N}-1$ and $\gamma_2\in -2\mathbb{N}+1$, 
then there exists polynomials $P$ and $Q$ such that
$$\phi_1(y)\propto \frac{P(y){\tan}\left(\frac{\pi}{2}\sqrt{2y+\mu_1^2}\right)}{Q(y)\sqrt{{2}y+\mu_1^2}}.$$
Here, $P$ and $Q$ are given by
$$P(y):=\prod_{k=0}^{\frac{\gamma_2}{2}-1}\Big(y-\y(s_2-k)\Big)\text{ and }Q(y):=\prod_{k=0}^{-\frac{\gamma_1}{2}-1}\Big(y-\y(s_1-k)\Big)$$
in the former case, and by
$$P(y):=\prod_{k=1}^{\frac{\gamma_1-1}{2}}\Big(y-\y(s_1+k)\Big)\text{ and }Q(y):=\prod_{k=1}^{-\frac{\gamma_2+1}{2}}\Big(y-\y(s_2+k)\Big)$$
in the latter.
\item If \textbf{either} $\gamma_1\in2\mathbb{N}$ and $\gamma_2\in -2\mathbb{N}$ \textbf{or} $\gamma_1\in -2\mathbb{N}+1$ and $\gamma_2\in 2\mathbb{N}-1$,
then there exists polynomials $P$ and $Q$ such that the Laplace transform $\phi_1$ satisfies
$$\phi_1(y)\propto \frac{P(y)\sqrt{2y+\mu_1^2}}{Q(y)\tan\left(\frac{\pi}{2}\sqrt{2y+\mu_1^2}\right)}.$$
Here, $P$ and $Q$ are given by
$$P(y):=\prod_{k=1}^{\frac{\gamma_1}{2}-1}\Big(y-\y(s_1+k)\Big)\text{ and }Q(y):=\prod_{k=1}^{-\frac{\gamma_2}{2}} \Big(y-\y(s_2+k)\Big)$$
in the former case, and by
$$P(y):=\prod_{k=0}^{\frac{\gamma_2+1}{2}-1}\Big(y-\y(s_2-k)\Big)\text{ and }Q(y):=\prod_{k=0}^{-\frac{\gamma_1+1}{2}}\Big(y-\y(s_1-k)\Big)$$
in the latter.
\item If $\gamma_1\in\mathbb{N}$ and $\gamma_2\in\mathbb{N}$, then there exists a polynomial $P\in\mathbb R[X]$ such that the Laplace transform $\phi_1$ satisfies
$$\phi_1(y)\propto\frac{P(y)\sqrt{2y+\mu_1^2}}{\mathrm{sin}\left(\pi\sqrt{2y+\mu_1^2}\right)}.$$
Here, the polynomial $P$ is given by
\begin{equation}\label{eq:decouplinggamma1gamma2}
P(y):=\prod_{k=1}^{\lfloor\frac{\gamma_1-1}{2}\rfloor} \Big(y-\y(s_1+k)\Big)\prod_{k=0}^{\lfloor\frac{\gamma_2}{2}-1\rfloor} \Big(y-\y(s_2-k)\Big).
\end{equation}
\end{enumerate}
\end{theorem}
\begin{proof}
When $\{\gamma_1,\gamma_2\}\subset \mathbb Z$ and $\gamma\notin \mathbb Z$, we know from Lemma~\ref{lemma:invariantDphi} that $(\varphi_1/D)^2$ is a Type II Invariant and that we can apply the invariant lemma of Proposition~\ref{lemma:inv}. To do this, we study the poles of $(\varphi_1/D)^2$ in $\overline{\mathfrak{B}}_2=\{s\in\mathbb C : -\mu_2/2 \leqslant\mathfrak{Re}(s)\leqslant \mu_1/2\}$. Note that Lemma~\ref{lemma:stripe} can be interpreted as follow: $\varphi_1$ has a poles in $\overline{\mathfrak{B}}_2$ if and only if $\gamma_1\in (\mu_1,1]$.
\begin{enumerate}
\item Under these conditions and as established in Proposition~\ref{prop:decoupling} the decoupling function $D$ is rational and given by 
$$D(s)\propto\frac{P(\y(s))}{Q(\y(s))(s-s_+)}$$
(see the Proposition for the explicit polynomials $P$ and $Q$). We distinguish two cases based on the signs and parities of $\gamma_1$ and $\gamma_2$.
\begin{itemize}
\item[(a)] Suppose $\gamma_1\in-2\mathbb{N}$ and $\gamma_2\in 2\mathbb{N}$. We know from Lemma~\ref{lemma:stripe} that $\varphi_1$ has a pole in $\overline{\mathfrak{B}}_2$ if and only if $0<\mu_1<\gamma_1\leqslant 1$ which cannot be the case here. According to~\eqref{eq:resulting1}, all the zeros (\textit{resp.} poles) of $D$ in $\mathbb C$ are simple, and situated at $s_2-k$ for $k=0,\dots,\gamma_2-1$ (\textit{resp.} $s_1-k$ for $k=0,\dots,-\gamma_1$). The strip $\overline{\mathfrak{B}}_2$ being of width $\frac{1}{2}$, $D$ has at most one zero and one pole in $\overline{\mathfrak{B}}_2$. More precisely,
$$
\begin{array}{ccccc}
\displaystyle s_2- k \in \overline{\mathfrak{B}}_2 &\iff&
\displaystyle-\frac{\mu_1}{2}+s_2\leq  k\leq   \frac{\mu_2}{2}+s_2&\iff&\displaystyle  \frac{\gamma_2-1}{2}\leqslant k \leqslant \frac{\gamma_2}{2}  \\[0.3cm]
s_1-k\in \overline{\mathfrak{B}}_2&\iff&\displaystyle-\frac{\mu_1}{2}+s_1\leq  k \leq   \frac{\mu_2}{2}+s_1 &\iff&\displaystyle -\frac{\gamma_1}{2}\leqslant k \leqslant -\frac{\gamma_1-1}{2}.
\end{array}
$$
Since both $\gamma_1$ and $\gamma_2$ are even, the first (\textit{resp.} second) inequality is satisfied by $k=\gamma_2/2$ (\textit{resp.} $-\gamma_1/2$), resulting in a simple zero at $s_-=-\mu_2/2$ and a simple pole at $s_+=\mu_1/2$. We now apply the invariant lemma of Proposition~\ref{lemma:inv}: there exists two real numbers $c$ and $c_{-,1}\ne 0$ such that 
$$\left(\frac{\varphi_1}{D}\right)^2=c+\frac{c_{-,1}}{\w-\w(s_-)}=c+\frac{c_{-,1}}{\w-1}.$$
We now seek to determine a relationship between the constants $c$ and $c_{-,1}$. Applying the double-angle formula~\eqref{rem2} leads to the following identity:
$$\left[\frac{\varphi_1(s)}{D(s)}\sin\left(\pi\left(s-s_-\right)\right)\right]^2\propto c\w(s)+(c_{-,1}-c).$$
By Lemma~\ref{lemma:sqrt}, either $c=0$ or $c^{2}=(c_{-,1}-c)^2$. If $c$ were equal to 0 we would obtain
$$\frac{\varphi_1(s)}{D(s)}\propto \frac{1}{\sin\Big(\pi(s-s_-)\Big)}$$
which is not possible, since at $s=s_+$, the left-hand side of the above relation would vanish, while the right-hand side would remain nonzero. Hence $c^2=(c_{-,1}-c)^2$, \textit{i.e.} $c_{-,1}=0$ or $c_{-,1}=2c$. But the quantity $c_{-,1}$ is nonzero (it would contradict the existence of a pole). We conclude that $c_{-,1}=2c$. By \eqref{rem2} we obtain
$$\left(\frac{\varphi_1(s)}{D(s)}\right)^2\propto \frac{\w(s)+1}{\sin\Big(\pi(s-s_-)\Big)^2}\propto\frac{1}{\tan\Big(\pi(s-s_-)\Big)^2},$$
and finally
$$\varphi_1(s)\propto \frac{P(\y(s))}{Q(\y(s))(s-s_+)} \frac{1}{\tan\Big(\pi(s-s_-)\Big)}.$$
We then recover $\phi_1(y)$ from $\varphi_1(s)$ setting $y=\y(s)$ and applying Equation~\eqref{eq:varphitophi}:
$$\frac{-1}{(s-s_+)\tan\left(\pi(s-s_-)\right)}=\frac{\tan\left(\displaystyle\frac{\pi}{2}(2(s-s_-)-1)\right)}{(s-s_+)}=\frac{\displaystyle\tan\left(\frac{\pi}{2}\sqrt{2\y(s)+\mu_1^2}\right)}{\sqrt{2\y(s)+\mu_1^2}}.$$
Here, we used the trigonometric identity $\tan(\,\cdot\,-\pi/2)=-1/\tan(\,\cdot\,)$.
\item[(b)] Suppose $\gamma_1\in2\mathbb{N}-1$ and $\gamma_2\in -2\mathbb{N}+1$. If $\gamma_1\in \mathbb N\setminus\{1\}$ (so that $\varphi_1$ has no pole in $\overline{\mathfrak{B}}_2$), using the expression of $D(s)$ given at~\eqref{eq:resulting2}, $D$ has at most one simple zero (\textit{resp.} one simple pole) in $\overline{\mathfrak{B}}_2$, of the form $s_1+k$ (\textit{resp.} $s_2+k$) where $k$ is an integer. When $\gamma_1$ and $\gamma_2$ are odd and $\gamma_1\ne 1$, $D$ has a simple zero at $s_-$ and a simple pole at $s_+$. If $\gamma_1=1$, then $\varphi_1$ has a simple pole at $s_1=s_-$, but $D$ has no longer a zero at $s_1=s_-$, still resulting in a unique simple pole for $\varphi_1/D$ at $s_-$. For both subcases $\gamma_1\ne 1$ and $\gamma_1=1$, the rest of the proof mirrors the argument given above.
\end{itemize}
\item 
Under these conditions, the proof follows the same steps as in case 1. We just give a few key steps, most of the argument relies on interchanging the roles of $s_+$ and $s_-$. Recall that, according to Lemma~\ref{lemma:stripe}, $\varphi_1$ has a pole in $\overline{\mathfrak{B}}_2$ if and only if $0<\mu_1<\gamma_1\leqslant 1$, hence in these cases, $\varphi_1$ does not have pole in the strip. Proceeding along the same lines as the previous case, one can show that the decoupling function $D$ has only one simple zero in $\overline{\mathfrak{B}}_2$ situated at $s_+$ and one simple pole at $s_-$. By the invariant Lemma, there exists $c$ and $c_{+,1}$ such that
$$\left(\frac{\varphi_1}{D}\right)^2 = c+\frac{c_{+,1}}{\w-\w(s_+)}=c+\frac{c_{+,1}}{\w+1}.$$
Once again by~\eqref{rem2},
\begin{equation}\label{eq:aux2}
\left[\frac{\varphi_1(s)}{D(s)}\cos\left(\pi\left(s-s_-\right)\right)\right]^2\propto c\w(s)+(c_{+,1}+c),
\end{equation}
which, applying Lemma~\ref{lemma:sqrt}, leads to $c_{+,1}=-2c$. We conclude as in case 1 using~\eqref{eq:sys} and a few trigonometric formulas.
\item Suppose $\gamma_1$ and $\gamma_2$ are both positive integers. Proposition~\ref{prop:decoupling} implies that $$D(s)\propto (s-s_+)P(\y(s))$$ for an (explicit) polynomial $P$. As for the previous statements, the case $\gamma_1=1$ should be treated aside. If $\gamma_1\ne 1$ (\textit{resp.} $\gamma_1=1$), then $\varphi_1$ has no pole in the strip (\textit{resp.} has a simple pole at $s_-$), and as described by~\eqref{eq:resulting3}, $D$ has exactly two simple zeros in $\overline{\mathfrak{B}}_2$, at $s_-$ and $s_+$ (\textit{resp.} one simple zero at $s_+$). In any case, $\varphi_1/D$ has a two simple poles, at $s_-$ and $s_+$. By the invariant Lemma, there exists $c$ and $c_{\pm,1}\ne 0$ such that
\begin{equation}\label{eq:invlemmaconsequence}
\left(\frac{\varphi_1}{D}\right)^2 =c+\frac{c_{-,1}}{\w-\w(s_-)}+\frac{c_{+,1}}{\w-\w(s_+)}=c+\frac{c_{-,1}}{\w-1}+\frac{c_{+,1}}{\w+1}.
\end{equation}
Recalling from Proposition~\ref{prop:decoupling}, $\text{deg}(D)=2\gamma=\gamma_1+\gamma_2-1>0$ and from Lemma~\ref{lemma:stripe} that $\lim \varphi_1(s)=0$ as $|s|\to +\infty$ in $\mathfrak{B}_2$, one can take the limit in~\eqref{eq:invlemmaconsequence} and obtain
$$c=\lim_{|s|\to +\infty}\left(c+\frac{c_{-,1}}{\w-1}+\frac{c_{+,1}}{\w+1}\right)=\lim_{|s|\to +\infty} \left(\frac{\varphi_1(s)}{D(s)}\right)^2=0.$$
Rearranging Equation~\eqref{eq:invlemmaconsequence} accordingly leads to
$$\left(\w(s)^2-1\right)\left(\frac{\varphi_1(s)}{D(s)}\right)^2=c_{-,1}(\w(s)+1)+c_{+,1}(\w(s)-1)=(c_{-,1}+c_{+,1})\w(s)+(c_{-,1}-c_{+,1}).$$
The left hand side of the above equation is the square of a meromorphic function, hence by Lemma~\ref{lemma:sqrt} either $c_{-,1}=-c_{+,1}$ or $(c_{-,1}+c_{+,1})^2=(c_{-,1}-c_{+,1})^2$. This second alternative is equivalent to $c_{-,1}=0$ or $c_{+,1}=0$ which would contradict the existence of a pole. Hence $c_{-,1}=-c_{+,1}$ and 
\begin{equation}\label{eq:aux1}
\phi_1(\y(s))^2=\varphi_1(s)^2=\frac{2c_{-,1}D(s)^2}{\w(s)^{2}-1}\propto \frac{P(\y(s))^2(s-s_+)^2}{\sin\left(2\pi(s-s_-)\right)^2}.
\end{equation}
We conclude as in the previous cases using~\eqref{eq:sys}.
\end{enumerate}
\end{proof}

\subsection{Cases where $\gamma\notin\mathbb{Z}$ and $\{\gamma_1,\gamma_2\}\not\subset\mathbb{Z}$}
The goal of this subsection is to give an expression of $\varphi_1$ in the remaining cases. To do so, we first state a lemma about the poles of $\varphi_1/D$ in these cases.

\begin{lemma}[Poles of $\varphi_1/D$]
\label{lemma:polephi1D}
If $\gamma\notin \mathbb Z$ and $\{\gamma_1,\gamma_2\}\not\subset \mathbb Z$, then the function
$$
\frac{\varphi_1(s)}{D(s)}=\varphi_1(s)\frac{\Gamma(s-s_2)\Gamma(s+s_1+\mu_2)}{\Gamma(s-s_1)\Gamma(s+s_2+\mu_2)}
$$ 
has at most two poles (or a unique double pole) in $\mathfrak{B}_1\cup \{ \mu_1 \}$. More precisely, the only poles are:
\begin{itemize}
\item $s_2-p$ for some $p\in\mathbb{N}$ when $s_2>0$,
\item $\zeta s_1-q=-s_1-\mu_2-q$ for some $q\in\mathbb{N}$ when $s_1<0$.
\end{itemize}
except when $s_2-p=-s_1-\mu_2-q$ as they merge into a unique double pole and this happen when $s_2-\zeta s_1\in\mathbb Z$.
\end{lemma}
\begin{proof}
We first list the poles of the various functions appearing in the expression above:
\begin{itemize}
\item The function $1/\Gamma$ is entire and thus $1/\Gamma(s-s_1)\Gamma(s+s_2+\mu_2)$ has no poles.
\item The poles of $\Gamma(s-s_2)$ are the real points $s_2-p$ for $p\in\mathbb{N}$. Since by Lemma~\ref{lemma:s1s2} we know that $s_2\notin [-\mu_2,0]$, there exists $p\in\mathbb{N}$ such that $s_2-p\in [-\mu_2,\mu_1]\subset \mathfrak{B}_1\cup \{ \mu_1 \}$ if and only if $s_2>0$.
\item The poles of $\Gamma(s+s_1+\mu_2)=\Gamma(s-\zeta s_1)$ are the real points $\zeta s_1 - q$ for $q\in\mathbb{N}$. There exists $q\in\mathbb{N}$ such that $\zeta s_1 - q \in (-\mu_2,\mu_1]\subset \mathfrak{B}_1\cup \{ \mu_1 \}$ if and only if $-s_1-\mu_2=\zeta s_1>-\mu_2$, that is, if $s_1<0$.
\item By Proposition~\ref{lemma:stripe}, $\varphi_1$ has a pole in $\mathfrak{B}_1\cup \{ \mu_1 \}$ if and only if $-\mu_2 \leqslant s_1 < 0$. When this happens, this pole is at $s_1$, it is unique, and it is simple.
\end{itemize}
We now check whether any zeros could compensate for the poles listed above:
\begin{itemize}
\item By Proposition~\ref{lemma:stripe}, $\varphi_1$ has no zeros in $(-\mu_2,\mu_1]$.
\item Let us use $\{\gamma_1,\gamma_2\}=\{\mu_1-2s_1,\mu_2+2s_2 \}\not\subset\mathbb Z$ and $\gamma=s_2-s_1\notin \mathbb Z$, to prove that the zeros of the function $1/\Gamma(s-s_1)\Gamma(s+s_2+\mu_2)$ cannot cancel the poles of $\Gamma(s-s_2)\Gamma(s+s_1+\mu_2)$. Let us use the notation $\sim$ from the proof of Proposition~\ref{prop1}. The poles of $\Gamma(s - s_i)$ are simple and occur at $s = s_i - k$ with $k \in \mathbb{N}$. Therefore, if $a \not\sim b$ are two complex numbers, the zeros of $1/\Gamma(s - a)$ cannot cancel the poles of $\Gamma(s - b)$. In the proof of Proposition \ref{prop1}, we observed that $\gamma \notin \mathbb{Z}$ implies
$$s_1 \not\sim s_2 \quad \text{and} \quad -s_1 - \mu_2 \not\sim -s_2 - \mu_2.$$
Hence, the zeros of $1/\Gamma(s - s_1)$ cannot cancel the poles of $\Gamma(s - s_2)$, and similarly, the zeros of $1/\Gamma(s + s_2 + \mu_2)$ cannot cancel the poles of $\Gamma(s + s_1 + \mu_2)$. The only remaining possibility is a cancellation between a zero of $1/\Gamma(s - s_1)$ and a pole of $\Gamma(s + s_1 + \mu_2)$, or between a zero of $1/\Gamma(s + s_2 + \mu_2)$ and a pole of $\Gamma(s - s_2)$. We also observed in the proof of Proposition~\ref{prop1} that
$$\gamma_1 \in \mathbb{Z} \iff s_1 \sim -s_1 - \mu_2, \quad \gamma_2 \in \mathbb{Z} \iff s_2 \sim -s_2 - \mu_2.$$
Therefore, if $\gamma_1$ and $\gamma_2$ are not integers, no cancellation is possible. Consider now the case $\gamma_2 = 2s_2 + \mu_2 \in \mathbb{Z}$ while $\gamma_1 \notin \mathbb{Z}$. From the previous discussion, we only need to consider the potential cancellation of a zero of $1/\Gamma(s - s_2)$ with a pole of $\Gamma(s + s_2 + \mu_2)$.
  However it is impossible to find $k$ and $k'\in-\mathbb{N}$ such that $$s_2-k=-s_2-\mu_2-k'\in(-\mu_2,\mu_1],$$ 
as this would implies that $-\mu_2<s_2<0$ which is not possible by~Lemma~\ref{lemma:s1s2}. The other case is when $\gamma_2\notin\mathbb{Z}$ and $\gamma_1=-2s_1+\mu_1\in\mathbb{Z}$. In this case we find $k$ and $k'\in\mathbb{N}$ such that $$s_1-k=-s_1-\mu_2-k'\in(-\mu_2,\mu_1],$$ 
this would implies that $-\mu_2<s_1<0$ and $k=k'=0$ and therefore $s_1=-\mu_2/2$. This very special case is treated below.
\item In this point, we assume that $\varphi_1$ has a pole (\textit{i.e.}\ $-\mu_2 \leqslant s_1 < 0$) and we show that if $s_1\neq -\mu_2/2$ (\textit{i.e.}, $s_1\neq \zeta s_1$), then $1/D(s_1)=0$, which cancels the pole of $\varphi_1$ (if $\varphi_1$ admits a pole in the strip, then this pole is simple, according to Lemma~\ref{lemma:stripe}). Indeed, $\Gamma(s-s_1)$ has a pole at $s_1$, while $\Gamma(s_1-s_2)=\Gamma(-\gamma)<\infty$ (since $\gamma\notin\mathbb{Z}$) and $\Gamma(2s_1+\mu_2)<\infty$ since we assumed $-1<-\mu_2<2s_1+\mu_2<\mu_2<1$, which could be an integer only if $2s_1+\mu_2=0$, which is excluded by assumption.
\item If $s_1 = -\mu_2/2$, then there is a simplification of the Gamma functions and we have
$$\frac{\varphi_1(s)}{D(s)}=\varphi_1(s)\frac{\Gamma(s-s_2)}{\Gamma(s+s_2+\mu_2)}$$
which has, in $\mathfrak{B}_1\cup\{\mu_1\}$, a simple pole at $s_1 = \zeta s_1 = -\mu_2/2$ and a simple pole at $s_2-p$ for some $p\in\mathbb{N}$ when $s_2>0$. Remark that the zeros of $1/\Gamma(s+s_2+\mu_2)$ are $-s_2-\mu_2-k$ and cannot coincide with either $s_1= -\mu_2/2$ or $s_2-p$ since $\gamma_2\notin\mathbb{Z}$.
\end{itemize}

In conclusion, the function $\varphi_1/D$ has at most two simple poles in $\mathfrak{B}_1\cup \{ \mu_1 \}$, namely $s_2-p$ for some $p\in\mathbb N$ when $s_2>0$ and $\zeta s_1-q$ for some $q\in\mathbb N$ when $s_1<0$, and these poles merge into a unique double pole when $s_2-p=-s_1-\mu_2-q$, \textit{i.e.} when $s_2-\zeta s_1\in\mathbb Z$. 
\end{proof}

\begin{theorem}[Laplace transform, $\{\gamma_1,\gamma_2\}\not\subset\mathbb Z$ and $\gamma\notin \mathbb Z$]\label{thm:general}
Suppose $\{\gamma_1,\gamma_2\}\not\subset \mathbb Z$ and $\gamma\notin \mathbb Z$.
\begin{enumerate}
\item If $s_1<0$ and $s_2>0$ then 
\begin{equation}\label{eq:general1}
    \varphi_1(s)\propto \frac{D(s)}{\sin(\pi(s+s_1+\mu_2))\sin(\pi(s-s_2))}.
\end{equation}
    \item If $s_1<0$ and $s_2<0$ then
    \begin{equation}\label{eq:general2}
        \varphi_1(s)\propto D(s) \frac{\sin(\pi(s+s_2+\mu_2))}{\sin(\pi(s+s_1+\mu_2))}.
    \end{equation}
    \item If $s_1>0$ and $s_2>0$ then
\begin{equation}\label{eq:general3}
\varphi_1(s)\propto D(s)\frac{\sin(\pi(s-s_1))}{\sin(\pi(s-s_2))}.
\end{equation}
\end{enumerate}
\end{theorem} 
We recall it is not possible that $s_1>0$ and $s_2<0$: if $s_1>0$ then $s_1>\mu_1$ (according to Lemma~\ref{lemma:s1s2}) and then $\gamma_1:=\mu_1-2s_1<0$. Similarly, if $s_2<0$, then $s_2<-\mu_2$ and $\gamma_2:=\mu_2+2s_2<0$. Hence $s_1>0$ and $s_2<0$ would contradict the second item of Lemma~\ref{lemma:impossible}.
\begin{proof}
By Lemma~\ref{lemma:invariantDphi}, the function $\varphi_1/D$ is $1$-periodic and it is possible to apply the Type I Invariant Lemma (Proposition~\ref{lemma:periodic}). To do this we need to determine the poles of $\varphi_1/D$ in ${\mathfrak{B}_1}\cup \{\mu_1\}$ (it is sufficient to consider $\mu_1$ only, since the function is $1$-periodic and its behavior at $-\mu_2$ is therefore analogous to that at $\mu_1$).
This pole analysis is provided by Lemma~\ref{lemma:polephi1D} above: the only poles of $\varphi_1/D$ in ${\mathfrak{B}_1}\cup \{\mu_1\}$ are $s_2-p$ for some $p\in\mathbb N$ when $s_2>0$, and $\zeta s_1-q=-s_1-\mu_2-q$ for some $q\in\mathbb N$ when $s_1<0$. 
Note that we have seen that these poles can merge to give a double pole when $s_2-\zeta s_1\in\mathbb{Z}$. The \emph{specific cases} where one or both of them lie at $\mu_1$, \textit{i.e.} when $-s_1-\mu_2-q=\mu_1$ (\textit{i.e.}, $s_1\in-\mathbb N$) or when $s_2-p=\mu_1$ (\textit{i.e.}, $s_2\in\mu_1+\mathbb N$), are treated in Lemma~\ref{lemma:specific}. We now assume that we are in the \emph{generic cases} where these two poles cannot be equal to $\mu_1$, \textit{i.e.} $s_1\notin-\mathbb N$ and $s_2\notin\mu_1+\mathbb N$. Applying Proposition~\ref{lemma:periodic}, we obtain the following expression: 
\begin{equation}\label{eq:notN}
\frac{\varphi_1(s)}{D(s)}=c+\frac{c_1 \mathds1_{s_1<0}}{\wb(s)-\wb(\zeta s_1)}+\frac{c_2 \mathds1_{s_2>0}}{\wb(s)-\wb(s_2)} +\frac{c_3 \mathds1_{s_1<0}\mathds1_{s_2>0} \mathds1_{s_2-\zeta s_1\in\mathbb{Z}}}{(\wb(s)-\wb(s_2))^2}.
\end{equation}
Roughly speaking, $c_1$ and $c_2$ encode the potential poles at $\zeta s_1-q$ and $s_2-p$, while $c_3$ encodes the times when these poles merge into a double pole. The rest of the proof consists, in each case, of studying limits involving $\varphi_1$ to establish linear relations among the remaining constants, and of using the following trigonometric identities to derive the claimed formulas:
\begin{align*}\label{eq:trigo}
   &(\text{\textsc{a}})\;\; 1+\tan(x)^2=\frac{1}{\cos(x)^2}, &&(\text{\textsc{b}})\;\;\tan(x_1)-\tan(x_2)=\frac{\sin(x_1-x_2)}{\cos(x_1)\cos(x_2)}.%\\[0.3cm]
   %&(\text{\textsc{c}})\;\; \cos(x\pm\pi/2) =\mp \sin(x),  &&(\text{\textsc{d}})\;\; \sin(x\pm \pi/2)= \pm\cos(x).
\end{align*}
    \begin{enumerate}
    \item Suppose $s_1<0$ and $s_2>0$. We need to distinguish two cases depending on whether the two poles in $\overline{\mathfrak{B}}_1$ (as given by Lemma~\ref{lemma:polephi1D}) merge. Suppose for now that $s_2-\zeta s_1 \notin \mathbb{Z}$, so that $\wb(\zeta s_1)\ne \wb(s_2)$ and we have
        $$\frac{\varphi_1(s)}{D(s)}=c+\frac{c_1}{\wb(s)-\wb(\zeta s_1)}+\frac{c_2}{\wb(s)-\wb(s_2)}.$$
By  Lemma~\ref{lemma:asymptotic}, $1/D(s)\sim s^{-2\gamma}\to 0$, hence by Lemma \ref{lemma:stripe} $\varphi_1/D$ goes to $0$ as $|s|\to+\infty$. Evaluating the above expression at $s=it$ and taking the limit as $t\to\pm \infty$ produces two linear equations satisfied by $c$, $c_1$ and $c_2$:
        $$0=c+\frac{c_1}{i-\wb(\zeta s_1)}+\frac{c_2}{i-\wb(s_2)}=c-\frac{c_1}{i+\wb(\zeta s_1)}-\frac{c_2}{i+\wb(s_2)}.$$
Once again, if $c=0$ then $c_1=c_2=0$ is the unique solution to the above equation, thus $c\ne 0$. This time, solving for $c_1$ and $c_2$ yields
        $$\frac{c_1}{c}=\frac{1+\wb(\zeta s_1)^2}{\wb(\zeta s_1)-\wb(s_2)}\text{ and }\frac{c_2}{c}=\frac{1+\wb(s_2)^2}{\wb(s_2)-\wb(\zeta s_1)}.$$
Putting all these quantities on the same denominator leads to
        \begin{align*}
            \frac{\varphi_1(s)}{D(s)}&\propto 1+\frac{1+\wb(\zeta s_1)^2}{\Big(\wb(\zeta s_1)-\wb(s_2)\Big)\Big(\wb(s)-\wb(\zeta s_1)\Big)}+\frac{1+\wb(s_2)^2}{\Big(\wb(s_2)-\wb(\zeta s_1)\Big)\Big(\wb(s)-\wb(s_2)\Big)}\\
            &=\frac{1+\wb(s)^2}{\Big(\wb(s)-\wb(\zeta s_1)\Big)\Big(\wb(s)-\wb(s_2)\Big)},
        \end{align*}
        and finally
        \begin{align*}
        \frac{\varphi_1(s)}{D(s)}\!\!\overset{\substack{(\text{\textsc{a}})\,+\,(\text{\textsc{b}})\\[0.1cm]\downarrow\\ \phantom{x}}}{\propto}\!\!\frac{\cos\Big(\pi(\zeta s_1+\mu_2-\frac12)\Big)\cos\Big(\pi(s_2+\mu_2-\frac12)\Big)}{\sin\Big(\pi(s-\zeta s_1)\Big)\sin\Big(\pi(s- s_2)\Big)}\propto \frac{1}{\sin\Big(\pi(s-\zeta s_1)\Big)\sin\Big(\pi(s- s_2)\Big)},
        \end{align*}
which is the claimed result. 

If however $s_2-\zeta s_1 \in \mathbb{Z}$ then $\wb(\zeta s_1)=\wb(s_2)$. We can take $c_1=0$ and we have
        \begin{equation}\label{eq:case1a}
            \frac{\varphi_1(s)}{D(s)}=c+\frac{c_2}{\wb(s)-\wb(s_2)}+\frac{c_3}{(\wb(s)-\wb(s_2))^2}.
        \end{equation}
Applying the same method as above, one can show that the constants $c$, $c_2$ and $c_3$ satisfy the following linear constraints:
\begin{equation}\label{eq:linearsystem}
0=c+\frac{c_2}{i-\wb(s_2)}+\frac{c_3}{(i-\wb(s_2))^2}=c-\frac{c_2}{i+\wb(s_2)}+\frac{c_3}{(i+\wb(s_2))^2}.
\end{equation}
If $c=0$, then solving the above relation would imply that $c_2=c_3=0$, and hence $\varphi_1\equiv 0$, which contradicts the positivity of the invariant measure. Therefore, $c\ne 0$ and solving for $c_2$ and $c_3$ yields
        $$\frac{c_2}{c}=2\wb(s_2)\text{ and }\frac{c_3}{c}=1+\wb(s_2)^2.$$
        By plugging these values back into~\eqref{eq:case1a}, one obtains
        $$\frac{\varphi_1(s)}{D(s)}\propto 1+\frac{2\wb(s_2)}{\wb(s)-\wb(s_2)}+\frac{1+\wb(s_2)^2}{(\wb(s)-\wb(s_2))^2}=\frac{1+\wb(s)^2}{(\wb(s)-\wb(s_2))^2}.$$
Now, we replace $\wb$ by its definition and apply the trigonometric relations:
        $$\frac{\varphi_1(s)}{D(s)}\overset{\substack{(\text{\textsc{a}})\,+\,(\text{\textsc{b}})\\[0.1cm]\downarrow\\ \phantom{x}}}{\propto} \frac{\cos\Big(\pi(s_2+\mu_2-\frac12)\Big)^2}{\sin\Big(\pi(s-s_2)\Big)^2}\propto \frac{1}{\sin\Big(\pi(s-s_2)\Big)^2},$$ which coincides with the claimed result when $s_2=\zeta s_1+n$ for some $n\in\mathbb Z$. 
        \item Suppose $s_1<0$ and $s_2<0$. In this case, 
        $$\frac{\varphi_1(s)}{D(s)}=c+\frac{c_1}{\wb(s)-\wb(\zeta s_1)}.$$
Moreover, $\gamma_1:=\mu_1-2s_1>0$ and $\gamma_2:=\mu_2+2s_2<\mu_2$, so either $\gamma_2<0$ or $\gamma_2\in [0,\mu_2)$, which would imply that $s_2\in [-\mu_2,0)$ and contradict Lemma~\ref{lemma:s1s2}. Therefore, $\gamma_1$ and $\gamma_2$ have opposite signs, and Lemma~\ref{lemma:impossible} ensures that $\gamma<0$. As a consequence of Lemma~\ref{lemma:asymptotic}, $D(s)\sim s^{2\gamma}\to 0$. Recall that $d_+:=s_++i\mathbb R$. This line is symmetric with respect to the real axis, and for all $s\in d_+$, $\boldsymbol{\mathrm{y}}(s)=\boldsymbol{\mathrm{y}}(\overline{s})$ and $\varphi_1(s)=\varphi_1(\overline{s})$. Hence
        $$1=\lim_{t\to +\infty} \frac{\varphi_1(s_++it)}{\varphi_1(s_+-it)}=\left(\lim_{t\to+\infty}\frac{D(s_++it)}{D(s_+-it)}\right) \left(c+\frac{c_1}{i-\wb(\zeta s_1)}\right)\left(\displaystyle c-\frac{c_1}{i+\wb(\zeta s_1)}\right)^{-1}.$$
        To compute the remaining limit, note that, according to Lemma~\ref{lemma:asymptotic}, as $t\to +\infty$, $$D(s_+\pm it)\sim (s_++it)^{2\gamma}\sim t^{2\gamma}(\pm i)^{2\gamma} =t^{2\gamma}\exp(\pm \gamma i \pi),$$
        so that 
        \begin{equation}\label{eq:D+/D-}
            \lim_{t\to +\infty}\frac{D(s_++it)}{D(s_+-it)}=e^{2i\pi \gamma}.
        \end{equation}
Finally, we obtain a linear equation linking $c$ and $c_1$,
$$e^{2i\pi \gamma}\left(c+\frac{c_1}{i-\wb(\zeta s_1)}\right)=c-\frac{c_1}{i+\wb(\zeta s_1)}.$$
Just like in the other cases, $c$ cannot be $0$ (otherwise, $\varphi_1\equiv 0$). The solutions to this linear equation satisfy
$$
\frac{c_1}{c}=\frac{(1+\wb(\zeta s_1)^2)(e^{-2i\pi \gamma}-1)}{\wb(\zeta s_1)(e^{-2i\pi \gamma}-1)-i(1+e^{-2i\pi \gamma})}=\frac{(1+\wb(\zeta s_1)^2)(e^{-i\pi \gamma}-e^{i\pi\gamma})}{\wb(\zeta s_1)(e^{-i\pi \gamma}-e^{i\pi \gamma})-i(e^{i\pi\gamma}+e^{-i\pi\gamma})}.%=\frac{\sin(\pi \gamma)}{\sin(\pi s_1)\sin(\pi s_2)}
$$
Applying Euler's formula for sine and cosine leads to
\begin{align*}
\frac{c_1}{c}=\frac{(1+\wb(\zeta s_1)^2)\sin(\pi\gamma)}{\wb(\zeta s_1)\sin(\pi\gamma)+\cos(\pi\gamma)}&\overset{\substack{(\text{\textsc{a}})\\[0.1cm]\downarrow\\ \phantom{x}}}{=}\frac{\sin(\pi\gamma)}{\cos\Big(\pi(\zeta s_1+\mu_2-\frac12)\Big)\cos\Big(\pi(\zeta s_1+\mu_2-\frac12-\gamma)\Big)}\\
&\,=\frac{\sin(\pi \gamma)}{\sin(\pi s_1)\sin(\pi s_2)},%=\frac{-\sin(\pi\gamma)}{\sin(\pi s_1)\cos(\pi(s_2+\mu_2))}
\end{align*}
and then\\[-0.5cm]
$$\frac{\varphi_1(s)}{D(s)}\propto 1+\frac{\sin(\pi \gamma)}{\sin(\pi s_1)\sin(\pi s_2)}\cdot\frac{1}{\wb(s)-\wb(\zeta s_1)}\overset{\substack{(\text{\textsc{b}})\\[0.1cm]\downarrow\\ \phantom{x}}}{=}1-\frac{\sin(\pi\gamma)\sin(\pi(s+\mu_2))}{\sin(\pi s_2)\sin(\pi(s-\zeta s_1))}.$$
(where we also used the fact that $\cos(\,\cdot\,-\pi/2)=\sin$). Finally, one can show using addition formulas that 
$$\frac{\varphi_1(s)}{D(s)}=\frac{\sin(\pi s_1)}{\sin(\pi s_2)}\cdot \frac{\sin\Big(\pi(s+s_2+\mu_2)\Big)}{\sin\Big(\pi(s+s_1+\mu_2)\Big)}\propto \frac{\sin\Big(\pi(s+s_2+\mu_2)\Big)}{\sin\Big(\pi(s+s_1+\mu_2)\Big)}.$$
        \item Suppose $s_1>0$ and $s_2>0$. In this case
        $$\frac{\varphi_1(s)}{D(s)}=c+\frac{c_2}{\wb(s)-\wb(s_2)}.$$
        Replacing $c_{1}$ with $c_{2}$ and $\wb(\zeta s_1)$ with $\wb(s_2)$ in the above argument immediatly yields
$$\frac{c_{2}}{c}=\frac{(1+\wb(s_2)^2)(e^{-2i\pi \gamma}-1)}{\wb(s_2)(e^{-2i\pi \gamma}-1)-i(1+e^{-2i\pi \gamma})}=\frac{\sin(\pi \gamma)}{\sin(\pi(s_2+\mu_2))\sin(\pi(s_1+\mu_2))},$$
and leads to
    $$\frac{\varphi_1(s)}{D(s)}\propto 1+\frac{\sin(\pi\gamma)}{\sin(\pi(s_2+\mu_2))\sin(\pi(s_1+\mu_2))}\cdot \frac{1}{\wb(s)-\wb(s_2)}=\frac{\sin(\pi(s_2+\mu_2))\sin(\pi(s-s_1))}{\sin(\pi(s_1+\mu_2))\sin(\pi(s-s_2))},$$
    which concludes the proof of the last case.
    \end{enumerate}
\end{proof}

\begin{lemma}[Specific cases]\label{lemma:specific} Under the assumptions of Theorem~\ref{thm:general}, and the additional condition that $s_1\in -\mathbb N$ or $s_2\in \mu_1+\mathbb N$, the conclusion of Theorem~\ref{thm:general} also holds.
\end{lemma}
\begin{proof} In each case, we determine an expression similar to that in~\eqref{eq:notN}, then compute the relations between the constants as in the proof of the generic cases, and finally verify that the resulting formula matches the one given in Theorem~\ref{thm:general}.
   \begin{itemize}
       \item If $s_1\in -\mathbb N$ and $s_2<0$ \textbf{or} $s_2\in \mu_1+\mathbb N$ and $s_1>0$ then $\varphi_1/D$ has a unique simple pole at $\mu_1$, and by the Type I invariant Lemma (Proposition~\ref{lemma:periodic}) there exist constants $c$ and $c_{1}$ such that
       $$\frac{\varphi_1(s)}{D(s)}=c+c_{1}\wb(s).$$  
        As already shown in the corresponding generic case, $\gamma<0$. Applying the same method as in 2. and 3. in the previous theorem, one can show that
        $$e^{2i\pi \gamma}(c+ic_{1})=c-ic_{1}.$$
        Therefore $c_{1}=-c\tan(\pi\gamma)$ and $$\varphi_1(s)\propto D(s)\left(1-\tan(\pi\gamma)\wb(s)\right).$$ Now we need to check that the corresponding generic cases match this expression: if $s_1<0$ and $s_2<0$, then taking $s_1=-n\in -\mathbb N$ in the generic case~\eqref{eq:general2} leads to
        \begin{align*}
        \frac{\varphi_1(s)}{D(s)}&\propto \frac{\sin(\pi(s+\mu_2+s_2))}{\sin(\pi(s+\mu_2-n))}=\frac{\sin(\pi(s+\mu_2+\gamma))}{\sin(\pi(s+\mu_2))}\\
        &=\cos(\pi\gamma)+\frac{\sin(\pi\gamma)}{\tan(\pi(s+\mu_2))}\propto 1-\tan(\pi\gamma)\wb(s),
        \end{align*}
        (as a simple consequence of the addition formula for the sine function). If $s_1>0$ and $s_2>0$ , then taking $s_2=\mu_1+n$ (or equivalently $s_1=\mu_1+n-\gamma$) with $n\in\mathbb N$ in the generic case~\eqref{eq:general3} also yields
        $$\frac{\varphi_1(s)}{D(s)}\propto \frac{\sin(\pi(s-\mu_1-n+\gamma))}{\sin(\pi(s-\mu_1-n))}=\cos(\pi\gamma)\Big(1-\tan(\pi\gamma)\wb(s)\Big).$$
        This confirms consistency with both generic cases.
       \item If $s_1\in -\mathbb N$, $s_2>0$ and $s_2\notin \mu_1+\mathbb N$, then $\varphi_1/D$ has two simple poles in the strip: one at $\mu_1$ and another one at $s_2-q$ for some $q\in\mathbb N$. Thus 
$$\frac{\varphi_1(s)}{D(s)}=c+\frac{c_1}{\wb(s)-\wb(s_2)} +c_2\wb(s)$$       
       for well-choosen constants $c$, $c_1$ and $c_2$. Moreover, since $\gamma>0$, $\varphi_1/D\to 0$. Taking the limits as in the first generic case yields the system
        $$0=c+\frac{c_1}{i-\wb(s_2)}+ic_2=c-\frac{c_1}{i+\wb(s_2)}-ic_2$$
        whose nonzero solutions satisfy
        $$\frac{c_1}{c}=\wb(s_2)+\frac{1}{\wb(s_2)}\text{ and }\frac{c_2}{c}=\frac{1}{\wb(s_2)}.$$
        Hence
        \begin{align*}
        \frac{\varphi_1(s)}{D(s)}&\propto 1+\left(\wb(s_2)+\frac{1}{\wb(s_2)}\right)\frac{1}{\wb(s)-\wb(s_2)}+\frac{\wb(s)}{\wb(s_2)}\propto\frac{1}{\sin(\pi(s+\mu_2))\sin(\pi(s-s_2))}.
        \end{align*}
        This expression is to be compared with~\eqref{eq:general1}: replacing $s_1$ with $-n\in-\mathbb N$ in that formula yields the one derived above, with no further trigonometric computation.
       \item If $s_2\in \mu_1+\mathbb N$, $s_1<0$ and $s_1\notin \mathbb N$ then applying the same symmetry principle as in the proof of the corresponding generic case yields the following result that is consistent with Theorem~\ref{thm:general}:
       $$\varphi_1(s)\propto D(s)\frac{\sin(\pi s_1)\tan(\pi s_1)}{\sin(\pi(s+\mu_2))\sin(\pi(s+s_1+\mu_2))}$$
       \item If $s_1\in -\mathbb N$ \textbf{and} $s_2\in \mu_1+\mathbb N$, then $\varphi_1/D$ has a unique double pole at $\mu_1$, and for some constants $c$, $c_1$ and $c_2$,
        $$\frac{\varphi_1(s)}{D(s)}=c+c_1\w(s)+c_2\w(s)^2.$$
        Since $\gamma>0$, $\varphi_1(s)/D(s)\to 0$ when $|s|\to +\infty$. Taking the limits as in the first generic case shows that
        $$0=c+ic_1-c_2=c-ic_1-c_2.$$
        Taking the difference of the two equations immediatly gives $c_1=0$ and then $c=c_2$:
        $$\frac{\varphi_1(s)}{D(s)}\propto 1+\wb(s)^2=\frac{1}{\sin(\pi(s+\mu_2))^2}.$$
        This result matches with~\eqref{eq:general1}: let us evaluate this generic case at $(s_1,s_2)=(-m,\mu_1+n)$ for two natural numbers $m$ and $n$:
        $$\frac{\varphi_1(s)}{D(s)}\propto \frac{1}{\sin(\pi(s+\mu_2-m))\sin(\pi(s-\mu_1-n))}\propto \frac{1}{\sin(\pi(s+\mu_2))^2},$$
        where we relied on the fact that $\mu_1+\mu_2=1$ by~\eqref{eq:H4}.
   \end{itemize}
\end{proof}

\begin{proposition}[Consistency with rational decoupling cases]
\label{prop:consistency}
The formulas of Theorem~\ref{thm:general} remain valid in the rational decoupling cases: more precisely, imposing $\gamma \in \mathbb{Z}$ or $\{\gamma_1,\gamma_2\} \subset \mathbb{Z}$ in the expressions of Theorem~\ref{thm:general} reproduces exactly the formulas given in Theorems~\ref{thm:N}, \ref{thm:-N}, and~\ref{thm:notN}.
\end{proposition}

\begin{proof}
The proof proceeds by a direct verification of the formulas in Theorem~\ref{thm:general} under the rational decoupling conditions. We split the proof into three cases corresponding to those of Theorem~\ref{thm:general}.
\begin{enumerate}
\item If $s_1<0$ and $s_2>0$, then, $\gamma_1:=\mu_1-2s_1>0$, $\gamma_2:=\mu_1+2s_2>0$, $\gamma:=s_2-s_1>0$  and according to Theorem~\ref{thm:general},
\begin{equation}\label{eq:proofconsist1}
\frac{\varphi_1(s)}{D(s)}\propto \frac{1}{\sin\left(\pi(s+s_1+\mu_2)\right)\sin\left(\pi(s-s_2)\right)}.
\end{equation}
There are two rational decoupling cases compatible with these conditions: either $\gamma \in \mathbb{N}$ or $\{\gamma_1, \gamma_2\} \subset \mathbb{N}$. If $\gamma\in\mathbb N$, then Equation~\eqref{eq:proof44} in the proof of Theorem~\ref{thm:-N} shows that
\begin{equation}\label{eq:proofconsist2}
\frac{\varphi_1(s)}{D(s)}\propto \frac{1}{\cos\left(2\pi(s-s_-)\right)+\cos(\pi\gamma_1)}.
\end{equation}
We apply the standard trigonometric identity $\sin(a)\sin(b)=\frac{1}{2}(\cos(a-b)-\cos(a+b))$ to get 
\begin{align*}
&\phantom{=}\;\sin\left(\pi(s+s_1+\mu_2)\right)\sin\left(\pi(s-s_2)\right)\\
&\propto \cos\left(\pi(s_1+s_2+\mu_2)\right)-\cos\left(\pi(2s+s_1-s_2+\mu_2)\right)\\
&=\cos\left(\pi(-\gamma_1+\gamma+1)\right)-\cos\left(2\pi(s-s_1)+\pi\gamma\right)\\
&\propto \cos(2\pi(s-s_-))-\cos(\pi\gamma_1),\text{ since (}\gamma\in\mathbb N\text{),}
\end{align*}
which proves the consistency between~\eqref{eq:proofconsist1} and~\eqref{eq:proofconsist2}. If $\{\gamma_1,\gamma_2\}\subset \mathbb N$, then Equation~\eqref{eq:aux1} in the proof of Theorem~\ref{thm:notN} states that
\begin{equation}\label{eq:proofconsist3}
\frac{\varphi_1(s)}{D(s)}\propto \frac{1}{\sin\left(2\pi(s-s_-)\right)}.
\end{equation}
One can replace $s_1$ and $s_2$ with $\frac{1}{2}(\mu_1-\gamma_1)$ and $\frac{1}{2}(\gamma_2-\mu_2)$ respectively in~\eqref{eq:proofconsist1} to obtain
\begin{align*}
\sin\left(\pi(s+s_1+\mu_2)\right)\sin\left(\pi(s-s_2)\right)&=\sin\left(\pi\left(s+\frac{\mu_1}{2}-\frac{\gamma_1}{2}+\mu_2\right)\right)\sin\left(\pi\left(s-\frac{\gamma_2}{2}+\frac{\mu_2}{2}\right)\right)\\
&=\sin\left(\pi\left(s-s_-\right)-\frac{\pi(\gamma_1-1)}{2}\right)\sin\left(\pi(s-s_-)-\frac{\pi\gamma_2}{2}\right)\\
&\propto \sin\left(\pi(s-s_-)\right)\cos\left(\pi(s-s_-)\right)\propto \sin\left(2\pi(s-s_-)\right).
\end{align*}
The last step in the above computation relies on the fact that $\gamma_1$ and $\gamma_2$ share the same parity (see Remark~\ref{rem:gamma}).
\item If $s_1<0$ and $s_2<0$, then $\gamma_1>0$ (by definition) and $\gamma_2<\mu_2$. If $\gamma_2>0$, then $s_2\in(-\mu_2/2,0)$, which contradicts Lemma~\ref{lemma:s1s2}. Hence $\gamma_2<0$, and Lemma~\ref{lemma:impossible} implies that $\gamma>0$. There are two rational decoupling cases satisfying these inequalities: either $\gamma\in\mathbb{N}$, or $\gamma_1,-\gamma_2\in\mathbb{N}$. One should therefore impose $\gamma\in\mathbb{N}$ in~\eqref{eq:proof44} and $(\gamma_1,\gamma_2)\in\mathbb{N}\times -\mathbb{N}$ in the proof of Theorem~\ref{thm:notN} (distinguishing two subcases according to the parity of $\gamma_1$) to recover~\eqref{eq:general2}. The remainder of the proof only involves standard trigonometric reasoning.
\item If $s_1>0$ and $s_2>0$, then by an argument similar to the previous case we obtain 
$\gamma_1<0$, $\gamma_2>0$, and $\gamma>0$. Under these conditions, $D$ is rational if and only if $\gamma\in\mathbb{N}$ or $-\gamma_1,\gamma_2\in\mathbb{N}$. As in the preceding case, basic trigonometry suffices to verify the consistency of the corresponding formulas.
\end{enumerate}
\end{proof}

\section{Differential hierarchy for the Laplace transform}\label{sec:hierarchy}

The goal of this section is to prove the classification given in the introduction by Table~\ref{tab:classification} which gives necessary and sufficient condition for the Laplace transforms $\phi$, $\phi_1$ and $\phi_2$ to belong to the hierarchy 
\begin{equation}\label{eq:hierarchy}
\text{rational}\subset\text{algebraic}\subset\text{D-finite}\subset\text{D-algebraic} .
\end{equation}
Let us now give more precise definitions of these classes of functions, in both the univariate and bivariate context:
\begin{definition}[differential and algebraic nature]\label{def:hierarchy} Let $U \subseteq \mathbb C^d$ be a domain and let 
$f:U\to\mathbb C$ be a meromorphic function. 
We say that $f$ is \textit{differentially algebraic} 
(abbreviated \textit{D-algebraic}) if there exists 
$n \in \mathbb N_0$ and, for each $1 \le k \le d$, 
a nonzero polynomial 
$$P_k \in \mathbb C(z_k)[X_0,\dots,X_n]$$
such that 
$$P_k\bigl(f, \partial_{z_k} f,\dots,\partial_{z_k}^{\,n} f\bigr) = 0.$$
We say that $f$ is \textit{differentially transcendental} 
(abbreviated \textit{D-transcendental}) if it is not 
differentially algebraic, and \textit{differentially finite} (abbreviated 
\textit{D-finite}, also called \textit{holonomic}) if, for each 
$k$, the corresponding polynomial $P_k$ is linear in the 
indeterminates $X_0,\dots,X_n$.\\
The function $f$ is said to be \textit{algebraic} if it satisfies 
a non-trivial polynomial relation
$$Q(z_1,\dots,z_d,f)=0,\qquad 
Q \in \mathbb C[z_1,\dots,z_d,X].$$
Note that in the definition of D-algebraic functions, it is equivalent to require the polynomial differential equations having coefficient in $\mathbb{R}$.
\end{definition}
We will show that $\phi_1$ is rational if and only if it is algebraic, if and only if it is $D$-finite. We will see that $\phi_1$ is $D$-algebraic if and only if there exists a decoupling function for $G$. In Section~\ref{subsec:subcases}, we will show that if $\gamma\in\mathbb{Z}$ or  $\{ \gamma_1,\gamma_2\}\subset \mathbb{Z}$ then $\phi_1$ is $D$-algebraic and that in these algebraic cases, $\phi_1$ is $D$-finite if and only if $\gamma\in-\mathbb{N}$. Due to the symmetries in the problem and in our results, $\phi_1$ and $\phi_2$ will share the same nature. The following proposition allows us to deduce that of $\phi$.
\begin{proposition}[Common nature] If $\phi_1$ and $\phi_2$ both belong to the same class of the hierarchy~\eqref{eq:hierarchy}, then so does $\phi$.
\end{proposition}

\begin{proof} By the functional equation~\eqref{eq:FE}
$$\phi(x,y)=-\frac{k_1(x,y)\phi_1(y)+k_2(x,y)\phi_2(x)}{K(x,y)},$$
where $k_1$, $k_2$ and $K$ are polynomials, given in~\eqref{eq:noyau} and~\eqref{eq:k12}. The conclusion comes from the fact that each class of the hierarchy~\eqref{eq:hierarchy} is a $\mathbb{C}(x,y)$-vector space.
\end{proof}

In Section~\ref{subsec:transcendance} we will show that if $\phi_1$ is $D$-algebraic then $\gamma\in\mathbb{Z}$ or  $\{ \gamma_1,\gamma_2\}\subset \mathbb{Z}$, which is the same thing, to show that if $\gamma\notin\mathbb{Z}$ and $\{ \gamma_1,\gamma_2\}\not\subset \mathbb{Z}$ then $\phi_1$ is $D$-transcendental.
This will complete the proof of the above,  since the remaining cases will be treated in Section \ref{subsec:subcases}.

\subsection{D-finite, algebraic and rational cases}\label{subsec:subcases}

\begin{lemma}\label{lem:wDADF} For all $a,b,c\in \mathbb R$ (with $a\ne 0$), the following functions are D-algebraic but not D-finite:
$$g_1(y):=\frac{1}{\cos(\sqrt{ay+b})-c},\quad  g_2(y):=\frac{\sqrt{ay+b}}{\sin\left(\sqrt{ay+b}\right)},\quad  g_3(y):=\frac{\sqrt{ay+b}}{\tan\left(\sqrt{ay+b}\right)},\quad  g_4:=\frac{1}{g_3} $$
\end{lemma}
\begin{proof}
Notice that $\cos$, $\sin$ $\tan$ and $\sqrt{\,\cdot \,}$ are algebraic and therefore D-algebraic. Recall from~\cite{Dalgebraic} that the class of D-algebraic functions is a field and is stable under composition (provided the composition is defined), hence $g_1$, $g_2$, $g_3$ and $g_4$ are D-algebraic. To check wether these functions are D-finite, we will use the fact that D-finite functions have finitely many singularities. Here, $g_1$, $g_2$, $g_3$ and $g_4$ clearly violates this rule, hence cannot be D-finite.
\end{proof}

\begin{proposition}[Sufficient condition for $D$-algebraicity]
\label{prop:CSDA}
If $\gamma\in\mathbb{Z}$ or  $\{ \gamma_1,\gamma_2\}\subset \mathbb{Z}$ then $\phi_1$ is $D$-algebraic. 
\end{proposition}
\begin{proof}
Theorems~\ref{thm:N},~\ref{thm:-N} and~\ref{thm:notN} in the previous section show that when $\gamma\in\mathbb{Z}$ or  $\{ \gamma_1,\gamma_2\}\subset \mathbb{Z}$,
$$
\phi_1(y)=\frac{P(y)}{Q(y)}g(y)
$$
where $P$ and $Q$ are polynomials coming from the decoupling function $D$, and $g$ is one of the D-algebraic functions listed in Lemma~\ref{lem:wDADF} (or $g\equiv 1$ when $\gamma\in-\mathbb N$). As for every rational function, $P/Q$ is also D-alebraic, so that their product $\phi_1$ is D-algebraic.
\end{proof}

\begin{proposition}[Sufficient condition for rationality]
If $\gamma\in-\mathbb{N}$ then $\phi_1$ is rational.
\end{proposition}
\begin{proof}
If $\gamma\in-\mathbb{N}$, Theorem~\ref{thm:N} shows that $\phi_1=1/Q$ is rational. 
\end{proof}

It just remains to show that in these $D$-algebraic cases, if $\gamma\notin-\mathbb{N}$ then $\phi_1$ is not $D$-finite (and then neither algebraic nor rational).
We will need the following lemma. 

\begin{proposition}[Necessary condition for D-finiteness]
If $\gamma\in\mathbb{Z}$ or  $\{ \gamma_1,\gamma_2\}\subset \mathbb{Z}$, and if $\gamma\notin-\mathbb{N}$, then $\phi_1$ is not $D$-finite (and then neither algebraic nor rational).
\end{proposition}
This proposition is equivalent to say that if $\gamma\in\mathbb{Z}$ or  $\{ \gamma_1,\gamma_2\}\subset \mathbb{Z}$, then $\phi_1$ is $D$-finite implies that $\gamma\in-\mathbb{N}$ (and is then rational by the previous proposition).
\begin{proof}
The class of D-finite functions is closed under multiplication (see~\cite{algoDfinite} for proofs and algorithms to construct the differential equation satisfied by the product of two D-finite functions) and contains all rational functions. Moreover, Lemma~\ref{lem:wDADF} asserts that, under the rational decoupling condition \eqref{eq:decouplcond}, if $\gamma\notin -\mathbb N$ then $\phi_1$ is the product of a rational function $P/Q$ by a nonholonomic function $g$. If $\phi_1$ was D-finite, one could show that $Q\phi_1/P=g$ is also D-finite, which is a contradiction.
\end{proof}

\subsection{Differential transcendance and D-algebraic cases}\label{subsec:transcendance}

We have already shown in the previous section that if $\gamma\in\mathbb{Z}$ or  $\{ \gamma_1,\gamma_2\}\subset \mathbb{Z}$ then $\phi_1$ is $D$-algebraic. In this section we want to show the reciprocal. We will have shown that $\phi_1$ is $D$-transcendental if and only if $\gamma\notin\mathbb{Z}$ and $\{ \gamma_1,\gamma_2\}\not\subset \mathbb{Z}$. This conclusion will follow from a result from difference Galois  theory, which we will use here as a black box, see Remark~\ref{rem:galois} below. 
The proposition below is a direct consequence of \cite{HS} (Corollary 3.4). 

\begin{proposition}[Differential algebraicity and difference equation \cite{HS} (Corollary 3.4)]
\label{cor3.4}
Let $g(x) \in \mathbb{C}(x)$.
If $u(x)$ is a nonzero function meromorphic on $\mathbb{C}$ satisfying $u(x+1)=g(x) u(x)$ and $u(x)$ is differentially algebraic over $\mathbb{C}(x)$ then $g(x)=c \frac{d(x+1)}{d(x)}$ for some nonzero $c \in \mathbb{C}$ and $d(x) \in \mathbb{C}(x)$.
\end{proposition}

\begin{proposition}[Necessary condition for $D$-algebraicity]
\label{prop:CNDA}
If $\phi_1$ is $D$-algebraic then there exist a rationnal decoupling and $\gamma\in\mathbb{Z}$ or  $\{ \gamma_1,\gamma_2\}\subset \mathbb{Z}$.
\end{proposition}
\begin{proof}
We assume that $\phi_1$ is $D$-algebraic which directly implies that $\varphi_1(s)=\phi_1(\y(s))$ is also $D$-algebraic by composition of D-algebraic functions. Let us recall the difference equation~\eqref{eq:qdiff} of Theorem~\ref{thm:prolong}: 
$$\varphi_1(s+1)=G(s)\varphi_1(s).$$ 
The Proposition~\ref{cor3.4} then implies that there exists $c\in \mathbb{C}^*$ and a nonzero rational function $d$ such that $G(s)=c \frac{d(s+1)}{d(s)}$. Taking the limit at infinity, we see that $\lim_\infty G=1$ and necessarily we have $c=1$ and we deduce that $d$ is a rational decoupling function for $G$. We conclude with Proposition~\ref{prop1} which states that if $G$ admits a rational decoupling function then $\gamma\in\mathbb{Z}$ or $\{ \gamma_1,\gamma_2\}\subset \mathbb{Z}$.
\end{proof}

\begin{remark}[Necessary and sufficient condition for $D$-algebraicity and decoupling]\label{rem:CNSdec}
Combining Proposition~\ref{prop:CSDA}, Proposition~\ref{prop:CNDA} and Proposition~\ref{prop1} we have thus shown that $\phi_1$ is differentially algebraic if and only if $G$ admits a rational decoupling function if and only if $\gamma\in\mathbb{Z}$ or  $\{ \gamma_1,\gamma_2\}\subset \mathbb{Z}$. Similar statements can be found in the classification of generating series of walks in the quarter plane, see \cite{dreyfus2021differential,hardouin2021differentially}.
\end{remark}

\section{Explicit expressions of the lateral density}\label{sec:density}
\noindent\textbf{Notation.} In this section, we denote by $\mathcal{L}$ the Laplace transform operator. For instance, $\phi_1 = \mathcal{L}\nu_1$.

\subsection{Rational cases $\gamma\in-\mathbb{N}$}\label{subsec:fractionalDO}

The aim of this section is to invert the Laplace transforms found in Section~\ref{sec:Laplace} to obtain the lateral invariant measure. Among all these expressions for $\phi_1$ found in this section, one can be easily inverted.

\begin{theorem}[Lateral measure, $\gamma\in-\mathbb N$]\label{thm:densityrat} If $\gamma\in-\mathbb N$, then $\nu_1$ is given by the following sum-of-exponentials
$$\nu_1(v)\propto\sum_{k=0}^{-\gamma-1}c_k\exp\left(-\y(s_1-k)v\right)\text{ where }c_k:=\prod_{\substack{j=0\\ j\ne k}}^{-\gamma-1}\frac{1}{\y(s_1-k)-\y(s_1-j)}.$$
\end{theorem}
\begin{proof} From Theorem~\ref{thm:N}, we know that $\phi_1(y) \propto 1/Q(y)$, where the roots of $Q$ are the values $\y(s_1 - k)$ for $k = 0, \dots, -\gamma - 1$. We claim that $Q$ has no double roots. Suppose, for contradiction, that there exist $0 \leq i < j \leq -\gamma - 1$ such that $\y(s_1 - i) = \y(s_1 - j)$. Then $s_1 - i = \eta(s_1 - j)$, which is equivalent to $i + j = -\gamma_1$. Hence, 
$$2\gamma+3= (\gamma+2)+(\gamma+1)\leqslant -i-j=\gamma_1 \leqslant -1<0.$$
According to Lemma~\ref{lemma:impossible}, $\gamma_1$ and $\gamma_2$ cannot both be negative simultaneously, so $ \gamma_2 > 0$. But then:
$$
2\gamma+3=\gamma_1+\gamma_2+2>\gamma_1+2
$$
which leads to the contradiction $\gamma_1>\gamma_1+2$. Therefore, the polynomial $Q$ has only simple roots. One can thus perform a partial fraction decomposition of $1/Q$, and invert term by term to obtain the claimed result.
\end{proof}

\subsection{Heuristic for D-algebraic cases $\gamma\in\mathbb{N}$ or $\{\gamma_1,\gamma_2\}\subset\mathbb{N}$}
\label{sec:heuri}
In Theorems~\ref{thm:-N} and~\ref{thm:notN} we saw that the D-algebraic cases that remain to be considered fall within the general form
$$\phi_1 = \frac{P}{Q}g,$$ where $P$ and $Q$ are polynomials and the function $g$ is explicitly defined in terms of the functions $g_i$ listed in Lemma ~\ref{lem:wDADF}. We describe a heuristic for inverting the Laplace transform in the special case $Q\equiv 1$, a situation that coincides with the cases where cases $\gamma\in\mathbb{N}$ (in this case $g$ is a function of type $g_1$) or $\{\gamma_1,\gamma_2\}\subset\mathbb{N}$ (in this case $g$ is a function of type $g_2$).
%\textcolor{blue}{a situation that coincides exactly with the D-algebraic cases for which $\gamma>0$. Indeed, Lemma~\ref{lemma:impossible} prohibits $\gamma$ from being positive when $\gamma_1$ and $\gamma_2$ have opposite signs, and also rules out the possibility that both $\gamma_1$ and $\gamma_2$ are negative}. 
Suppose 
$$\phi_1(y)=P(y)(\mathcal{L}f)(y)$$ 
where $P$ is a polynomial and $f$ is a function such that $g=\mathcal{L}f$. Then integrating by parts yields 
$$\phi_1(y)=P(y)g(y)=\int_0^{+\infty}f(v)P(y)e^{yv}\mathrm{d}v=\int_0^{+\infty}f(v)\cdot\boldsymbol{\mathrm{P}}e^{yv}\mathrm{d}v=\int_0^{+\infty}\boldsymbol{\mathrm{P}}^*f(v)\cdot e^{yv}\mathrm{d}v,$$
where $\boldsymbol{\mathrm{P}}$ is the differential operator $P(\mathrm{d}/\mathrm{d}v)$ and $\boldsymbol{\mathrm{P}}^*$ is its adjoint in the Hilbert space $L^2(\mathbb R_+)$, \textit{i.e.}, $\boldsymbol{\mathrm{P}}^*=P(-\mathrm{d}/\mathrm{d}v)$. By injectivity of the Laplace transform operator, the equality $\mathcal{L}\{\nu_1\}=\mathcal{L}\{\boldsymbol{\mathrm{P}}^*f\}$ implies $\nu_1=\boldsymbol{\mathrm{P}}^*f$.\\

 To ensure that all boundary terms arising from integration by parts vanish in the above identity, the function $f$ must satisfy the following two conditions:
 \begin{itemize}
 \item \textit{flatness condition}: $f$ and its derivatives up to order $\mathrm{deg}(P)-1$ must vanish at $v=0$,
 \item \textit{asymptotic control condition}: $f$ and its derivatives must grow at most exponentially at infinity, in the sense that $f^{(j)}(v) e^{yv} \to 0$ as $v\to +\infty$.
\end{itemize}  
Note that it is enough to verify this second condition for real values of $y$ below some threshold $y_0$, since the result then extends by analytic continuation.

As it happens, it is possible to express the function $f$ as a function of type theta. The aim of the next section is to define and study these functions, in particular, check that they satisfy the flatness and asymptotic control conditions. The Section~\ref{subsec:density} will set out and demonstrate the results that give the explicit formula in these cases. 

\subsection{Jacobi Theta functions and Mittag-Leffler expansions}\label{subsec:theta}
Let us introduce the following Jacobi Theta–like functions:
\begin{align}
\theta_{\mathtt{a}}(q) &:= \sum_{n \in \mathbb{Z}} \left(n + \tfrac{\gamma_1}{2}\right) q^{\frac{(2n+\gamma_1)^2 - \mu_1^2}{2}},\label{eq:thetaa} \\
\theta_{\mathtt{b}}(q) &:= \sum_{n \in \mathbb{Z}} (-1)^n n^2q^{\frac{n^2 - \mu_1^2}{2}}.\label{eq:thetab}
\end{align}
In Theorems~\ref{thm:-N} and~\ref{thm:notN}, we provided explicit formulas for the Laplace transform $\phi_1$, involving the functions analyzed in Lemma~\ref{lem:wDADF}. In what follows, we establish connections between these functions and the theta functions introduced above.

\begin{remark}[Link with classical Theta functions]
The function $\theta_{\mathtt{b}}$ is closely related to the classical Jacobi Theta function
$$\theta_4(q) := \sum_{n \in \mathbb{Z}} (-1)^n q^{n^2},$$
via the identity
$$\theta_{\mathtt{b}}(q) = q^{\frac{1-\mu_1^2}{2}}\theta_4'(\sqrt{q}).$$
Jacobi Theta functions already exhibit profound connections with probability theory; see, for example, \cite{Salminen_Vignat_2024}.
\end{remark}

\begin{proposition}[Theta-type representation]\label{prop:jacobi} The functions $v\mapsto \theta_\star(e^{-v})$ have the following series representation 
\begin{align}
\theta_{\mathtt{a}}(e^{-v})&=\left(\frac{\pi}{2v}\right)^{3/2}\sum_{n\in\mathbb Z}n \sin(\pi\gamma_1 n)\exp\left(-\frac{\pi^2n^2}{2v}+\frac{\mu_1^2 v}{2}\right),\\
\theta_{\mathtt{b}}(e^{-v})&=\left(\frac{2\pi}{v^3}\right)^{1/2}\sum_{n\in\mathbb Z}\left(1-\frac{\pi^2(2n-1)^2}{v}\right)\exp\left(\frac{\mu_1^2v}{2}-\frac{\pi^2(2n-1)^2}{2v}\right).
\end{align}
\end{proposition}
\begin{proof}
Consider the function $v\mapsto\theta_{\mathtt{a}}(e^{-v})$. For a fixed value of $v>0$, one can write
$$\theta_{\mathtt{a}}(e^{-v})=\sum_{n\in\mathbb Z}h_{\mathtt{a}}(n),\text{ where }h_{\mathtt{a}}(t):=\left(t+\frac{\gamma_1}{2}\right)\exp\left(-v\frac{(2t+\gamma_1)^2-\mu_1^2}{2}\right).$$
For all $v>0$, the function $h_{\mathtt{a}}$ is rapidly decreasing (\textit{i.e.}, belongs to the Schwartz space $\mathcal{S}(\mathbb R)$) because the factor $\exp(-2 v t^2 - 2 v \gamma_1 t)$ dominates any polynomial growth of $t$. Hence, one can apply the Poisson summation formula to express this sum in terms of its Fourier transform $\widehat{h}_{\mathtt{a}}$:
$$\sum_{n\in\mathbb Z}h_{\mathtt{a}}(n)=\sum_{n\in\mathbb Z}\widehat{h}_{\mathtt{a}}(n),\text{ where }\widehat{h}_{\mathtt{a}}(\omega):=\int_{-\infty}^{+\infty}h_{\mathtt{a}}(t)e^{-2i\pi\omega t}\mathrm{d}t.$$
Let us compute explicitly the Fourier transform $\widehat{h}_{\mathtt{a}}$. For all $\omega\in\mathbb R$,
\begin{align*}
\widehat{h}_{\mathtt{a}}(\omega)&=e^{i\pi \omega \gamma_1}\int_{-\infty}^{+\infty} \left(t+\frac{\gamma_1}{2}\right)\exp\left(-2i\pi \omega\left(t+\frac{\gamma_1}{2}\right)-v\frac{(2t+\gamma_1)^2-\mu_1^2}{2}\right)\mathrm{d}t\\
&=e^{i\pi \omega\gamma_1}\left(\frac{-1}{2i\pi }\right)\frac{\mathrm{d}}{\mathrm{d}\omega}\int_{-\infty}^{+\infty} \exp\left(-2i\pi \omega\left(t+\frac{\gamma_1}{2}\right)-v\frac{(2t+\gamma_1)^2-\mu_1^2}{2}\right)\mathrm{d}t,
\end{align*}
where the second equality follows from the dominated convergence theorem, which always holds for functions in the Schwartz space $\mathcal{S}(\mathbb{R})$. We recognize the well-known Gaussian integral
$$\int_{-\infty}^{+\infty} \exp\left(-2i\pi \omega\left(t+\frac{\gamma_1}{2}\right)-v\frac{(2t+\gamma_1)^2-\mu_1^2}{2}\right)\mathrm{d}t=\sqrt{\frac{\pi}{2v}}\exp\left(-\frac{\pi^2\omega^2}{2v}+\frac{\mu_1^2 v}{2}\right),$$
and we deduce that
$$\widehat{h}_{\mathtt{a}}(\omega)=\frac{-e^{i\pi\omega\gamma_1}}{2i\pi}\frac{\mathrm{d}}{\mathrm{d}\omega}\sqrt{\frac{\pi}{2v}}\exp\left(-\frac{\pi^2\omega^2}{2v}+\frac{\mu_1^2 v}{2}\right)=-i\omega\left(\frac{\pi}{2v}\right)^{3/2}e^{i\pi\omega\gamma_1}\exp\left(-\frac{\pi^2 \omega^2}{2v}+\frac{\mu_1^2v}{2}\right).$$
Finally,
$$\theta_{\mathtt{a}}(e^{-v})=\left(\frac{\pi}{2v}\right)^{3/2}\sum_{n\in\mathbb Z}n\Big[\sin(\pi n\gamma_1)-i\cos(\pi n\gamma_1)\Big]\exp\left(-\frac{\pi^2 n^2}{2v}+\frac{\mu_1^2v}{2}\right).$$
One obtains the claimed result by noticing that the imaginary parts of the terms cancel pairwise for $n\ne 0$, and that the zeroth term vanishes, leaving only the real part of the sum.\\
Concerning $\theta_{\mathtt{b}}$, the same method shows that
$$\theta_{\mathtt{b}}(e^{-v})=\sum_{n\in\mathbb Z}h_{\mathtt{b}}(n)=\sum_{n\in\mathbb Z}\widehat{h}_{\mathtt{b}}(n),\text{ where }h_{\mathtt{b}}(t):=t^2\exp\left(i\pi t-v\frac{t^2-\mu_1^2}{2}\right)$$
and
\begin{align*}
\widehat{h}_{\mathtt{b}}(\omega):=\int_{-\infty}^{+\infty}h_{\mathtt{b}}(t)e^{-2i\pi \omega t}\mathrm{d}t &=\frac{-1}{4\pi^2}\frac{\mathrm{d}^2}{\mathrm{d}\omega^2}\int_{-\infty}^{+\infty}\exp\left(i\pi t-v\frac{t^2-\mu_1^2}{2}-2i\pi\omega t\right)\mathrm{d}t\\
&=\frac{-1}{4\pi^2}\frac{\mathrm{d}^2}{\mathrm{d}\omega^2}\left[\sqrt{\frac{2\pi}{v}}\exp\left(\frac{\mu_1^2v}{2}-\frac{\pi^2(1-2\omega)^2}{2v}\right)\right]\\
&=\left(\frac{2\pi}{v^3}\right)^{1/2}\left(1-\frac{\pi^2(1-2\omega)^2}{v}\right)\exp\left(\frac{\mu_1^2 v}{2}-\frac{\pi^2(1-2\omega)^2}{2v}\right),
\end{align*}
which establishes the claimed result.
\end{proof}
\begin{corollary}[Flatness]\label{cor:flat} For $\star\in\{\mathtt{a},\mathtt{b}\}$, The function $v\mapsto \theta_\star(e^{-v})$ is flat at $0$, that is 
$$\lim_{v\downarrow 0}\frac{\mathrm{d}^n}{\mathrm{d}v^n}\theta_\star(e^{-v})=0,\text{ for all }k\in\mathbb N_0$$
\end{corollary}
\begin{proof} For $\star=\mathtt{a}$ we have, according to Proposition~\ref{prop:jacobi},
$$\theta_{\mathtt{a}}(e^{-v})=\left(\frac{\pi}{2v}\right)^{3/2}\sum_{n\in\mathbb Z}n \sin(\pi\gamma_1 n)\exp\left(-\frac{\pi^2n^2}{2v}+\frac{\mu_1^2 v}{2}\right)$$
For each nonzero $n$, the term $n \sin(\pi \gamma_1 n) \exp(-\pi^2 n^2/(2v))$ vanishes faster than any power of $v$ as $v \downarrow 0$. Differentiating term by term preserves this property, so all derivatives also tend to zero in the limit, and $\theta_{\mathtt{a}}(e^{-v})$ is flat at $v=0$. The same argument applies to $\theta_{\mathtt{b}}(e^{-v})$.
\end{proof}

\begin{lemma}[Mittag-Leffler expansions]\label{lemma:mittag} The following series expansions hold
\begin{align*}
    \frac{\sin(a)}{\cos(\sqrt{z})-\cos(a)}&=-2\sum_{n\in\mathbb Z}\frac{a+2\pi n}{z-(a+2\pi n)^2},\quad\quad&\frac{1}{\sqrt{z}\sin(\sqrt{z})}=\sum_{n\in\mathbb Z}\frac{(-1)^n}{z-(n\pi)^2}.
\end{align*}
\end{lemma}
\begin{proof}
These formulas are direct applications of the Mittag-Leffler expansion, which describes how to reconstruct certain meromorphic functions from its poles and residues. See~\cite[section 2.1]{ahlfors} for more details.
\end{proof}
\begin{remark}[Convergence of the Mittag-Leffler expansions] The first sum must be understood in the sense of \textit{principal value} for bi-infinite series, \textit{i.e.},
$$\frac{\sin(a)}{\cos(\sqrt{z})-\cos(a)}=-2\lim_{N\to +\infty}\sum_{n=-N}^N\frac{a+2\pi n}{z-(a+2\pi n)^2}.$$
\end{remark}
\begin{proposition}[Laplace transforms of Theta functions]\label{prop:TLtheta}
The functions $v \mapsto \theta_\star(e^{-v})$ admit Laplace transforms given by
\begin{align}
\mathcal{L}\left\{\theta_{\mathtt{a}}(e^{-v})\right\}(y) &= \frac{\pi}{2}\cdot \frac{\sin(\pi \gamma_1)}{\cos\left(\pi\sqrt{2y+\mu_1^2}\right)-\cos(\pi\gamma_1)},\\
\mathcal{L}\left\{\theta_{\mathtt{b}}(e^{-v})\right\}(y) &= \frac{-2\pi\sqrt{2y+\mu_1^2}}{\sin\left(\pi\sqrt{2y+\mu_1^2}\right)}.
\end{align}
\end{proposition}
\begin{proof}
This statement is a simple consequence of the Mittag-Leffler expansions given in Lemma~\ref{lemma:mittag}, and of the fact that, for $\mathfrak{Re}(y)<0$,
$$\mathcal{L}\left\{\exp(av)\right\}(y):=\int_0^{+\infty}\exp(av+yv)\mathrm{d}v=-\frac{1}{y+a}.$$
For $\star=\mathtt{a}$,
\begin{align*}
\mathcal{L}\left\{\theta_{\mathtt{a}}(e^{-v})\right\}(y)&=\int_0^{+\infty}\sum_{n\in\mathbb Z}\left[\left(n+\frac{\gamma_1}{2}\right)\exp\left(-v\frac{(2n+\gamma_1)^2-\mu_1^2}{2}\right)\right]e^{yv}\mathrm{d}v\\
&=\sum_{n\in\mathbb Z}\left(n+\frac{\gamma_1}{2}\right)\int_0^{+\infty}\exp\left(-\frac{(2n+\gamma_1)^2-\mu_1^2}{2}v+yv\right)\mathrm{d}v\\
&=-\sum_{n\in\mathbb Z}\left(n+\frac{\gamma_1}{2}\right)\left(y-\displaystyle\frac{(2n+\gamma_1)^2-\mu_1^2}{2}\right)^{-1}\\
&=\frac{\pi}{2}\cdot -2\sum_{n\in\mathbb Z}\frac{\pi\gamma_1+2\pi n}{\pi^2(2y+\mu_1^2)-(\pi\gamma_1+2\pi n)^2}\\
&=\frac{\pi}{2}\cdot\frac{\sin(\pi\gamma_1)}{\cos\left(\pi\sqrt{2y+\mu_1^2}\right)-\cos(\pi\gamma_1)}.
\end{align*}
The crucial step is the exchange between the sum and the integral, which follows from a standard dominated convergence argument. For the Laplace transform of $\theta_{\mathtt{b}}(e^{-v})$, note that
$$\theta_{\mathtt{b}}(e^{-v})=-\left(2\frac{\partial}{\partial v}-\mu_1^2\right)\left[\sum_{n\in\mathbb Z}(-1)^n \exp\left(-v\frac{n^2-\mu_1^2}{2}\right)\right].$$
The function to which the differential operator $2\partial/\partial v -\mu_1^2$ is applied in the above expression is flat at $0$. Indeed, by applying the Poisson summation formula (see the proof of Proposition~\ref{prop:jacobi}), one obtains
$$\sum_{n\in\mathbb Z}(-1)^n\exp\left(-v\frac{n^2-\mu_1^2}{2}\right)=\sqrt{\frac{2\pi}{v}}\sum_{n\in\mathbb Z}\exp\left(\frac{\mu_1^2v}{2}-\frac{\pi^2(2n-1)^2}{2v}\right),$$
which clearly satisfies the flatness condition. We can therefore integrate by parts:
\begin{align*}
\mathcal{L}\{\theta_{\mathtt{b}}(e^{-v})\}(y)&=(2y+\mu_1^2)\mathcal{L}\left\{\sum_{n\in\mathbb Z}(-1)^n\exp\left(-v\frac{n^2-\mu_1^2}{2}\right)\right\}(y)\\
&=-(2y+\mu_1^2)\sum_{n\in\mathbb Z}(-1)^n \left(y+\frac{\mu_1^2-n^2}{2}\right)^{-1}\\
&=-2\sum_{n\in\mathbb Z}\frac{\pi^2(2y+\mu_1^2)}{\pi^2(2y+\mu_1^2)-(n\pi)^2}=\frac{-2\pi\sqrt{2y+\mu_1^2}}{\sin\left(\pi\sqrt{2y+\mu_1^2}\right)}.
\end{align*}
The last step above is a straightforward application of the Mittag–Leffler expansion given in Lemma~\ref{lemma:mittag}.
\end{proof}

\begin{corollary}[Growth condition]\label{cor:growth} For $\star\in\{\mathtt{a},\mathtt{b}\}$, there exists $y_\star<0$ such that, for all $y<y_\star$,
$$\lim_{v\to+\infty}e^{yv}\frac{\mathrm{d}^j}{\mathrm{d}v^j}\theta_\star(e^{-v})=0,\text{ for all }j\in\mathbb N_0.$$
\end{corollary}

\begin{proof}
Using the expression given Proposition~\ref{prop:TLtheta}, one can determine that the poles of $\mathcal{L}\{\theta_{\mathtt{a}}(e^{-v})\}$ are precisely at $\y(s_1+n)$ for $n\in\mathbb{Z}$. These poles are all real and satisfy $\y(s_1 + n) \geqslant \y(s_+) = -\mu_1^2/2$. Let us denote  
$$y_{\mathtt{a}} := \min_{n \in \mathbb{Z}} \y(s_1 + n),$$ 
which corresponds to the abscissa of convergence of the Laplace transform. According to the Hardy-Littlewood Tauberian theorem (see~\cite{Feller_1971}), we then have  
$$\theta_{\mathtt{a}}(e^{-v}) = O\left(e^{-y_{\mathtt{a}} v}\right), \quad \text{as } v \to +\infty,$$
which proves the claim for $j = 0$. The same reasoning applies to the $j$-th derivative, since all the poles of its Laplace transform are likewise greater than $-\mu_1^2/2$. Similarly, the poles of $\mathcal{L}\{\theta_{\mathtt{b}}(e^{-v})\}$ are located at $(n^2-\mu_1^2)/2$ for $ n\in\mathbb{Z}$. They all are greater than $y_{\mathtt{b}} := \y(s_+)$, and the remaining argument proceeds exactly as in the case $\star = \mathtt{a}$.
\end{proof}

\subsection{Density of the lateral measure when $\gamma\in\mathbb{N}$ or $\{\gamma_1,\gamma_2\}\subset\mathbb{N}$}\label{subsec:density}
As mentioned in Section~\ref{sec:heuri}, we can invert the Laplace transform $\phi_1$ to recover the lateral measure $\nu_1$ when $\phi_1$ is D-algebraic and $\gamma>0$ (that is, $\gamma\in\mathbb N$ or ${\gamma_1,\gamma_2}\subset\mathbb N$). See Remark~\ref{rem:limitation} for a discussion of the limitations in the other D-algebraic cases.

\begin{theorem}[Density, $\gamma\in\mathbb N$]\label{thm:densitygamma} If $\gamma\in \mathbb N$ then the density $\nu_1$ of the lateral measure $\boldsymbol{\nu}_1$ is given by
$$\nu_1(v)\propto\sum_{n\in\mathbb Z}\left(n+\frac{\gamma_1}{2}\right)P\left(\frac{(2n+\gamma_1)^2-\mu_1^2}{2}\right)\exp\left(-v\frac{(2n+\gamma_1)^2-\mu_1^2}{2}\right),$$
where $P(y):=\prod_{k=1}^{\gamma}\Big(y-\y(s_1+k)\Big)$. Equivalently
\begin{equation}\label{eq:opetheta1}
\nu_1(v)\propto \boldsymbol{\mathrm{P}}^*\theta_{\mathtt{a}}(e^{-v})\text{ where }\boldsymbol{\mathrm{P}}^*:=P\left(-\frac{\partial}{\partial v}\right)=\prod_{k=1}^{\gamma}\left(-\frac{\partial}{\partial v}-\y(s_1+k)\right).
\end{equation}
\end{theorem}
\begin{proof}By Theorem~\ref{thm:-N},
$$\mathcal{L}\{\nu_1\}(y)=\phi_1(y)\propto \frac{P(y)}{\cos\left(\pi\sqrt{2y+\mu_1^2}\right)-\cos(\pi\gamma_1)}.$$
We recognize in this expression a factor corresponding to a Mittag-Leffler expansion, as described in Lemma~\ref{lemma:mittag},  which was shown in Proposition~\ref{prop:TLtheta} to be (up to a multiplicative constant) the Laplace transform of $\theta_{\mathtt{a}}(e^{-v})$:  
$$\phi_1(y)\propto P(y)\mathcal{L}\left\{\theta_{\mathtt{a}}(e^{-v})\right\}(y).$$
We know from Corollaries~\ref{cor:flat} and~\ref{cor:growth} that $v\mapsto\theta_{\mathtt{a}}(e^{-v})$ satisfies both the \textit{flatness} and the \textit{asymptotic control conditions}. We can therefore apply the heuristic from Section~\ref{subsec:fractionalDO} yields
\begin{align*}
\nu_1(v)&\propto \boldsymbol{\mathrm{P}}^*\theta_{\mathtt{a}}(e^{-v})= \left[\prod_{k=1}^{\gamma}\left(\frac{\partial}{\partial v}+\y(s_1+k)\right)\right]\left(\sum_{n\in\mathbb Z}\left[\left(n+\frac{\gamma_1}{2}\right)\exp\left(-v\frac{(2n+\gamma_1)^2-\mu_1^2}{2}\right)\right]\right)\\
&=\sum_{n\in\mathbb N}\left(n+\frac{\gamma_1}{2}\right)\prod_{k=1}^{\gamma}\left(\frac{\partial}{\partial v}+\y(s_1+k)\right)\exp\left(-v\frac{(2n+\gamma_1)^2-\mu_1^2}{2}\right)\\
&=\sum_{n\in\mathbb N}\left(n+\frac{\gamma_1}{2}\right)\prod_{k=1}^{\gamma}\left[\left(\frac{\mu_1^2-(2n+\gamma_1)^2}{2}+\y(s_1+k)\right)\right]\exp\left(-v\frac{(2n+\gamma_1)^2-\mu_1^2}{2}\right).
\end{align*}
\end{proof}

Figure~\ref{fig:plots} shows plots of $\nu_1$ corresponding to several parameter choices.

\begin{figure}
  \centering
  \includegraphics[width=\textwidth]{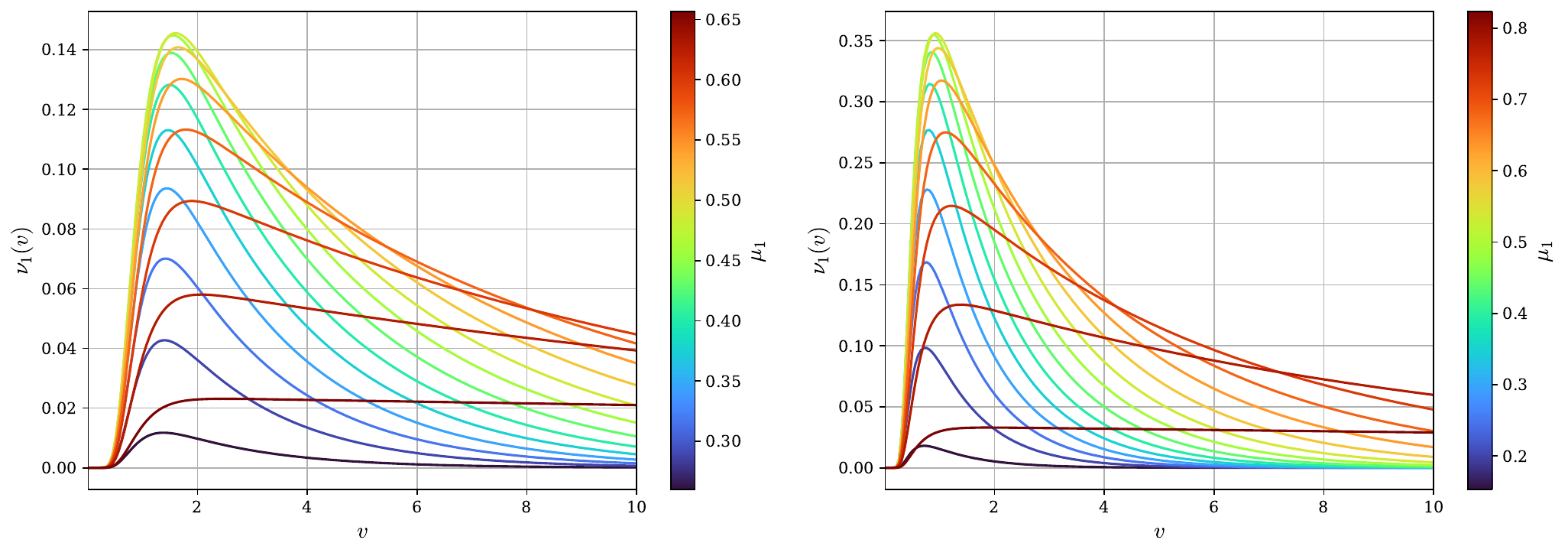} 
  \caption{Graphs of $\nu_1$ for $\gamma=1$ (left) and $\gamma=4$ (right). In both cases, the values of $r_1$ and $r_2$ are fixed — $(r_1,r_2)=(\frac12,\frac13)$ on the left and $(r_1,r_2)=(\frac15,\frac16)$ on the right — while $\mu_1=1-\mu_2$ varies between $r_2/(1+r_2)$ and $1/(1+r_1)$, the bounds imposed by the recurrence condition~\eqref{eq:H2}.}
  \label{fig:plots}
\end{figure}

\begin{theorem}[Density, $\{\gamma_1,\gamma_2\}\subset \mathbb N$]\label{thm:densitygamma12} If $\{\gamma_1,\gamma_2\}\subset \mathbb N$ and $\gamma\notin\mathbb Z$, then the density $\nu_1$ of the lateral measure $\boldsymbol{\nu}_1$ is given by
$$\nu_1(v)\propto \sum_{n\in\mathbb Z}(-1)^nn^2P\left(\frac{n^2-\mu_1^2}{2}\right)\exp\left(-v\frac{n^2-\mu_1^2}{2}\right),$$
where 
$$P(y):=\prod_{k=1}^{\lfloor\frac{\gamma_1-1}{2}\rfloor}\Big(y-\y(s_1+k)\Big)\prod_{k=0}^{\lfloor\frac{\gamma_2}{2}-1\rfloor}\Big(y-\y(s_2-k)\Big).$$ 
Equivalently, $\nu_1(v)\propto \boldsymbol{\mathrm{P}}^*\theta_{\mathtt{b}}(e^{-v})$, with 
$$\boldsymbol{\mathrm{P}}^*:=P\left(-\frac{\partial}{\partial v}\right)=\left[\prod_{k=1}^{\lfloor\frac{\gamma_1-1}{2}\rfloor}\left(-\frac{\partial}{\partial v}-\y(s_1+k)\right)\right]\left[\prod_{k=0}^{\lfloor\frac{\gamma_2}{2}-1\rfloor}\left(-\frac{\partial}{\partial v}-\y(s_2-k)\right)\right]$$
\end{theorem}

\begin{proof}
From Theorem~\ref{thm:notN}, when $\{\gamma_1,\gamma_2\}\subset \mathbb N$,
$$\phi_1(y)\propto \frac{P(y)\sqrt{2y+\mu_1^2}}{\sin\left(\pi\sqrt{2y+\mu_1^2}\right)}\propto P(y)\mathcal{L}\left\{\theta_{\mathtt{b}}(e^{-v})\right\}(y),$$
where $P$ is given in Proposition~\ref{prop:decoupling}. The function $v\mapsto\theta_{\mathtt{b}}(e^{-v})$ satisfies both the \textit{flatness} and the \textit{asymptotic control conditions} by Corollaries~\ref{cor:flat} and~\ref{cor:growth}, respectively. We can therefore apply the heuristic given in Section~\ref{subsec:fractionalDO}, and use the linearity of $\boldsymbol{\mathrm{P}}^*$ to obtain the desired result.
\end{proof}

\begin{remark}[On the remaining cases]\label{rem:limitation} This method to invert the Laplace transform $\phi_1$ encounters an obstruction when $\gamma_1$ and $\gamma_2$ are integers with opposite signs. For instance, if $\gamma_1\in -2\mathbb N+1$ and $\gamma_2\in 2\mathbb N-1$, then Theorem~\ref{thm:notN} ensures that $Q(y)\phi_1(y)\propto (2y+\mu_1^2)P(y)\widetilde g_3(y)$ where 
$$\widetilde g_3(y):=\frac{1}{\sqrt{2y+\mu_1^2}\tan\left(\frac{\pi}{2}\sqrt{2y+\mu_1^2}\right)}\propto \mathcal{L}\left\{\theta_{\mathtt{c}}(e^{-v})\right\}(y),\quad \theta_{\mathtt{c}}(q):=\sum_{n\in\mathbb Z}q^{\frac{(2n)^2-\mu_1^2}{2}}= \theta_3(q^2)q^{-\mu_2/2},$$
Similarly, if $\gamma_1\in-2\mathbb N$ and $\gamma_2\in 2\mathbb N$, then $Q(y)\phi_1(y)\propto P(y)g_4(y)$ where
$$g_4(y):=\frac{\tan\left(\frac{\pi}{2}\sqrt{2y+\mu_1^2}\right)}{\sqrt{2y+\mu_1^2}}\propto\mathcal{L}\left\{\theta_{\mathtt{d}}(e^{-v})\right\}(y),\quad \theta_{\mathtt{d}}(q):=\sum_{n\in\mathbb Z}q^{\frac{(2n+1)^2-\mu_1^2}{2}} = \theta_2(q^2)q^{-\mu_2/2}.$$
In the previous formulas, $\theta_2$ and $\theta_3$ are the classical Jacobi theta functions. Although $v\mapsto \theta_{\mathtt{c}}(e^{-v})$ and $v\mapsto \theta_{\mathtt{d}}(e^{-v})$ satisfy the \textit{asymptotic control condition}, they do not fulfil the \textit{flatness condition}. More precisely, for $\star\in\{\mathtt{c},\mathtt{d}\}$, it follows from classical results on Jacobi theta functions that
$$\theta_{\star}(e^{-v})\sim \sqrt{\frac{\pi}{2v}},\text{ when }v\downarrow 0.$$
In particular, the integration by parts formula cannot be applied directly. It would be interesting to investigate whether this method can be adapted to the case where $\gamma_1$ and $\gamma_2$ are integers with opposite signs in future work.
\end{remark}

\appendix
\section{Special cases}\label{sec:r1r2}
To avoid an excessive number of cases, we handle the ones where $r_1=-1$ or $r_2=-1$ in this appendix, separately. Note that, according to~\eqref{eq:H1}, $r_1$ and $r_2$ cannot be equal to $-1$ simultaneously.
\begin{remark}[Interpretations] We offer two complementary interpretations of these limit cases: a probabilistic one regarding the Brownian particles model itself, and a geometric one concerning the vanishing sets of the polynomial coefficients $K$, $k_1$ and $k_2$, appearing in the functional equation~\eqref{eq:FE}.
\begin{enumerate}
\item (\textit{probabilistic interpretation}) Using the notations from the introduction, $r_1=-1$ implies that $q_1^-=0$ and $q_2^+=1$. This corresponds to the case where the leading particle $R_1$ of the ordered process $R$ behaves as a free, \textit{i.e.} unreflected, Brownian motion, and the medium particle $R_2$ reflects on it, taking all the local time for itself.
\item (\textit{geometric interpretation}) When $r_1=-1$ (\textit{resp.} $r_2=-1$), the vanishing set of $k_1$ (\textit{resp.} $k_2$) coincides with the symmetry axis of the parabola $K=0$, see Figure~\ref{fig:overdegenerate}.

\begin{figure}
\centering
\includegraphics[width=15cm]{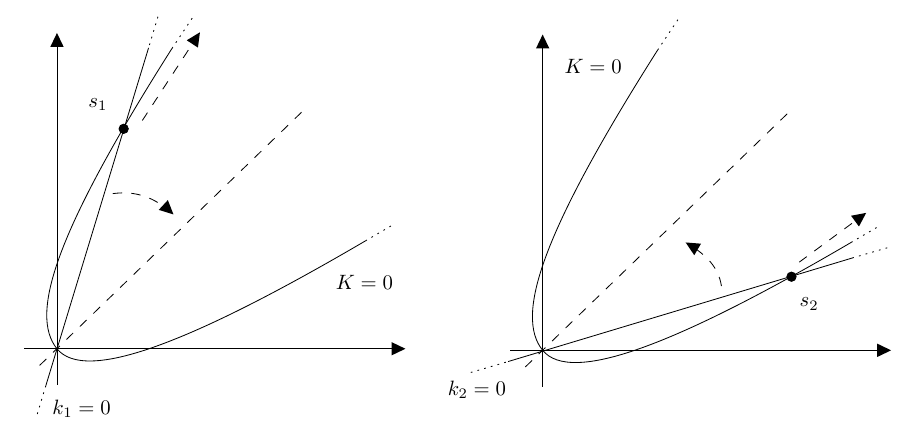}
\caption{The axis of symmetry of the parabola $K=0$ is the limit position of $k_1=0$ as $r_1\to -1$ (on the left) and $k_2=0$ as $r_2\to -1$ (on the right).}
\label{fig:overdegenerate}
\end{figure}
\end{enumerate}
\end{remark}
\subsection{First case: $r_1=-1$}
Suppose $r_1=-1$. Setting $(\x,\y)$ as the uniformization given in Proposition~\ref{prop:unif}, on can easily compute
$$k_1(s)=k_1(\x(s),\y(s))=2s,$$
and $k_2(s)=k_2(\x(s),\y(s))=2(1+r_2)s(s-s_2)$ where $s_2$ is alredy defined in~\eqref{eq:k2}.  Morally, the second root of $k_1(s)$ is sent to infinity. As a direct consequence, neither $s_1$, nor $\gamma$, nor $\gamma_1$ can be defined in this context {(in Equation~\eqref{eq:k2}, the denominator of $s_1$ becomes $0$ when $r_1=-1$)}. We can nonetheless define $\gamma_2$ as in~\eqref{eq:gammai}. The following lemma describes the possible values for $\gamma_2$ under the existence and recurrence condition.
\begin{lemma}[Positivity of $\gamma_2$]\label{lemma:gamma2pos} Suppose $r_1=-1$. Then $\gamma_2>0$.
\end{lemma}
\begin{proof} When $r_1=-1$,~\eqref{eq:H1} is equivalent to $1+r_2>0$. Using the expression of $\gamma_2$ given in~\eqref{eq:gamma}, $\gamma_2$ and $1+\mu_1-r_2\mu_2$ share the same sign, which is positive by~\eqref{eq:H2}.
\end{proof}
Recall that $\zeta (s)=-s+2s_-$. From the above expression we obtain  
$$\frac{k_2 (s)}{k_1 (s)}=(1+r_2)(s-s_2).$$ 
Then, following the steps of Section~\ref{subsec:prolongement}, we can show an alternative version for Theorem~\ref{thm:prolong}:
\begin{theorem}[Meromorphic continuation and difference equation, $r_1=-1$]\label{thm:prolongementr1} The function $\varphi_1$ admits a meromorphic continuation to $\mathbb C$, pole-free on $\overline{\Delta}$ and nonzero on $\overline{\Delta}\cap\mathbb R$, along with the functional equation~\eqref{eq:eqfunc} and the invariance properties~\eqref{eq:inv1} and~\eqref{eq:inv2}, through the formula 
\begin{equation}\label{eq:diffeqr1}
\varphi_1(s+1)=G(s)\varphi_1(s)
\end{equation}
where 
\begin{equation}\label{eq:defGr1}
G(s):=\frac{k_1(s)k_2(\zeta s)}{k_2(s)k_1(\zeta s)}=-\frac{s-(2s_--s_2)}{s-s_2}=-\frac{s-(s_2-\gamma_2)}{s-s_2}.
\end{equation}
\end{theorem}
\begin{remark} The notation $G$ is consistent with the general difference equation~\eqref{eq:Gdef}, 
since as $r_1 \to -1$ we have $k_1(s)/k_1(\zeta s)\to -1$.
\end{remark}
\begin{proof}
We refer to Section~\ref{subsec:prolongement} for step-by-step proofs of the meromorphic continuation to $\mathbb C$. To study the poles and zeros of $\varphi_1$ (\textit{resp.} $\varphi_2$) in the domain $\Delta$, recall that it is already pole-free on $\Delta_{\mathrm{y}}$ (\textit{resp.} $\Delta_{\mathrm{x}}$) and nonzero on $\Delta_{\mathrm{y}}\cap \mathbb R$ (\textit{resp.} $\Delta_{\mathrm{x}}\cap\mathbb R$). For all $s\in \Delta_{\mathrm{y}}$,
$$\varphi_1(s)=-(1+r_2)(s-s_2)\varphi_2(s).$$
Thus, $\varphi_1$ has no poles in $\Delta_{\mathrm{y}}$, and $\varphi_1(s)=0$ if and only if $s=s_2\in \Delta_{\mathrm{x}}$, which is impossible by Lemma~\ref{lemma:s1s2}.
\end{proof}
The function $-G$ admits the following decoupling function (see Definition~\ref{def:decoupling}):
\begin{equation}\label{eq:defDdagger}
D_\dagger(s):=\frac{\Gamma(s+s_2+\mu_2)}{\Gamma(s-s_2)}.
\end{equation}
The next proposition identifies when $D_\dagger$ is rational, gives explicit formulas in these cases, and shows that no rational decoupling exists otherwise.
\begin{proposition}[Rational decoupling]\label{prop:rationalr1} If $r_1=-1$ then the following three properties are equivalent:
\begin{enumerate}
\item The function $-G$ admits a rational decoupling,
\item The decoupling function $D_\dagger$ is rational,
\item $\gamma_2\in\mathbb N$.
\end{enumerate}
Moreover, if $\gamma_2\in\mathbb N$ then
\begin{equation}\label{eq:rationalr1}
D_\dagger(s)=2^{-\gamma_2/2}P(\y(s))\Big[\sqrt{2}(s-s_+)\Big]^\varepsilon,
\end{equation}
with
\begin{equation*}
P(y):=\prod_{k=0}^{\left\lfloor\frac{\gamma_2}{2}-1\right\rfloor}\Big(y-\y(s_2-k)\Big)
\end{equation*}
and $\varepsilon=1$ if $\gamma_2$ is odd and $\varepsilon=0$ otherwise.
\end{proposition}
\begin{proof}
The implication 2.$\implies$1. is trivial.
 The techniques to prove that $\gamma_2\in\mathbb N$ implies that $D_\dagger\in \mathbb C(s)$ are very similar to the proof of Proposition \ref{prop:decoupling} but in a simpler setting so we will only sketch the proof. Assume that $\gamma_2\in\mathbb N$. Using the relation~\eqref{eq:Gammarec} satisfied by the gamma function, one obtains
$$D_\dagger(s)=\frac{\Gamma(s-s_2+\gamma_2)}{\Gamma(s-s_2)}=\prod_{k=0}^{\gamma_2 -1}\Big(s-(s_2-k)\Big).$$
With $s-s_{+}=s-(s_2 -\frac{\gamma_2 -1}{2})$, we may reorder the term to obtain
$$\prod_{k=0}^{\gamma_2 -1}\Big(s-(s_2-k)\Big)\propto \frac{1}{(s-s_{+})^{\varepsilon}}\prod_{k=0}^{\left\lfloor\frac{\gamma_2}{2}-1\right\rfloor}\Big(\y(s)-\y(s_2-k)\Big)$$
where $\varepsilon=1$ if $\gamma_2$ is odd and $\varepsilon=0$ otherwise. The remaining proportionality constants arise from the fact that the leading coefficient of $\y$ is $2$. The last implication 1.$\implies$3. is a direct consequence of Put-Singer rational decoupling criterion (Lemma~\ref{lem1}).
\end{proof}
\begin{remark}[Alternative decoupling function] Alternatively, we could consider the following decoupling function for $G$ itself:
$$\widetilde{D}_\dagger(s):=\sin\left(\pi s\right)D_\dagger(s).$$
This decoupling function is never rational.
\end{remark}
\begin{lemma}[On the function $\varphi_1/D_\dagger$]\label{lemma:phi1Dr1} Suppose $r_1=-1$.
\begin{enumerate}
\item $(\varphi_1/D_\dagger)^2$ is a Type I invariant. Moreover, if $\gamma_2\in\mathbb N$, then $(\varphi_1/D_\dagger)^2$ is a Type II invariant.
\item For $s$ staying in $\mathfrak{B}_1$, $\varphi_1/D_\dagger(s)\to 0$ as $|s|\to +\infty$.
\item If $\gamma_2\notin \mathbb N$, then $\varphi_1/D_\dagger$ admits exactly one (simple) pole in $\mathfrak{B}_1\cup \{\mu_1\}$, at $s_2-p$ for some non-negative integer $p$. 
\end{enumerate}
\end{lemma}
\begin{proof}
\begin{enumerate}
\item The function $D_\dagger$ is a decoupling function for $-G$; therefore, $D_\dagger^2$ is a decoupling function for $G^2$. Squaring both sides of the difference equation~\eqref{eq:diffeqr1} and replacing $G(s)$ by $D_\dagger(s+1)/D_\dagger(s)$ shows that $(\varphi_1/D_\dagger)^2$ is $1$-periodic (\textit{i.e.}, a Type~I invariant). Moreover, if $\gamma_2\in\mathbb{N}$, then $D_\dagger^2(s)$ depends only on $\y(s)$, so $(\varphi_1/D_\dagger)^2$ is invariant under both $\eta$ and $\eta\circ\zeta$, and is therefore a Type~II invariant.
\item Using the asymptotic behavior given by~\eqref{eq:ratiogamma}, we obtain $D_\dagger(s)\sim s^{\gamma_2}$. 
Lemma~\ref{lemma:gamma2pos} shows that $\gamma_2>0$, hence $D_\dagger(s)\to+\infty$ as $|s|\to\infty$. 
Moreover, by a standard result on the Laplace transform, $\varphi_1(s)\to 0$ as $s\to\infty$ in the strip $\mathfrak{B}_1$ (see Lemma~\ref{lemma:stripe}). 
Combining these two limits establishes the desired result.
\item We already know, by Theorem~\ref{thm:prolongementr1}, that $\varphi_1$ has no pole in $\overline{\mathfrak{B}}_1\subset \overline{\Delta}$. Hence the poles of $\varphi_1/D_\dagger$ in the strip are exactly those of $1/D_\dagger$. One can easily check (mimicking the proof of Lemma~\ref{lemma:polephi1D}) that these poles are all simple and of the form $s_2-p$ with $p\in\mathbb{N}_0$. Moreover, they cannot be canceled by a zero in the ratio of Gamma functions, since this would imply $\gamma_2\in\mathbb{N}$. Finally, since $\gamma_2>0$ implies $s_2>0$, exactly one of these poles lies in $\mathfrak{B}_1\cup\{\mu_1\}$.
\end{enumerate}
\end{proof}
\begin{theorem}[Laplace transform, $r_1=-1$ and $\gamma_2\in \mathbb N$]\label{thm:DAr1} Suppose $r_1=-1$ and $\gamma_2\in\mathbb N$. Then there exists a polynomial $P\in\mathbb R[X]$ such that the Laplace transform $\phi_1$ satisfies 
$$\phi_1(y)\propto \frac{P(y)\sqrt{2y+\mu_1^2}}{\sin\left(\frac{\pi}{2}\sqrt{2y+\mu_1^2}\right)}$$
if $\gamma_2$ is odd, and 
$$\phi_1(y)\propto \frac{P(y)}{\cos\left(\frac{\pi}{2}\sqrt{2y+\mu_1^2}\right)}$$
if $\gamma_2$ is even. Here the polynomial $P$ is given by
$$P(y):=\prod_{k=0}^{\left\lfloor\frac{\gamma_2}{2}-1\right\rfloor}\Big(y-\y(s_2-k)\Big).$$
\end{theorem}
\begin{proof}
According to Lemma~\ref{lemma:phi1Dr1}, if $\gamma_2\in\mathbb N$ then $(\varphi_1/D_\dagger)^2$ is a Type II invariant. Applying Theorem~\ref{thm:prolongementr1}, $\varphi_1$ does not have any pole in the strip $\overline{\mathfrak{B}}_2$. The poles of $\varphi_1/D_\dagger$ are therefore the roots of $D_\dagger$, \textit{i.e.}, $s_2-k$ for $k$ between $0$ and $\gamma_2-1$. One can easily check that $\varphi_1/D_\dagger$ has a unique (simple) pole in $\overline{\mathfrak{B}}_2$ situated at $s_\pm$ ($s_-$ if $\gamma_2$ is even, and $s_+$ otherwise). Hence, by the (Type II) Invariant Lemma (see Proposition~\ref{lemma:inv}) there exists $c$ and $c_{\pm,1}\ne 0$ such that  
$$\left(\frac{\varphi_1}{D_\dagger}\right)^2=c+\frac{c_{\pm,1}}{\w-\w(s_\pm)}$$
or equivalently
$$\left(\frac{\varphi_1}{D_\dagger}\right)^2\big(\w-\w(s_\pm)\big)=c\w+\big(c_{\pm,1}-c\w(s_\pm)\big).$$
Noticing that $\w(s_\pm)=\mp 1$, one can use the double-angle formulas~\eqref{rem2}, to show that the left hand side of the above equation is the square of a meromorphic function, hence (by Lemma~\ref{lemma:sqrt}), either 
\begin{equation}\label{eq:alternative}
c=0\text{ or }c^2=\big(c_{\pm,1}\pm c\big)^2.
\end{equation}
Given that $c_{\pm,1}\ne 0$, the second alternative in~\eqref{eq:alternative} is equivalent to $c_{\pm,1}=\mp 2c$. Suppose that this second alternative holds. Then if $\gamma_2$ is even, it would lead to 
$$\varphi_1(s)\propto \frac{1}{\tan\Big(\pi(s-s_-)\Big)}\prod_{k=0}^{\gamma_2-1}\Big(s-(s_2-k)\Big),$$
and if $\gamma_2$ was odd,
$$\varphi_1(s)\propto \tan\Big(\pi(s-s_-)\Big)\prod_{k=0}^{\gamma_2-1}\Big(s-(s_2-k)\Big),$$
(see, for example, the proof of Theorem~\ref{thm:general} for similar trigonometric manipulations). In both cases, $\varphi_1$ would vanish at $s=s_\mp \in\Delta\cap\mathbb R$. But that would contradict Lemma~\ref{lemma:stripe}, hence $c=0$ and the rest of the proof follows similarly as in Theorem~\ref{thm:notN}.
\end{proof}

\begin{theorem}[Laplace transform, $r_1=-1$ and $\gamma_2\notin \mathbb N$]\label{thm:DTr1} Suppose $r_1=-1$ and $\gamma_2\notin\mathbb N$. Then
$$\varphi_1(s)\propto \frac{D_\dagger(s)}{\sin\left(\pi(s-s_2)\right)}=\frac{\Gamma(s+s_2+\mu_2)}{\Gamma(s-s_2)\sin\left(\pi(s-s_2)\right)}.$$
\end{theorem}
\begin{proof} According to Lemma~\ref{lemma:phi1Dr1}, $(\varphi_1/D_\dagger)^2$ is a Type I invariant, and its only pole in $\overline{\mathfrak{B}}_1\cup \{\mu_1\}$ is at $s_2-p$ for some $p\in\mathbb N$ and has multiplicity 2.
Assume for the moment that $s_2-\mu_1\notin\mathbb N$ (so that $s_2-p\ne \mu_1$). By the Type I invariant Lemma (see Proposition~\ref{lemma:periodic}), there exist three constants $c$, $c_{1,1}$ and $c_{1,2}$ (with $c_{1,2}\ne 0$) such that
$$\left(\frac{\varphi_1}{D_\dagger}\right)^2=c+\frac{c_{1,1}}{\wb-\wb(s_2-p)}+\frac{c_{1,2}}{\Big(\wb-\wb(s_2-p)\Big)^2}.$$
Note that $\wb(s_2-p)=\wb(s_2)$. Still from Lemma~\ref{lemma:phi1Dr1}, the left-hand side must tends to $0$ as $|s|\to +\infty$. In particular, for $s=\pm it$ with $t\to+\infty$, we obtain two linear equations involving $c_{1,1}$ and $c_{1,2}$:
$$0=c+\frac{c_{1,1}}{i-\wb(s_2)}+\frac{c_{1,2}}{\Big(i-\wb(s_2)\Big)^2}=c-\frac{c_{1,1}}{i+\wb(s_2)}+\frac{c_{1,2}}{\Big(i+\wb(s_2)\Big)^2}.$$
This system of linear equations is similar to the one of~\eqref{eq:linearsystem}. Solving it and performing the same trigonometric simplifications, one obtains
\begin{equation}\label{eq:r1general}
\varphi_1(s)^2\propto \frac{D_\dagger(s)^2}{\sin\Big(\pi(s-s_2)\Big)^2}.
\end{equation}
Let us now suppose that $s_2-\mu_1\in\mathbb N$ (so that $\varphi_1/D_\dagger$ has a pole at $\mu_1$) and derive a compatible conclusion. Applying Proposition~\ref{lemma:periodic} in this context gives
$$\left(\frac{\varphi_1}{D_\dagger}\right)^2=c+c_{\mu,1}\wb+c_{\mu,2}\wb^2,$$
for well choosen constants $c$, $c_{\mu,1}$ and $c_{\mu,2}\ne 0$. Arguing as before, we obtain $c+ic_{\mu,1}-c_{\mu,2}=0$ and $c-ic_{\mu,1}-c_{\mu,2}=0$, which implies $c_{\mu,1}=0$ and $c=c_{\mu,2}\ne 0$. Finally,
\begin{equation}\label{eq:r1special}
\left(\frac{\varphi_1(s)}{D_\dagger(s)}\right)^2\propto 1+\wb(s)^2=\frac{1}{\cos\Big(\pi(s+\mu_2-\frac12)\Big)^2}=\frac{1}{\sin\Big(\pi(s+\mu_2)\Big)^2}.
\end{equation}
Finally, substituting $s_2$ with $\mu_1+n$ for some $n\in\mathbb N$ in~\eqref{eq:r1general} recovers~\eqref{eq:r1special}, thereby establishing the unified formula.
\end{proof}
For a discussion on well-definedness and holomorphicity of $\sin(\sqrt{z})/\sqrt{z}$ and $\cos(\sqrt{z})$, see Remark~\ref{rem:sinc}.
\begin{corollary}[Algebraic and differential nature, $r_1=-1$]\label{cor:classificationr1} Suppose $r_1=-1$. Then $\phi_1$ is differentially algebraic if and only if $\gamma_2\in\mathbb N$. Moreover, $\phi_1$ is never D-finite.
\end{corollary}
\begin{proof}
Assume that $\gamma_2\in\mathbb N$. We claim that $\phi_1$ is D-algebraic but not D-finite. By Theorem~\ref{thm:DAr1}, the function $\phi_1$ can be expressed as the product of a polynomial $P$ with one of the two functions $$g_1(y):=\frac{1}{\cos(\sqrt{ay+b})},\qquad g_2(y):=\frac{\sqrt{ay+b}}{\sin(\sqrt{ay+b})}$$ for suitably chosen constants $a\ne 0$ and $b$. Lemma~\ref{lem:wDADF} ensures that both $g_1$ and $g_2$ are D-algebraic but not D-finite. Since the class of D-algebraic functions is closed under multiplication, it follows that $\phi_1$ is itself D-algebraic. Suppose, for contradiction, that $\phi_1$ were D-finite. As a rational function, $1/P$ is trivially D-finite, forcing $\phi_1/P$ to be D-finite. Then, either $g_1$ or $g_2$ is D-finite, contradicting Lemma~\ref{lem:wDADF}. Hence $\phi_1$ is D-algebraic but not D-finite, as claimed.\\
\noindent Conversely, if $\gamma_2\notin \mathbb N$, then Proposition~\ref{prop:rationalr1} asserts that $G$ admits no rational decoupling. By contraposition of Proposition~\ref{cor3.4}, this rules out the possibility that $\varphi_1$ is D-algebraic. Finally, obtaining $\phi_1(y)$ from $\varphi_1(s)$ amounts to performing the algebraic substitution $y = \y(s)$. Since the class of D-transcendental functions is stable under nonconstant algebraic substitutions, the conclusion follows.
\end{proof}

Finally, we conclude this subsection by inverting the Laplace transform of the lateral measure, using the argument from section~\ref{sec:density}.

\begin{corollary}[Density when D-algebraic, $r_1=-1$] Suppose $r_1=-1$ and $\gamma_2\in\mathbb N$. Then 
$$\nu_1(v)\propto \sum_{n\in\mathbb Z}(-1)^n\left[\prod_{k=0}^{\frac{\gamma_2-1}{2}}\left(\frac{\mu_1^2-4n^2}{2}+\y(s_2-k)\right)\right]\exp\left(-v\frac{4n^2-\mu_1^2}{2}\right)$$
if $\gamma_2$ is odd, and
$$\nu_1(v)\propto \sum_{n\in\mathbb Z}(4n+1)\left[\prod_{k=0}^{\frac{\gamma_2}{2}-1}\left(\frac{\mu_1^2-(4n+1)^2}{2}+\y(s_2-k)\right)\right]\exp\left(-v\frac{(4n+1)^2-\mu_1^2}{2}\right)$$
if $\gamma_2$ is even.
\end{corollary}

\subsection{Second case: $r_2=-1$}

We now want to determine $\phi_1(y)$ when $r_2=-1$.
It would be possible to adapt all the proofs of the preceding section to this case to obtain the desired results. We shall instead derive them by exploiting the symmetries of the model {to first obtain $\phi_2(x)$} using the results of the case $r_1=-1$ discussed in the previous section and then using the functional equation linking $\varphi_1$ and $\varphi_2$ thanks to the equation~\eqref{eq:eqfunc} which states that $k_1(s)\varphi_1(s)+k_2(s)\varphi_2(s)=0$. To obtain $\varphi_2(s)$ by symmetry, we exchange the role of several quantities accordingly: $r_1$ with $r_2$ (as desired), $\x$ with $\y$, $s$ with $-s$ (when working through the uniformization), $\mu_1$ with $\mu_2$, $s_1$ with $-s_2$, $\eta$ with $\zeta$, and $\gamma_1$ with $\gamma_2$.

In this context, only $s_1$ and $\gamma_1$ can be defined, whereas $s_2$, $\gamma_2$, and $\gamma$ cannot {(in Equation~\eqref{eq:k2}, the denominator of $s_2$ becomes $0$ when $r_2=-1$)}. Lemma~\ref{lemma:gamma2pos} immediately translates into $\gamma_1>0$ whenever $r_2=-1$
\begin{theorem}[Laplace transform, $r_2=-1$ and $\gamma_1\in\mathbb N$] Suppose $r_2=-1$ and $\gamma_1\in\mathbb N$. Then there exists a polynomial $P\in \mathbb R[X]$ such that the Laplace transform $\phi_1$ satisfies 
$$\phi_1(y)\propto \frac{P(y)\sqrt{2y+\mu_1^2}}{\sin\Big(\frac{\pi}{2}\sqrt{2y+\mu_1^2}\Big)}$$
if $\gamma_1$ is even, and
$$\phi_1(y)\propto \frac{P(y)}{\cos\Big(\frac{\pi}{2}\sqrt{2y+\mu_1^2}\Big)}$$
if $\gamma_1$ is odd. Here, the polynomial $P$ is given by
\begin{equation}\label{eq:defPr1}
P(y):=\prod_{k=1}^{\left\lfloor\frac{\gamma_1-1}{2}\right\rfloor}\Big(y-\y(s_1+k)\Big).
\end{equation}
\end{theorem}
\begin{proof} Assume for the moment that $\gamma_1$ is an odd integer. By applying the substitutions described above to Theorem~\ref{thm:DAr1}, we obtain
\begin{equation}\label{eq:phi2x}
\phi_2(x)\propto \frac{H(x)\sqrt{2x+\mu_2^2}}{\sin\Big(\frac{\pi}{2}\sqrt{2x+\mu_2^2}\Big)}\text{ where }H(x):=\prod_{k=0}^{\frac{\gamma_1-1}{2}-1}\Big(x-\x(s_1+k)\Big).
\end{equation}
This expression can be evaluated at $x=\x(s)$, and then substituted into the functional equation~\eqref{eq:eqfunc},
which relates $\varphi_1$ and $\varphi_2$:
\begin{equation}\label{eq:eqfuncr2}
\varphi_1(s)=\frac{-k_2(s)\varphi_2(s)}{k_1(s)}=\frac{\varphi_2(s)}{(1+r_1)(s-s_1)}\propto\frac{H(\x(s))\sqrt{2\x(s)+\mu_2^2}}{\sin\Big(\frac{\pi}{2}\sqrt{2\x(s)+\mu_2^2}\Big)(s-s_1)}.
\end{equation}
To express $\varphi_1$ in terms of $\y(s)$ and thus recover $\phi_1(y)$, note that
$$\sqrt{2\x(s)+\mu_2^2}=2s+\mu_2=2s-\mu_1+1=\sqrt{2\y(s)+\mu_1^2}+1,$$
and hence
\begin{equation}\label{eq:sinsqrt}
\sin\left(\frac{\pi}{2}\sqrt{2\x(s)+\mu_2^2}\right)=\sin\left(\frac{\pi}{2}\sqrt{2\y(s)+\mu_1^2}+\frac{\pi}{2}\right)=\cos\left(\sqrt{2\y(s)+\mu_1^2}\right).
\end{equation}
Also,
\begin{equation}\label{eq:H(x(s))}
H(\x(s))\sqrt{2\x(s)+\mu_2^2}=\prod_{k=0}^{\gamma_1-1}\Big(s-(s_1+k)\Big)\propto (s-s_1)\prod_{k=1}^{\frac{\gamma_1-1}{2}}\Big(\y(s)-\y(s_1+k)\Big).
\end{equation}
and the second product coincides with $P(\y(s))$ defined in~\eqref{eq:defPr1}. Substituting~\eqref{eq:sinsqrt} and~\eqref{eq:H(x(s))} into~\eqref{eq:eqfuncr2} yields the claimed formula in the odd case. The proof for the even case is analogous and omitted.
\end{proof}
\begin{theorem}[Laplace transform, $r_2=-1$ and $\gamma_1\notin\mathbb N$] Suppose $r_2=-1$ and $\gamma_1\notin\mathbb N$. Then 
$$\varphi_1(s)\propto \frac{\Gamma(s-s_1)}{\Gamma(s+s_1+\mu_2)\sin(\pi(s+s_1+\mu_2))}.$$
\end{theorem}
\begin{proof} Applying the substitutions described above in Theorem~\ref{thm:DTr1}, we obtain
$$\varphi_2(s)\propto \frac{\Gamma(-s-s_1+\mu_1)}{\Gamma(-s+s_1)\sin(\pi(s-s_1))}.$$
From the functional equation~\eqref{eq:eqfunc} (which is unchanged by the substitutions), we deduce that
$$\varphi_1(s)=\frac{-k_2(s)\varphi_2(s)}{k_1(s)}\propto\frac{\varphi_2(s)}{s-s_1}
=\frac{\Gamma(-s-s_1+\mu_1)}{\Gamma(-s+s_1)\sin(\pi(s-s_1))(s-s_1)}.$$
The desired result then follows directly from the functional equation $\Gamma(z+1)=z\Gamma(z)$ together with Euler’s reflection formula $\Gamma(1-z)\Gamma(z)=\pi/\sin(\pi z)$.
\end{proof}
\begin{corollary}[Algebraic and differential nature, $r_2=-1$]\label{cor:classificationr2} Suppose $r_2=-1$. Then $\phi_1$ is differentially algebraic if and only if $\gamma_1\in\mathbb N$. Moreover, $\phi_1$ is never D-finite.
\end{corollary}
In Table~\ref{tab:comparaison} (Appendix~\ref{sec:C}), the classifications given in Corollaries~\ref{cor:classificationr1} and~\ref{cor:classificationr2} are summarized and compared with the one presented in Section~\ref{sec:hierarchy} (in the case where neither $r_1=-1$ nor $r_2=-1$), as well as with the classification established by Bousquet-Mélou et al.~\cite{BoMe-El-Fr-Ha-Ra} in the nondegenerate case.\\

To conclude this appendix, we invert the Laplace transform $\phi_1$ when $\gamma_1\in\mathbb N$ (\textit{i.e.}, when $\phi_1$ is differentially algebraic).

\begin{corollary}[Density when D-algebraic, $r_2=-1$] Suppose $r_2=-1$ and $\gamma_1\in\mathbb N$. Then 
$$\nu_1(v)\propto \sum_{n\in\mathbb Z}(-1)^n\left[\prod_{k=1}^{\frac{\gamma_1}{2}}\left(\frac{\mu_1^2-4n^2}{2}+\y(s_1+k)\right)\right]\exp\left(-v\frac{4n^2-\mu_1^2}{2}\right)$$
if $\gamma_1$ is even, and
$$\nu_1(v)\propto \sum_{n\in\mathbb Z}(4n+1)\left[\prod_{k=1}^{\frac{\gamma_1-1}{2}}\left(\frac{\mu_1^2-(4n+1)^2}{2}+\y(s_1+k)\right)\right]\exp\left(-v\frac{(4n+1)^2-\mu_1^2}{2}\right)$$
if $\gamma_1$ is odd.
\end{corollary}

\section{Homogeneity relations}\label{sec:homogeneity}

In this appendix we explain how we can extend our results to a general degenerate covariance and drift with simple change of variables. All the quantities should depend on the list of parameters of the model
\begin{equation}
    \Omega := (\mu_1,\mu_2,\sigma_1,\sigma_2,r_1,r_2).
\end{equation}
satisfying \eqref{eq:H1},~\eqref{eq:H2} and~\eqref{eq:H3}. We note $\phi(x,y;\Omega)$, $\phi_1(y;\Omega)$ and $\phi_2(x;\Omega)$ the associate Laplace transforms.
We will also note
\begin{equation}
    \mathrm{q}:=\frac{\mu_1}{\sigma_1}+\frac{\mu_2}{\sigma_2}
\end{equation}
We introduce an alternative list of parameters 
\begin{equation}
    \widetilde{\Omega}=(\widetilde\mu_1,\widetilde\mu_2,\widetilde\sigma_1,\widetilde\sigma_2,\widetilde r_1, \widetilde r_2):=\left(\frac{\mu_1}{\sigma_1\mathrm{q}},\frac{\mu_2}{\sigma_2\mathrm{q}},1,1,\frac{r_1\sigma_1}{\sigma_2},\frac{r_2\sigma_2}{\sigma_1}\right)
\end{equation}
which therefore satisfy the analogue of~\eqref{eq:H4}, \textit{i.e.}, $\widetilde\mu_1+\widetilde\mu_2=1$ and $\widetilde\sigma_1=\widetilde\sigma_2=1$ and also the analogue of~\eqref{eq:H1},~\eqref{eq:H2} and~\eqref{eq:H3}.
The main theorems of this article determine the Laplace transforms $\phi(x,y;\widetilde\Omega)$, $\phi_1(y;\widetilde\Omega)$ and $\phi_2(x;\widetilde\Omega)$.

\begin{proposition}[Changing variables] The Laplace transforms satisfy the following homogeneity relations
\begin{equation}
\label{eq:chgtvarphi}
    \phi(x,y;\Omega)=\phi\left(\displaystyle\frac{\sigma_1 x}{\mathrm{q}},\displaystyle\frac{\sigma_2 y}{\mathrm{q}};\widetilde{\Omega}\right),
    \quad \phi_1(y;\Omega)=\mathrm{q}\sigma_1\phi_1\left(\displaystyle\frac{\sigma_2 y}{\mathrm{q}};\widetilde{\Omega}\right),
    \quad\phi_2(x;\Omega)=\mathrm{q}\sigma_2\phi_2\left(\displaystyle\frac{\sigma_1 x}{\mathrm{q}};\widetilde{\Omega}\right).
\end{equation}
We deduce that, up to a (Lebesgue) null set,
\begin{equation}
\label{eq:chgtvarp} 
\pi(u,v ; \Omega)=\displaystyle\frac{\mathrm{q}^2}{\sigma_1\sigma_2}\pi\left(\displaystyle\frac{\mathrm{q} u}{\sigma_1},\displaystyle\frac{\mathrm{q} v}{\sigma_2};\widetilde{\Omega}\right),
    \quad \nu_1( v;\Omega)=\displaystyle\frac{\mathrm{q}^2\sigma_1}{\sigma_2}\nu_1\left(\displaystyle\frac{\mathrm{q} v}{\sigma_2};\widetilde{\Omega}\right),
    \quad \nu_2(u;\Omega)=\displaystyle\frac{\mathrm{q}^2\sigma_2}{\sigma_1}\nu_2\left(\displaystyle\frac{\mathrm{q} u}{\sigma_1};\widetilde{\Omega}\right).
\end{equation}
\end{proposition}
\begin{proof} 
Let us start with the functional equation~\eqref{eq:FE}:
\begin{equation}
\label{eq:FEOm}
-K\left(x,y;\Omega\right)\phi(x,y;\Omega)
=k_1\left(x,y;\Omega\right)\phi_1(y;\Omega)+k_2\left(x,y;\Omega\right)\phi_2(x;\Omega)
\end{equation}
A direct computation show that the polynomial coefficients of the functional equation obey the following change of parameters rules 
    \begin{equation*}
        K\left(\frac{\mathrm{q}x}{\sigma_1},\frac{\mathrm{q}y}{\sigma_2};\Omega\right)=\mathrm{q}^2K\left(x,y;\widetilde{\Omega}\right), 
    \end{equation*}
    \begin{equation}
  k_1\left(\frac{\mathrm{q}x}{\sigma_1},\frac{\mathrm{q}y}{\sigma_2};\Omega\right)=\frac{\mathrm{q}}{\sigma_1}k_1\left(x,y;\widetilde{\Omega}\right),    \quad  k_2\left(\frac{\mathrm{q}x}{\sigma_1},\frac{\mathrm{q}y}{\sigma_2};\Omega\right)=\frac{\mathrm{q}}{\sigma_2}k_2\left(x,y;\widetilde{\Omega}\right).
    \end{equation}
Then, evaluating the functional equation~\eqref{eq:FEOm} at $(\frac{\mathrm{q}x}{\sigma_1},\frac{\mathrm{q}y}{\sigma_2})$ 
and dividing by $\mathrm{q}^2$ both sides of the equation we obtain 
$$
K\left(x,y;\widetilde{\Omega}\right)
\underbrace{\phi\left(\frac{\mathrm{q}x}{\sigma_1},\frac{\mathrm{q}y}{\sigma_2};\Omega\right)}_{\displaystyle=\phi(x,y;\widetilde{\Omega})}
=k_1\left(x,y;\widetilde{\Omega}\right)\underbrace{\frac{1}{\mathrm{q}\sigma_1}\phi_1\left(\frac{\mathrm{q}y}{\sigma_2};\Omega\right)}_{\displaystyle=\phi_1(y;\widetilde{\Omega})}+k_2\left(x,y;\widetilde{\Omega}\right)\underbrace{\frac{1}{\mathrm{q}\sigma_2}\phi_2\left(\frac{\mathrm{q}x}{\sigma_1};\Omega\right)}_{\displaystyle=\phi_2(x;\widetilde{\Omega})}.
$$
The equalities under the braces follows from the following. Since $\widetilde\Omega$ satisfy the analogue of~\eqref{eq:H1},~\eqref{eq:H2},~\eqref{eq:H3} and~\eqref{eq:H4}, we showed in this article that the following functional equation
$$
K\left(x,y;\widetilde{\Omega}\right)
\phi(x,y;\widetilde{\Omega})
=k_1\left(x,y;\widetilde{\Omega}\right)\phi_1(y;\widetilde{\Omega})+k_2\left(x,y;\widetilde{\Omega}\right)\phi_2(x;\widetilde{\Omega})
$$
has a unique solution composed by the triplet $(\phi(x,y;\widetilde{\Omega}),\phi_1(y;\widetilde{\Omega}),\phi_2(x;\widetilde{\Omega}))$ under the assumption that $\phi$ is the Laplace transform of a probability measure (positive measure of mass 1) and $\phi_1$ and $\phi_2$ are Laplace transforms of positive measures. This directly implies Equation~\eqref{eq:chgtvarphi} and inverting the Laplace transforms and using time scaling we obtain Equation~\eqref{eq:chgtvarp}.
\end{proof}

\section{Links with the nondegenerate case}\label{sec:C}

In the nondegenerate case, Bousquet-Mélou et al.~\cite{BoMe-El-Fr-Ha-Ra} proved that the differential nature of the invariant measure only depends on three quantities $\alpha$, $\alpha_1$ and $\alpha_2$, defined as follows: consider a (nondegenerate) reflected brownian motion in $\mathbb R_+^2$, with nonsingular covariance matrix $\Sigma$. This model is equivalent (through a linear automorphism) to a reflected brownian motion a wedge with opening angle $\beta\in (0,\pi)$, identity covariance matrix, reflection angles $\delta$ and $\varepsilon$, and drift angle $\theta$ (see Figure~\ref{fig:wedge}). Then 
\begin{equation}\label{eq:defalphas}
\alpha:=\frac{\delta+\varepsilon-\pi}{\beta},\quad \alpha_1:=\frac{2\varepsilon+\theta-\beta-\pi}{\beta},\quad \alpha_2:=\frac{2\delta-\theta-\pi}{\beta}.
\end{equation}

\begin{figure}
\begin{center}
\includegraphics[width=12cm]{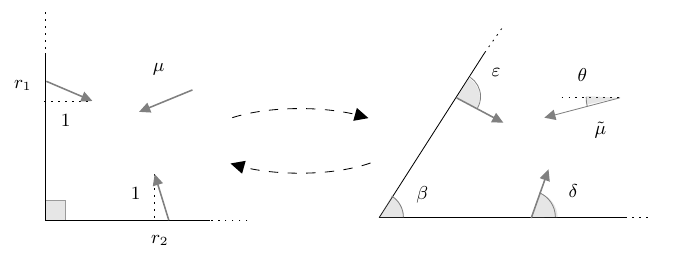}
\end{center}
\caption{from the quadrant to the wedge: definition for the angles $\beta$, $\delta$, $\varepsilon$ and $\theta$.}
\label{fig:wedge}
\end{figure}

\noindent These key parameters satisfy the relation $\alpha_1+\alpha_2=2\alpha-1$ which is reminiscent of~\eqref{eq4}. The opening angle satisfies
$$\beta=\arccos\left(-\frac{\sigma_{12}}{\sqrt{\sigma_1\sigma_2}}\right),\text{ where }\Sigma=\left(\begin{array}{cc}
\sigma_1 & \sigma_{12}\\
\sigma_{12} & \sigma_2
\end{array}\right),$$
so that $\Sigma\to \mathcal{A}$ (the degenerate covariance matrix defined in the introduction) if and only if $\sigma_{12}\to -\sqrt{\sigma_1\sigma_2}$ if and only if $\beta\to 0$: the degenerate model can be interpreted as the limit as $\beta \to 0$ of the nondegenerate model. One can show the following relations between the quadrant and the wedge parameters:\\[-0.2cm]
$$
\tan(\delta)=\frac{\sin(\beta)}{r_2+\cos(\beta)},\quad \tan(\varepsilon)=\frac{\sin(\beta)}{r_1+\cos(\beta)},\quad \tan(\theta)=\frac{\mu_2\sin(\beta)}{\mu_1+\mu_2\cos(\beta)},
$$
hence,
$$\lim_{\beta\to 0}\left(\frac{\delta}{\beta}\right)=\frac{1}{1+r_2},\quad \lim_{\beta\to 0}\left(\frac{\varepsilon}{\beta}\right)=\frac{1}{1+r_1},\quad \lim_{\beta\to 0}\left(\frac{\theta}{\beta}\right)=\frac{\mu_2}{\mu_1+\mu_2}=\mu_2.$$
As a direct consequence,
$$\lim_{\beta\to 0}\left(\alpha+\frac{\pi}{\beta}\right)= \gamma+1,\quad \lim_{\beta\to 0}\left(\alpha_1+\frac{\pi}{\beta}\right)=\gamma_1,\quad \lim_{\beta\to 0}\left(\alpha_2+\frac{\pi}{\beta}\right)= \gamma_2.$$
One can, \textit{in some sense} take the limit as $\beta$ goes to $0$ in the nondegenerate classification given in~\cite{BoMe-El-Fr-Ha-Ra} to recover the classification obtained in the present article. Here are two key observations:
\begin{itemize}
\item In the nondegenerate case, Bousquet-Mélou et al. needed to distinguish whether the opening angle was a rational multiple of $\pi$ or not. In the limit, this distinction is no longer necessary.
\item The nature of the Laplace transforms appears to be \textit{singularly perturbed}: although \textit{strictly algebraic} (algebraic but not rational) or \textit{strictly D-finite}  (D-finite but not algebraic) Laplace transforms appear in the nondegenerate case, the degenerate case does not admit any. It is worth noting that the classes of rational, algebraic, D-finite, and D-algebraic functions are not closed under uniform convergence on compact subsets (for example, Taylor approximations can be used to construct counterexamples).
\end{itemize}
Table~\ref{tab:comparaison} compares the classification in the nondegenerate case from~\cite{BoMe-El-Fr-Ha-Ra} with the one established in the present article. The second line corresponds to the Corollaries~\ref{cor:classificationr1} and~\ref{cor:classificationr2}, and the indices $i$ and $j$ are such that $\{i,j\}=\{1,2\}$.
\vfill
\begin{table}[!htbp]
\centering
\begin{tabular}{
!{\vrule width 1pt}c!{\vrule width 1pt}c|c|c|c!{\vrule width 1pt}}
\noalign{\hrule height 1pt}
\begin{tabular}{@{}c@{}}\textbf{nature of} $\boldsymbol{\phi_1}$\\  $\boldsymbol{\phi_2}$ \textbf{and} $\boldsymbol{\phi}$ \end{tabular} & rational & algebraic & D-finite & D-algebraic  \\ 
\noalign{\hrule height 1pt} 
\begin{tabular}{@{}c@{}} \textbf{degenerate} \\ and $r_1,r_2\ne -1$  \end{tabular} & \multicolumn{3}{c|}{$\gamma\in-\mathbb{N}$} & \begin{tabular}{@{}c@{}}$\gamma\in\mathbb{Z}$ or \\  $\{ \gamma_1,\gamma_2\}\subset \mathbb{Z}$\end{tabular}     \\  \hline
\begin{tabular}{@{}c@{}} \textbf{degenerate} \\ and $r_i=-1$  \end{tabular} & \multicolumn{3}{c|}{never} & $\gamma_j\in\mathbb N$  \\  \noalign{\hrule height 1pt}
\begin{tabular}{@{}c@{}} \textbf{nondegenerate} \\ and $\beta\notin\pi\mathbb Q$  \end{tabular} & $\alpha\in-\mathbb N$ & \begin{tabular}{@{}c@{}} $\alpha\in-\mathbb N$ or \\ $\{\alpha_1,\alpha_2\}\subset \mathbb Z$ \end{tabular} & \begin{tabular}{@{}c@{}} $\alpha\in-\mathbb N+\frac{\pi}{\beta}\mathbb Z$ or \\ $\{\alpha_1,\alpha_2\}\subset \mathbb Z\cup\left(-\mathbb N+\frac{\pi}{\beta}\mathbb Z\right)$ \end{tabular} & \begin{tabular}{@{}c@{}}$\alpha\in\mathbb{Z}+\frac{\pi}{\beta}\mathbb Z$ or \\  $\{ \gamma_1,\gamma_2\}\subset \mathbb{Z}+\frac{\pi}{\beta}\mathbb Z$\end{tabular}     \\ \hline
\begin{tabular}{@{}c@{}} \textbf{nondegenerate} \\ and $\beta\in\pi\mathbb Q$  \end{tabular} & $\alpha\in-\mathbb N$ & \multicolumn{2}{c|}{\begin{tabular}{@{}c@{}}$\alpha\in\mathbb{Z}+\frac{\pi}{\beta}\mathbb Z$ or \\  $\{ \gamma_1,\gamma_2\}\subset \mathbb{Z}+\frac{\pi}{\beta}\mathbb Z$\end{tabular}} & always     \\ 
\noalign{\hrule height 1pt}
\end{tabular} 
\caption{Comparison with the classification in the nondegenerate case established in~\cite{BoMe-El-Fr-Ha-Ra}.}
\label{tab:comparaison}
\end{table}
\newpage
\section{Index of notations}\label{sec:notations}
In this appendix, we list the notations most frequently used in this article, along with a short description and a direct reference to where each first appears.\\
{
\renewcommand{\arraystretch}{1.4}
\setlength{\extrarowheight}{3pt}
\begin{tabular}{m{0.09\linewidth}|p{0.32\linewidth}|m{0.40\linewidth}|m{0.09\linewidth}}
\textbf{Symbol} & \textbf{Description} & \textbf{Formula} & \textbf{Ref.} \\
\hline
$\mathcal{A}$ &Covariance matrix of the reflected Brownion motion (degenerate) &\raisebox{-0.6\height}{\renewcommand{\arraystretch}{1}
\setlength{\extrarowheight}{0pt}$\mathcal{A}:=\left(\begin{array}{cc}
\sigma_1 & -\sqrt{\sigma_1\sigma_2}\\
-\sqrt{\sigma_1\sigma_2} & \sigma_2
\end{array}\right)$} &Eq.\eqref{eq:defparam}\\ \hline
$\mathfrak{B}_1$ &Fundamental domain of the canonical Type I invariant $\wb$ &\raisebox{-0.4\height}{$\mathfrak{B}_1:=\{s\in\mathbb C : -\mu_2<\mathfrak{Re}(s)<\mu_1\}$} &Eq.\eqref{eq3}\\ \hline
$\mathfrak{B}_2$ &Fundamental domain of the canonical Type II invariant $\w$& \raisebox{-0.4\height}{$\mathfrak{B}_2:=\{s\in\mathbb C : -\mu_2<2\mathfrak{Re}(s)<\mu_1\}$}& Eq.\eqref{eq3}\\ \hline
$\delta$ &Intersection of $\Delta_{\mathrm{x}}$ with $\Delta_{\mathrm{y}}$& $\delta:=\Delta_{\mathrm{x}}\cap\Delta_{\mathrm{y}}$&Eq.\eqref{eq:Delta}\\ \hline
$\Delta$ &Uninon of $\Delta_{\mathrm{x}}$ with $\Delta_{\mathrm{y}}$& $\Delta:=\Delta_{\mathrm{x}}\cup\Delta_{\mathrm{y}}$ &Eq.\eqref{eq:Delta}\\ \hline
$\widetilde{\Delta}$&Uninon of $\Delta$, $\eta\Delta$ and $\zeta \Delta$& $\widetilde{\Delta}:=\Delta\cup \eta \Delta\cup \zeta \Delta$ & Eq.\eqref{eq:Delta}\\ \hline
$\Delta_{\mathrm{x}}$&Initial domain for $\varphi_2$& $\Delta_{\mathrm{x}}:=\{s\in\mathbb C: \mathfrak{Re}(\x(s))<0\}$ & Eq.\eqref{eq:domains}\\ \hline
$\Delta_{\mathrm{y}}$&Initial domain for $\varphi_1$&$\Delta_{\mathrm{y}}:=\{s\in\mathbb C: \mathfrak{Re}(\y(s))<0\}$& Eq.\eqref{eq:domains}\\ \hline
$D$&Decoupling function for $G$&\raisebox{-0\height}{$\displaystyle D(s):=\frac{\Gamma(s-s_1)\Gamma(s+s_2+\mu_2)}{\Gamma(s-s_2)\Gamma(s+s_1+\mu_2)}$} & Eq.\eqref{eq:decouplageGamma}\\ \hline
$d_1$, $d_2$ &Boundary lines of the strip $\mathfrak{B}_1$& $\begin{array}{l}
d_1:=\{s\in\mathbb C : \mathfrak{Re}(s)=\mu_1\}\\[-0.1cm]
d_2:=\{s\in\mathbb C : \mathfrak{Re}(s)=-\mu_2\}
\end{array}$ & Eq.\eqref{eq2}\\ \hline
$d_-$, $d_+$&Boundary lines of the strip $\mathfrak{B}_2$& $d_\pm := \{s\in\mathbb C : \mathfrak{Re}(s)=s_\pm\}$ & Eq.\eqref{eq2}\\ \hline
$\phi(x,y)$&Laplace transform of the invariant measure $\boldsymbol{\pi}$& \raisebox{-0.3\height}{$\phi(x,y):=\displaystyle\int\!\!\!\!\!\int_{\mathbb R_+^2} \pi(u,v)e^{xu+yv}\mathrm{d}u\mathrm{d}v$} & Eq.\eqref{eq:defphi}      \\ \hline
$\phi_1(y)$ & Laplace transform of the lateral measure $\boldsymbol{\nu}_1$& \raisebox{-0.2\height}{$\phi_1(y):=\displaystyle\int_0^{+\infty} \nu_1(v)e^{yv}\mathrm{d}v$} & Eq.\eqref{eq:defphi1}\\ \hline
$\phi_2(x)$ & Laplace transform of the lateral measure $\boldsymbol{\nu}_2$& \raisebox{-0.2\height}{$\phi_2(x):=\displaystyle\int_0^{+\infty} \nu_2(u)e^{xu}\mathrm{d}u$} & Eq.\eqref{eq:defphi2}\\ \hline
$\varphi_k(s)$ & Laplace transform $\phi_k$ through the uniformization $(\x,\y)$& \raisebox{-0.2\height}{$\varphi_1(s)=\phi_1(\y(s))$, $\varphi_2(s)=\phi_2(\x(s))$} & Eq.\eqref{eq:defvarphi}\\ \hline
$f^\pm$&Functions defining the boundary of the domain $\Delta_{\mathrm{x}}$& \raisebox{-0.2\height}{$f^{\pm}(b):=\displaystyle -\frac{\mu_2}{2}\pm \sqrt{\left(\frac{\mu_2}{2}\right)^2+b^2}$} & Eq.\eqref{eq:fgpm}\\ \hline
$\gamma$&Key parameter in the degenerate case& \raisebox{-0.2\height}{$\gamma:=s_2-s_1=\displaystyle \frac{1-r_1r_2}{(1+r_1)(1+r_2)}$} &Eq. \eqref{eq:gammai}\\ \hline
$\gamma_1$, $\gamma_2$&Additional parameters in the degenerate case& \raisebox{-0.2\height}{$\gamma_1:=\mu_1-2s_1$, $\gamma_2:=\mu_2+2s_2$}& Eq.\eqref{eq:gammai}\\ \hline
$g^\pm$&Functions defining the boundary of the domain $\Delta_{\mathrm{y}}$& \raisebox{-0.2\height}{$g^{\pm}(b):=\displaystyle\frac{\mu_1}{2}\pm \sqrt{\left(\frac{\mu_1}{2}\right)^2+b^2}$}&Eq.\eqref{eq:fgpm}\\ \hline
$\mathbf{G}$&Gap process& $\mathbf{G}=(G_1,G_2)$ &Eq.\eqref{def:G}\\ \hline
$G$&Rational coefficient in the difference equation& \raisebox{-0.2\height}{$G(s):=\displaystyle\frac{(s-s_1)(s+s_2+\mu_2)}{(s-s_2)(s+s_1+\mu_2)}$} & Eq.\eqref{eq:Gdef}
\end{tabular}
\noindent
\begin{tabular}{m{0.09\linewidth}|p{0.32\linewidth}|m{0.40\linewidth}|m{0.09\linewidth}}
\textbf{Symbol} & \textbf{Description} & \textbf{Formula} & \textbf{Ref.} \\
\hline
$\zeta$&Galois automorphism of $\mathcal{S}$ fixing the first coordinate& \raisebox{-0.3\height}{$\zeta(s):=-s+\mu_1$} &Eq.\eqref{eq:etazeta}\\ \hline
$\eta$&Galois automorphism of $\mathcal{S}$ fixing the second coordinate& \raisebox{-0.2\height}{$\eta(s):=-s-\mu_2$} & Eq.\eqref{eq:etazeta}\\ \hline
$K$&Polynomial coefficient of $\phi$ in the functional equation, kernel&\raisebox{-0.1\height}{$K(x,y):=( \sigma_1x-\sigma_2y)^2-2\mu_1 x-2\mu_2 y$} & Eq.\eqref{eq:noyau}\\ \hline
$k_i(x,y)$&Polynomial coefficients in the functional equation& $\!\!\!\begin{array}{l}
k_1(x,y):=x+r_1y \\[-0.1cm]
k_2(x,y):=y+r_2 x
\end{array}$ & Eq.\eqref{eq:k12}\\ \hline
$k_i(s)$&Short for $k_i(\x(s),\y(s))$& $k_i(s)=2(1+r_i)s(s-s_i)$ & Eq.\eqref{eq:defk12s}\\ \hline
$\mathbf{L}$&Local time of $\mathbf{G}$ at $0$& $\mathbf{L}=(L^{G_1},L^{G_2})$ &Eq.\eqref{def:G}\\ \hline
$-\bm\mu$&Drift of the RBM&$-\bm\mu=-(\mu_1,\mu_2)$&Eq.\eqref{eq:defparam}\\ \hline
$\boldsymbol{\nu}_1$, $\boldsymbol{\nu}_2$&Lateral measures& $\displaystyle\boldsymbol{\nu}_i(A):=\mathbb E_{\boldsymbol{\pi}}\left[\int_0^2\mathds{1}_A({G}_j(t))\mathrm{d}L^{{G}_i}(t)\right]$ & Eq.\eqref{eq:deflateral}\\ \hline
$\nu_i$ &Densities of $\boldsymbol{\nu}_i$ with respect to the Lebesgue measure& \raisebox{-0.2\height}{$\displaystyle\nu_i:=\frac{\mathrm{d}\boldsymbol{\nu}_i}{\mathrm{d}\lambda^{(1)}}$} & Eq.\eqref{eq:densitedef}\\ \hline
$\pi$ &Density of $\boldsymbol{\pi}$ with respect to the Lebesgue measure& \raisebox{-0.2\height}{$\displaystyle\pi:=\frac{\mathrm{d}\boldsymbol{\pi}}{\mathrm{d}\lambda^{(2)}}$} & Eq.\eqref{eq:densitedef}\\ \hline
$\theta_{\mathtt{a}}$& First Jacobi Theta-like function& $\theta_{\mathtt{a}}(q):=\displaystyle\sum_{n\in\mathbb Z} \left(n+\frac{\gamma_1}{2}\right)q^{\frac{(2n+\gamma_1)^2-\mu_1^2}{2}}$&Eq.\eqref{eq:thetaa}\\ \hline
$\theta_{\mathtt{b}}$&Second Jacobi Theta-like function& $\theta_{\mathtt{b}}(q):=\displaystyle\sum_{n\in\mathbb Z} (-1)^nn^2q^{\frac{n^2-\mu_1^2}{2}}$ & Eq.\eqref{eq:thetab}\\ \hline
$\mathcal{R}$&Reflection matrix& \raisebox{-0\height}{\renewcommand{\arraystretch}{1}
\setlength{\extrarowheight}{0pt}$\mathcal{R}:=\left(\begin{array}{cc}
1 & r_2 \\
r_1 & 1
\end{array}\right)$}   & Eq.\eqref{eq:defparam}\\ \hline
$\mathcal{S}$&Vanishing set of $K$&  $\mathcal{S}:=\{(x,y)\in\mathbb C^2 : K(x,y)=0\}$ &Eq.\eqref{eq:parabola}\\ \hline
$s_i$& Nonzero root of $k_i(s)$ under the condition $r_i\ne -1$& \raisebox{-0.3\height}{$s_1:=\displaystyle\frac{r_1\mu_1-\mu_2}{1+r_1}$, $s_2:=\displaystyle\frac{\mu_1-r_2\mu_2}{1+r_2}$} & Eq.\eqref{eq:k2}\\ \hline
$s_+$&Fixed point of $\eta$&  $s_+:=\mu_1/2$  & Eq.\eqref{eq:s+s-}\\ \hline
$s_-$& Fixed point of $\zeta$&  $s_-:=-\mu_2/2$ & Eq.\eqref{eq:s+s-}\\ \hline
$\wb$&Canonical Type I invariant& $\wb(s):=\tan\left(\pi(s+\mu_2-\frac{1}{2})\right)$  &Eq.\eqref{eq:w1}\\ \hline
$\w$&Canonical Type II invariant&  $\w(s):=\cos\left(2\pi(s-s_-)\right)$ &Eq.\eqref{eq:w}\\ \hline
$(\x,\y)$&Uniformization of the surface $\mathcal{S}$& $\left(\x(s),\y(s)\right):=\left(2s(s+\mu_2),2s(s-\mu_1)\right)$& Eq.\eqref{eq:unif}
\end{tabular}}

\newpage
\setcounter{section}{0}

\subsection*{Funding} 
This project has received funding from Agence Nationale de la Recherche, ANR JCJC programme under the Grant Agreement ANR-22-CE40-0002 (ANR RESYST).  The IMB receives support from the EIPHI Graduate School (contract ANR-17-EURE-0002)

\bibliographystyle{apalike}
\bibliography{biblio}

\end{document}